\documentclass[myspacing,11pt]{ectaart2}

\RequirePackage[OT1]{fontenc} \RequirePackage{amsthm}
\RequirePackage[cmex10]{amsmath}

\RequirePackage{amsfonts} \RequirePackage{amssymb}
\RequirePackage{xcolor}

\RequirePackage[normalem]{ulem}

\RequirePackage{graphicx}
\RequirePackage{verbatim}
\usepackage{cite} 
\usepackage{url}

\startlocaldefs \numberwithin{equation}{section}
\RequirePackage{a4wide}

\theoremstyle{plain}

\newtheorem{theorem}{Theorem}[section]
\newtheorem{corollary}{Corollary}[section]

\newtheorem{lemma}{Lemma}[section]

\endlocaldefs

\def\@bysame#1{\vrule height 1.5pt depth -1pt width 3em \hskip
0.5em\relax}

\newcommand{\N}{ \mathbb{N} }

\newcommand{\Z}{ \mathbb{Z} }

\newcommand{\R}{ \mathbb{R} }

\newcommand{\trunc}[1]{ {\lfloor #1 \rfloor} }
\newcommand{\wh}[1]{ \widehat{ #1 } }
\newcommand{\wt}[1]{ \widetilde{ #1 } }

\newcommand{\eins}{{ 1}}

\newcommand{\Var}{{\mbox{ $\mathbb V\textnormal{ar}$\,}}}
\newcommand{\Cov}{{\mbox{ $\mathbb C\textnormal{ov}$\,}}}

\endlocaldefs

\usepackage{subfigure}
\usepackage{mathtools} 
\usepackage{bm}
\usepackage{dsfont}
\usepackage{color}
\usepackage{amsthm}

\usepackage{float}                 

\begin{document}

\begin{frontmatter}
\title{Estimation of the Asymptotic Variance of Univariate and Multivariate Random Fields and Statistical Inference }
\runtitle{Estimation of the Asymptotic Variance for Random Fields}


\begin{aug}
\author{\fnms{Annabel} \snm{Prause} 
\ead[label=e2]{prause@stochastik.rwth-aachen.de}}
\author{\fnms{Ansgar} \snm{Steland}
\ead[label=e1]{steland@stochastik.rwth-aachen.de}}
\ead[label=u1,url]{www.stochastik.rwth-aachen.de}

\runauthor{A. Prause, A. Steland}

\address{Institute of Statistics\\ RWTH Aachen University \\ W\"ullnerstr. 3, D-52056 Aachen, Germany\\
\printead{e1}} \printaddresses \bigskip

\date{August 2017} 
\end{aug}

January, 2018

\begin{abstract}
Correlated random fields are a common way to model dependence structures in high-dimensional data, especially for data collected in imaging. One important parameter characterizing the degree of dependence is the asymptotic variance which adds up all autocovariances in the temporal and spatial domain. Especially, it arises in the standardization of test statistics based on partial sums of random fields and thus the construction of tests requires its estimation.
In this paper we propose consistent estimators for this parameter for strictly stationary  $\varphi$-mixing random fields with arbitrary dimension of the domain and taking values in a Euclidean space of arbitrary dimension, thus allowing for multivariate random fields. We establish consistency, provide central limit theorems and show that distributional approximations of related test statistics based on sample autocovariances of random fields can be obtained by the subsampling approach.

As in applications the spatial-temporal correlations are often quite local, such that a large number of autocovariances vanish or are negligible, we also investigate a thresholding approach where sample autocovariances of small magnitude are omitted. Extensive simulation studies show that the proposed estimators work well in practice and, when used to standardize image test statistics, can provide highly accurate image testing procedures. 
\\[0.2cm]
\textbf{MSC 2010 subject classifications:} 62H86, 62E20, 60G60, 62H12.
\\[0.2cm]
\textbf{Keywords and phrases:} Data science, high-dimensional data, imaging, long-run variance estimation, , thresholding estimator, subsampling, random field, spatial statistics, testing.
\end{abstract}
\end{frontmatter}

\section{Introduction}

Random fields are a natural approach to model high-dimensional data. They arise in a natural way when analyzing digitized image data as arising in medical imaging, e.g. MRI or CT images, or in industrial quality control, e.g. images of materials obtained from cameras capturing the visual or infrared spectrum or electroluminescence images of solar cells, in order to analyze the structure of the material of interest and to detect defects. Although the main results go beyond that scope and even allow to treat, for example, voxel-by-voxel multivariate brain data collected at a sample of individuals, let us first stick to the following setting, in order to motivate the approach, outline basic ideas and fix notations: Consider a sequence of $ n_1 $ images represented by $ n_2 \times n_3 $ dimensional matrices 
\begin{equation}
\label{SeqImages}
 ( Y_{(i_1,i_2,i_3)} )_{(i_2,i_3) \in (1:n_2) \times (1:n_3)}, \qquad  i_1 = 1, \dots, n_1, 
\end{equation}
of real-valued random variables, where for any $q \in \N $ we put $ {\bm a}:{\bm b} = \times_{i=1}^q \{ a_i, \dots, b_i \} $ for $ {\bm a} = (a_1, \dots, a_q), {\bm b} = (b_1, \dots, b_q) \in \N^q $ and $ {\bm 1} = (1, \dots, 1) \in \N^q $. The sequence (\ref{SeqImages}) of matrices can be seen as the corresponding subset $ \{ Y_{\bm i} \}_{{\bm i} \in {\bm 1}:{\bm n} } $  of a three-dimensional random field $ \{ Y_{\bm i} : \bm i \in \Z^3 \} $ of random variables defined on a common probability space.  A statistical test to check whether or not one observes a reference image
$ {\bm m}^{(0)} = ( m^{(0)}_{(i_2,i_3)} )_{(i_2,i_3) \in {\bm 1}:(n_2,n_3)} $ can be based on the sum
\[
  S_{\bm n} = \sum_{i_1=1}^{n_1} \sum_{i_2=1}^{n_2} \sum_{i_3 = 1}^{n_3 } 
  ( Y_{\bm i} - m^{(0)}_{(i_2,i_3)} ), \qquad \bm n = (n_1, n_2, n_3) \in \N^3.
\]
In order to test for the presence of a known reference series of images
$ m^{(0)}_{\bm i} $, $ {\bm i} \in {\bm 1}:{\bm n} $, one simply replaces $ m^{(0)}_{(i_2,i_3)} $ by $ m^{(0)}_{\bm i} $. If the null hypothesis $ H_0: \mathbb E(Y_{\bm i}) = m^{(0)}_{\bm i} $ $ (\forall {\bm i} \in {\bm 1}:{\bm n} )$ holds true, 
under fairly weak conditions on the error random field $ \xi_{\bm i} = Y_{\bm i} - \mathbb E( Y_{\bm i} ) $, $ {\bm i} \in \Z^q $, the partial sum scaled by the squared root of $ |\bm n| = \prod_i n_i $ converges weakly (in distribution) to a Gaussian law, 
\begin{equation}
\label{CLTRF}
  \frac{1}{ \sqrt{ | {\bm n} |}} S_{\bm n} \Rightarrow \sigma B({\bm 1}), 
\end{equation}
as $ \bm n \to \infty $, where from now on $ \bm n \to \infty $ is understood as $  \min_j n_j \to \infty $. Here
$ B(x,y,z) $, $ x,y,z \ge 0 $, is a standard Brownian motion in dimension $3$. For example, (\ref{CLTRF}) has been shown for weakly stationary  linear processes, see \cite{MarinucciPoghosyan2001} and \cite{Paulauskas2010}, strictly stationary $ \varphi $-mixing fields on which we shall focus in the present paper, see \cite{Deo}, or, using other notions of weak dependence, in 
\cite{BulinskiKaene}, \cite{BerkesMorrow}, \cite{WangWoodroofe}, \cite{MachkouriVonlyWu} and \cite{BulinskiShaskin2006}. 

The asymptotic variance $ \sigma^2 $ is given by $ \sigma^2 = \lim_{\bm n \to \infty} \sigma_{\bm n}^2 $, where
\begin{align*}
\sigma_{\bm n}^2 & = \text{Var}\left( \frac{1}{ | {\bm n} |} \sum_{ {\bm i} \in 1:{\bm n}} {\xi}_{\bm i} \right)
\end{align*} 
For a weakly stationary random field it holds
\begin{align*}
\sigma_{\bm n}^2
& = 
\sum_{-(\bm n - \bm 1) \le \bm \ell \le \bm n - \bm 1}
\prod_{j=1,2,3} \frac{ n_j - |\ell_j| }{ n_j }  \mathbb E( \xi_{\bm 0} \xi_{\bm \ell} ),
\end{align*}
and for an isotropic field where $ \mathbb E( \xi_{\bm 0} \xi_{\bm \ell} ) $ is a function of $ \| \bm \ell \| $
\begin{align*}
\sigma_{\bm n}^2
& = \mathbb E( \xi_{\bm 0}^2 ) + \sum_{\bm 0 \le^* \bm \ell \le \bm n - \bm 1}
\prod_{j=1,2,3} \frac{ n_j - \ell_j }{ n_j } 2^{\| \bm \ell \|_{0} } \mathbb E( \xi_{\bm 0} \xi_{\bm \ell} ).
\end{align*}
Here  $ \bm a \le^* \bm b $ for  vectors $ \bm a, \bm b $ means $ \bm a \le \bm b $ with $ \bm a \not= \bm b $, $ \| \bm \ell \| $ is the Euclidean vector noerm and $ \| \bm \ell \|_0  $ denotes the number of non-zero coordinates of $ \bm \ell = (\ell_1, \ell_2, \ell_3 ) $. 

In order to use the central limit theorem (\ref{CLTRF}), it is crucial to be in a position to estimate  $ \sigma^2 $ consistently, and the lack of such estimators, contrary to the time series literature, motivated this paper. For example, for the above mentioned image testing problem an asymptotic level $ \alpha $ test, $ \alpha \in (0,1) $, for $ H_0 $ can be devised by rejecting $ H_0 $ if
\[  
	|{\bm n} |^{-1/2} | S_{\bm n} | > \wh{\sigma}_{\bm n} \Phi^{-1}({1-\alpha/2}),
\] 
where $ \Phi$ denotes the distribution function of the standard normal law, as long as $ \wh{\sigma}_{\bm n}^2 $ is a consistent estimator of $ \sigma^2 $. 

In the same vain, one can construct a statistical test for a single image, in order to test for departures from a reference model $ m^{(0)} $ for the expectation of the two-dimensional random field $ \{ \xi_{\bm i} : \bm 0 < \bm i \le \bm n \} $ representing the image of pixel resolution $ n_1 \times n_2 $. In our simulations we do not only investigate the proposed estimators in their own right, but also examine the accuracy of such an image test in terms of the type I error rate, see Section~\ref{SimsH0Test}. It turns out that our approach allows for highly accurate test procedures for images as required by present day image analysis.

The aims and contributions of the present paper are therefore as follows: Going beyond the above scope of sequences of images, we consider general random fields of arbitrary dimension $q$ of the domain and taking values in the $p$-dimensional Euclidean space $ \R^p $. We propose and study a nonparametric estimator of the asymptotic variance, which directly generalizes the class of Bartlett-type long run variance estimators studied in time series analysis. In addition, we provide a central limit theorem for that estimator and establish the consistency of subsampling under weak conditions. A concise model when capturing a sample of images of constant scenery, e.g. in order to estimate the true underlying object by the noisy images and to estimate the spatial-temporal correlations, is to assume that the data is given by a superposition of a time series and a spatial random field. Then one may center the observed images at their temporal average. We show that the corresponding estimator of $ \sigma^2 $ is still consistent under fairly weak conditions. Lastly,
as spatial (or time) correlations in image data resp. sequences of images are often of a local nature, so that autocovariances corresponding to larger lags are negligible or even vanish, we introduce and study a class of cut-off or thresholding estimators, which resemble to some extent thresholding estimators studied in high-dimensional statistics. Those cut-off estimators aim at reducing the estimation variability by neglecting autocovariances of small order. We present extensive simulation results, in order to shed some light onto the accuracy of the proposed estimators, to identify situations where the thresholding estimator improves upon the classical lag-truncation approach, and to investigate how the estimators perform when used in statistical image testing.

Consider a general  $q$-dimensional real-valued random field $\{\xi_{\bm n}: \bm{n}\in\mathbb Z^q\}$ with $\mathbb E(\xi_{\bm n})=0$ for all $\bm{n}\in\mathbb Z^q$, where, as above, the first dimension typically represents time. Then the asymptotic variance of the random field is defined as
\[
  \sigma^2=\sum_{\bm{k}\in\mathbb{Z}^q}\mathbb E(\xi_{\bm{0}}\xi_{\bm{k}}).
\] 
When the random field attains values in $ \R^p $ for some $ p \in \N $, i.e. 
$ \xi_{\bm n} : (\Omega, \mathcal{F}, \mathbb P) \to (\R^p, \mathcal{B}^p) $, for $ {\bm n} \in \Z^q $, 
where $ (\Omega, \mathcal{F}, \mathbb P) $ denotes the underlying probability space and $ \mathcal{B}^p $ the usual Borel $ \sigma $-field on $ \R^p $, the asymptotic variance is the matrix
\[
	\sigma^2 = \lim_{\bm n \to \infty} \frac{1}{| \bm n |} \mathbb E( S_{\bm n} S_{\bm n}' ),
\]
since now $ \sigma_{\bm n}^2 = \Var( | \bm n |^{-1/2} S_{\bm n} ) = | \bm n |^{-1} \mathbb E(S_{\bm n} S_{\bm n}'  )$. 
Observe that $\sigma^2$ adds up all (cross-) autocovariances in the temporal and spatial domain. Obviously, for $ q = 1 $ we are given a univariate times series and $ \sigma^2 $ is the well known long run variance, for which an extensive literature on its estimation exists,  e.g. smoothing window type estimators, cf.\ \cite{LiuWu}, or estimators based on batched means, cf.\ \cite{AlexopoulosGoldsman} and \cite{Wu}.  But for a random field of dimension $ q > 1 $ there are only a few results which address certain special cases.
Indeed, most of the existing results are motivated from an economic point of view and concern the spatial heteroscedasticity and autocorrelation consistent estimation of covariance matrices with applications in two dimensions as in  \cite{DriscollKraay}, \cite{Conley}, \cite{KelejianPrucha} and \cite{KimSun}. \cite{DriscollKraay} obtain consistent estimates of the 
matrix of cross-sectional correlations by averaging over the time dimension, i.e.\ they request that the time dimension grows while the size of the cross-sectional dimension stays fixed. They construct an estimator that relies on the standard Newey and West estimator of the time series literature, see \cite{NeweyWest}, and show that it is robust to very general forms of spatial and temporal dependence as the time dimension becomes large.

\cite{KelejianPrucha} as well as \cite{KimSun} study spatial heteroscedasticity and autocorrelation consistent estimators of covariance matrices of parameter estimators, where the spatial dependence is measured by a so-called economic distance. If this economic distance $d_{ij,n}$ between two units $i$ and $j$ is small these units are highly dependent, if it is large, however, the units are nearly independent. Examples for an economic distance include geographic distances as well as transportation costs. Both papers also allow for errors in the measurement of the distance. The main disadvantage is, however, that both papers focus on linear processes with i.i.d. innovations, which include certain non-stationary models, but exclude non-linear models.

A slightly different approach is the one of \cite{Conley} who considers the estimation of the asymptotic variance for strictly stationary and $\alpha$-mixing random fields in two dimensions. He constructs an estimator as the weighted average of products of the observations and finally shows its consistency. 

However,  to the best of our knowledge, no further estimators of the asymptotic variance for arbitrary stationary mixing random fields in $q$ dimensions exist in the literature so far. Thus, in Section~\ref{EstAsyVar} we propose an estimator defined as a weighted sum of the sample autocovariances of the random field and show its consistency for mixing random fields and also provide results about the asymptotic distribution. In Section~\ref{ImpEst} we show how one can improve the estimator and propose a data-adaptive procedure to select remaining unknowns which make use of subsampling. Section~\ref{Simu} is devoted to an extensive simulation study of both estimators. The simulations study for several models the behavior of the estimators and demonstrate that the thresholding estimator is preferable. We also investigate the statistical properties of the image test and show that the thresholding estimator leads to very accurate image tests. For its application and accuracy in change detection, see \cite{Prause} and \cite{PrauseSteland}. Rigorous proofs of all results are provided in Section~\ref{Sec:Proofs}.

\section{Nonparametric Estimation of the Asymptotic Variance}
\label{EstAsyVar}

In this section, we introduce the proposed nonparametric estimator for the asymptotic variance of a random field. The estimator is based on the formula (\ref{sigmaquadrat}) and belongs to the class of lag-truncation estimators. We provide an interesting extension to the case of multiplicative random fields which are well suited to analyze image samples of a constant scenery.

\subsection{Estimation for general random fields}

Let $\left\{\xi_{\bm{i}}\right\}$ be a random field defined on $\Gamma_{\bm{n}}\coloneqq\{1,\ldots,n_1\}\times\ldots\times\{1,\ldots,n_q\}$ and taking values in $ \R^p $ with $ \max_{1 \le j \le p} \mathbb E ( \xi_{\bm i}^{(j)})^2  < \infty $ for $ {\bm i} \in \Gamma_{\bm n} $, where $ \xi_{\bm i} = ( \xi_{\bm i}^{(1)}, \dots, \xi_{\bm i}^{(p)} )' $. For a fixed $\bm{j}\in\mathbb Z^q$  denote by 
\[
  \gamma(\bm{j})=\mathbb E(\xi_{\bm{0}}\xi_{\bm{j}}'), \qquad {\bm j \in \Z^q}
\]  
the autocovariances of $ \{ \xi_{\bm i} \} $. Define the sample autocovariances
\begin{align}\label{autocovestest}
\widehat{\gamma}_{\bm{n}}(\bm{j})\coloneqq\frac{1}{\left|\widetilde{\Gamma}_{\bm{n}}(\bm{j})\right|}\sum_{{\bm i}\in\widetilde{\Gamma}_{\bm{n}}(\bm{j})}\xi_{\bm{i}}\xi_{\bm{i}+\bm{j}}', \qquad\widetilde{\Gamma}_{\bm{n}}(\bm{j})\coloneqq\left\{\bm{i}\in\Gamma_{\bm{n}}:\bm{i}+\bm{j}\in\Gamma_{\bm{n}}\right\},
\end{align}
and set
\begin{align}\label{sigmaquadrat}
\widehat{\sigma}_{\bm{n}}^2\coloneqq\sum_{|\bm{j}|\leq \bm{m}}w_{\bm{m}}(\bm{j})\widehat{\gamma}_{\bm{n}}(\bm{j}).
\end{align}
Here and  in what follows $|\bm{j}|\leq \bm{m}$ is understood component-by-component, i.e.\ $|j_i|\leq m_i$ for all $1\leq i\leq q$.  The weights $w_{\bm{m}}(\bm{j})$,
 for fixed $\bm{j}\in\mathbb Z^q$ and $\bm{m}=\bm{m}_{\bm{n}}\in\mathbb N^q$, 
 arising in (\ref{sigmaquadrat})  are assumed to satisfy the following assumptions.

(W1) $w_{\bm{m}}(\bm{j})\to 1$ as $\bm{m}\to\infty$ for all $\bm{j}\in\mathbb Z^q$.

(W2) $\left|w_{\bm{m}}(\bm{j})\right|\leq C_w<\infty$ independently of $\bm{j}$ and $\bm{m}$.

\noindent The weights $w_{\bm{m}}(\bm{j})$ can be chosen as the product of one-dimensional weights, i.e.\ we can put
$
w_{\bm{m}}(\bm{j})\coloneqq\prod_{i=1}^qw_{m_i}(j_i),
$
where the $w_{m_i}(j_i)$, $1\leq i\leq q,$ are weights satisfying (W1) and (W2). Examples for such weights are
the Bartlett weights,
$
  w_{m_i}(j_i)=1-\frac{j_i}{m_i}\textnormal{\ for\ } 1\leq i\leq q,
$
or the Tukey-Hanning weight sequence,
$
w_{m_i}(j_i)=\frac{1+\cos\left(\pi j_i/m_i\right)}{2}, 1\leq i\leq q.
$

For multivariate random fields, i.e. for $ p > 1 $, $ \sigma^2 $ and its estimator are matrix-valued, such that the evaluation of an estimator's accuracy requires to select a matrix norm. For simplicity of presentation,  in the sequel the matrix maximum norm on $ \R^{p\times p} $ defined by 
\[
\| \bm A \|_\infty = \max_{1 \le i, j \le p} | a_{ij} |, \quad \text{for a matrix
	$ \bm A = ( a_{ij} )_{1 \le i \le p \atop 1 \le j \le p } $},
\]
will be used. One could also employ the frequently used Frobenius norm $ \| \bm A \|_F = \sqrt{\sum_{i,j=1}^p |a_{ij}|^2} $, which, however, satisfies
\[
\| \bm A \|_\infty = \sqrt{ \max_{i,j} | a_{ij} |^2 } \le \| \bm A \|_F \le p \| \bm A \|_\infty,
\]
such that all results can be easily reformulated in terms of $ \| \cdot \|_F $ instead of $ \| \cdot\|_\infty $, or any other matrix norm, since all norms on $ \R^{p \times p} $ are equivalent.

Our main results require the random field to be $ \varphi $-mixing. For related and other notions of weak dependence for random fields we refer to \cite{Doukhan} and \cite{Bradley} and the discussion at the end of this section. Let us briefly recall  the definition of $ \varphi $-mixing.

For each $j$ with $1\leq j\leq q$ and for each $r\geq 0$, define 
\begin{align}
\label{DefA+}
\mathcal A^+(j;r) & \coloneqq \sigma\left(\left\{\xi_{n_1,\ldots,n_q}: n_j\geq r, n_i\ \textnormal{unrestricted for}\ i\neq j \right\}\right), \\
\label{DefA-}
\mathcal A^-(j;r) & \coloneqq \sigma\left(\left\{\xi_{n_1,\ldots,n_q}: n_j\leq r, n_i\ \textnormal{unrestricted for}\ i\neq j \right\}\right).
\end{align}
Further, for each $r\geq 1$ introduce
\begin{align}\label{phijr}
\varphi(j;r)\coloneqq\sup\left\{\left|\mathbb P(B|A)- \mathbb P(B)\right|: A\in\mathcal A^-(j;0), B\in\mathcal A^+(j;r), P(A)>0\right\}
\end{align}
and define the (half-space) $ \varphi $-mixing coefficients
\begin{align}\label{phimax}
\varphi(r)\coloneqq\max_{1\leq j\leq q}\varphi(j;r).
\end{align}
Putting $\varphi(0)=1$ we can observe that the set $\{\varphi(r)\}$ is a decreasing sequence of real numbers. The  random field $\{\xi_{\bm n}: \bm{n}\in\mathbb Z^q\}$ is now called $\varphi$-mixing, if $\varphi(r)\to 0$ as $r\to\infty$.

\textbf{Assumption 1:}
Let $\left\{\xi_{\bm{i}},\bm{i}\in\Gamma_{\bm{n}}\right\}$ be a strictly stationary,  $\varphi$-mixing random field with $\mathbb E(\xi_{\bm{0}})=\bm 0$, $\mathbb \max_{1 \le j \le p} \mathbb E\left(\xi_{\bm{0}}^{(j)}\right)^4<\infty$ and mixing coefficients satisfying
\begin{align}\label{condphimix}
\sum_{r=1}^{\infty}r^{q-1}\varphi^{\frac{1}{2}}(r)<\infty.
\end{align}

We are now in a position to formulate the following theorem on the consistency of the estimator $\widehat{\sigma}_{\bm{n}}^2$ under mild regularity conditions for a general mutivariate random field.

\begin{theorem}\label{consistencyvariance} Suppose that Assumption 1 holds and that
the weights $w_{\bm{m}}(\bm{j})$ fulfill $(W1)$ and $(W2)$. Furthermore, assume that $ \bm m \to \infty $, as $ \bm n \to \infty $, with 
\begin{align}\label{mnconv}
\frac{\left(m^{\star}\right)^3}{n_{\star}}=o(1),
\end{align}
where $m^{\star}=\max_{1\leq i\leq q}m_i$ and $n_{\star}=\min_{1\leq i\leq q}n_i$.
Then the estimator $\widehat{\sigma}_{\bm{n}}^2$ is weakly consistent, i.e.\
\[
  \| \widehat{\sigma}_{\bm{n}}^2 - \sigma^2 \|_\infty  \stackrel{\mathbb P}\to 0
\]
as $\bm{n}\to\infty$.
\end{theorem}

It is interesting to note that the consistency can be strengthened to $ L_2 $-consistency resp. $ L_1 $-consistency, as shown by the following theorem.

\begin{theorem}\label{consistencyvariancel2}
Let the conditions of Theorem~\ref{consistencyvariance} be satisfied. Then for $ p = 1 $ 
\[
\mathbb E\left[\left(\widehat{\sigma}_{\bm{n}}^2-\sigma^2\right)^2\right] \to 0
\]
as $\bm{n}\to\infty$, and for $ p > 1 $ 
\[
\mathbb E\left\| \widehat{\sigma}_{\bm{n}}^2-\sigma^2\right\|_\infty \to 0,
\]
as $\bm{n}\to\infty$.
\end{theorem}

A direct consequence of Theorem~\ref{consistencyvariance} is the consistency of the autocovariance estimators, $\widehat{\gamma}_{\bm{n}}(\bm{j})$, used to define the estimator of $ \sigma^2 $, which are, of course, of independent interest when analyzing random field data.

\begin{corollary}\label{autocovest}
The estimator $\widehat{\gamma}_{\bm{n}}(\bm{j})$ of the lag $ {\bm j} $-autocovariance of random field $ \{ \xi_{\bm i} \} $, 
\[
\widehat{\gamma}_{\bm{n}}(\bm{j})\coloneqq\frac{1}{\left|\widetilde{\Gamma}_{\bm{n}}(\bm{j})\right|}\sum_{i\in\widetilde{\Gamma}_{\bm{n}}(\bm{j})}\xi_{\bm{i}}\xi_{\bm{i}+\bm{j}}'
\]
for $\bm{j}\in\mathbb Z^q$, is $L_r$- and thus also weakly consistent for the lag $ {\bm j} $ autocovariance $\gamma(\bm{j})$, where $ r = 2 $ for $ p = 1 $ and $ r = 1 $ otherwise.
\end{corollary}

\subsection{Asymptotic distribution theory}

Let us now turn to the asymptotic distribution theory, which we study for univariate random fields.
We show that, under weak regularity conditions, the sample autocovariances are asymptotically Gaussian and provide an associated limit theorem for the estimator $ \widehat{\sigma}_{\bm n}^2 $. The latter is interesting in its own right, but is also needed to establish the subsampling central limit theorem to be discussed and applied in the next section.

We need the following approximation result, which shows that the sample autocovariances can be scaled with $ \sqrt{ | \wt{\Gamma}_{\bm n}( \bm j) | } $ or $ \sqrt{ | \bm n | } $.

\begin{lemma} 
\label{CLTLemma}
	Under Assumption 1 it holds
	\[ 
	\max_{|\bm j| \le \bm m} \left\|  \sqrt{ | \widetilde{\Gamma}_{\bm n}( \bm j ) | } ( \wh{\gamma}_{\bm n}( \bm j ) - \gamma( \bm j) ) - \frac{1}{\sqrt{|\bm n|}} \sum_{ {\bm i} \in 1:{\bm n}} ( \xi_{\bm i} \xi_{\bm i+\bm j} - \gamma(\bm j)) \right\|_\infty = o_{\mathbb P}(1),
   \]
   as $ \bm n \to \infty $.
\end{lemma}

We have the following result providing the weak convergence of the sample autocovariances and the estimator $ \widehat{\sigma}_{\bm n}^2 $ when appropriately centered and scaled. Let us denote by $ \Rightarrow $ the weak convergence in the Euclidean space $ \R^l $, $ l \in \N $. Recall that a matrix $ \bm A = ( a_{\nu \mu} )_{\nu\mu} $ of dimension $ p \times p $ defines the quadratic form $ (\lambda_\nu )_{\nu=1}^p \mapsto \sum_{\nu, \mu=1}^p \lambda_\nu \lambda_\mu a_{\nu \mu} $,  $ (\lambda_\nu )_{\nu=1}^p \in \R^p $, which is positive definite if it is positive for all $ (\lambda_\nu )_{\nu=1}^p  \not= \bm 0 $. More generally, a scheme $ ( a_{\bm j}^{(\nu,\mu)} : | \bm j |\le \bm m, |\bm k|  \le \bm m, 1 \le \nu, \mu \le p ) $ of real numbers defines a quadratic form via  $ ( \lambda_{\bm j}^{(\nu)} : | \bm j | \le \bm m, 1 \le \nu \le p ) \mapsto \sum_{| \bm j | \le \bm m} \sum_{| \bm k | \le \bm m} \sum_{\nu, \mu=1}^p \lambda_{\bm j}^{(\nu)}  \lambda_{\bm k}^{(\mu)} a_{\bm j}^{(\nu,\mu)} $, which is positive definite if it is positive for all $  ( \lambda_{\bm j}^{(\nu)} : | \bm j | \le \bm m, 1 \le \nu \le p )  \not= \bm 0 $.

\begin{theorem}
	\label{CLT}
	Suppose Assumption 1 holds true  with $ p = 1 $ and fix $ \bm m \ge 1 $.
\begin{itemize}
	\item[(i)]  $ \sqrt{ | \widetilde{\Gamma}_{\bm n}( \bm j ) | } ( \wh{\gamma}_{\bm n}( \bm j ) - \gamma( \bm j) ) \Rightarrow B_{\bm j} $, 
  as $ \bm n \to \infty $, where $ ( B_{\bm j} : |\bm j| \le \bm m ) $ is Gaussian with mean zero and covariances given by 
  \begin{equation}
  \label{CovFunctionB_univ}
    \Cov( B_{\bm j}, B_{\bm j'} ) = \sum_{\bm i \in \Z^q} \mathbb E( \xi_{\bm 0} \xi_{\bm j} - \gamma(\bm j) ) (  \xi_{\bm i} \xi_{\bm i + \bm j'} - \gamma(\bm j')  ),
  \end{equation}
  for $ |\bm j| \le \bm m $ and $  |\bm j'| \le \bm m$, provided (\ref{CovFunctionB_univ}) induces a positive definite quadratic form.
  \item[(ii)] For $ \bm m \ge \bm 1 $ we have
  \[
    \sqrt{ | \bm n |} ( \wh{\sigma}_{\bm n}^2 - \mathbb E(  \wh{\sigma}_{\bm n}^2) ) \Rightarrow 
    \mathcal{S},
  \]
  as $ \bm n \to \infty $, where $ \mathcal{S} \stackrel{d}{=} \sum_{|\bm j| \le \bm m} w_{\bm m}(\bm j) B_{\bm j} $.
\end{itemize}
\end{theorem}

For a multivariate random field the autocovariances $ \gamma(\bm j) $ are $ p \times p $ matrices with elements $  \gamma(\bm j)_{\nu\mu} $, $ 1 \le \nu, \mu \le p $. We have the following multivariate extension of Theorem~\ref{CLT}.

\begin{theorem}
	\label{CLTmultiv}
	Suppose that Assumption 1 holds true. 
	\begin{itemize}
		\item[(i)] We have 
$ \sqrt{ | \widetilde{\Gamma}_{\bm n}( \bm j ) | } ( \wh{\gamma}_{\bm n}( \bm j ) - \gamma( \bm j) ) \Rightarrow B_{\bm j} $, 
as $ \bm n \to \infty $, with $ B_{\bm j} = ( B_{\bm j}^{(\nu,\mu)} )_{1 \le \nu \le p \atop 1 \le \mu \le p } $, for
$ |\bm j | \le \bm m $, where $ ( B_{\bm j}^{(\nu,\mu)} ) $ is Gaussian with mean zero and covariances given by 
\begin{equation}
\label{CovFunctionB}
  \Cov( B_{\bm j}^{(\nu,\mu)}, B_{\bm j'}^{(\nu',\mu')} ) = \sum_{\bm i \in \Z^q} \mathbb E( \xi_{\bm 0}^{(\nu)} \xi_{\bm j}^{(\mu)} - (\gamma(\bm j) )_{\nu\mu} ) (  \xi_{\bm i}^{\nu'} \xi_{\bm i + \bm j'}^{(\mu')} - (\gamma(\bm j'))_{\nu'\mu'} ),
\end{equation}
for $ 1 \le \nu, \mu, \nu', \mu' \le p $, $ |\bm j| \le \bm m $ and $  |\bm j'| \le \bm m$, provided (\ref{CovFunctionB}) defines a positive definite quadratic form.
\item[(ii)] For $ \bm m \ge \bm 1 $ we have under the assumption of (i)
\[
	\sqrt{ | \bm n |} ( \wh{\sigma}_{\bm n}^2 - \mathbb E(  \wh{\sigma}_{\bm n}^2) ) \Rightarrow \mathcal{S},
\]
as $ \bm n \to \infty $, where $ \mathcal{S} \stackrel{d}{=} \sum_{|\bm j| \le \bm m} w_{\bm m}(\bm j) B_{\bm j} $.
\end{itemize}
\end{theorem}

It is worth mentioning that, by Lemma~\ref{propertiesY} and Lemma~\ref{RPHI} (a), the asymptotic variance of the random field $ \xi_{\bm i} \xi_{\bm i + \bm j} - \gamma( \bm j ) $, $ \bm i \in \Z^q $, 
\[
  \zeta_{\bm j}^2 = \Var( B_{\bm j} ) = \sum_{\bm i \in \Z^q} \mathbb E( \xi_{\bm 0} \xi_{\bm j} - \gamma( \bm j ))( \xi_{\bm i} \xi_{\bm i + \bm j} - \gamma( \bm j ) ),
\]
exists for each $ \bm j \in \Z^q $.

Next we aim at studying the centered and scaled sample autocovariances as a process indexed by the lag $ \bm j $, thus extending the above results. Observe that for fixed $ \bm n $ the estimator $ \wh{\gamma}_{\bm n}( \bm j ) $ is only well defined for $ \bm j \le \bm n $, and it is common to put $ \wh{\gamma}_{\bm n}( \bm j ) = \bm 0 $ if $ j_i > n_i $ for some $ i \in \{1, \dots, q \} $. This motivates to consider the multivariate random field
	\[
	  G_{\bm n}(\bm j) =  \sqrt{ | \widetilde{\Gamma}_{\bm n}( \bm j ) | } ( \wh{\gamma}_{\bm n}( \bm j ) - \gamma( \bm j) ), \qquad \bm j \in \Z^q,
	\]
	indexed by $ \Z^q $ and thus taking values in the space $ S = (\R^{p \times p} )^{\Z^q} $ of mappings $ \Z^q \to \R^{p \times p}  $. The space $ S = (\R^{p \times p})^{\Z^q} $ is a separable and complete metric space when equipped with the metric
\[ \rho( x, y) = \sum_{\bm k \in \Z^q} 2^{-\bm k} d_0( x_{\bm k}, y_{\bm k}  ), \] 
for $ x = \{ x_{\bm k} : \bm k \in \Z^q \}, y = \{ y_{\bm k} : \bm k \in \Z^q \} \in (\R^{p\times p})^\infty $, where $ d_0( \bm a, \bm b ) = \min( 1, \| \bm a - \bm b \|_\infty) $ for matrices $ \bm a, \bm b \in \R^{p \times p} $ and $ 2^{-|\bm k|} = \prod_{j=1}^k 2^{-|k_j|} $ for $ \bm k = (k_1, \dots, k_q) \in \Z^q $.

The question arises whether $ G_{\bm n} $ converges weakly in the space $ (S, \rho) $, as Theorem~\ref{CLTmultiv} already provides the convergence of the finite-dimensional distributions. Since weak convergence in $ (S,\rho) $ is not elaborated in the literature, let us briefly discuss some details. First note that 
convergence with respect to $ \rho $ is pointwise convergence. For $ x \in S $ and $ t_1, \dots, t_k \in \Z^q $ , $ k \in \N $, the projection $ \pi_{t_1, \dots, t_k} = ( x_{t_1}, \dots, x_{t_k} ) \in \R^{p \times pk} $, is continuous. The finite-dimensional distributions of a random element $ X = ( X_{\bm k} : \bm k \in  \Z^q )  $ taking values in $ (\R^{p \times p})^{\Z^q} $ are given by the laws of the random matrices $ ( X_{t_1}, \dots, X_{t_n} ) $ of dimension $ p \times (pn) $,  or, equivalently, 
by the laws of the $p^2n$-dimensional random vectors $ ( \text{vec} X_{t_1}, \dots \text{vec} X_{t_n}' )' $, for $ t_1, \dots, t_n \in \N $, where $ \text{vec} A $ denotes the vector obtained by stacking the columns of a matrix $ A $. The finite-dimensional sets, $  \pi_{t_1, \dots, t_k}^{-1}( H )  = \{ z \in S : ( z_{t_1}, \dots, z_{t_k} ) \in H $, $ H \subset \R^{p\times pk} $ measurable, are a $  \pi $-system.  For $ x \in S $ and $ \varepsilon > 0 $ let $ \mathcal{A}_{x,\varepsilon} $ be the system of sets $ A $ with $\mathring{A} \subset A \subset B(x,\varepsilon) $, where $ B(x, \varepsilon) $ is the open ball around $x$ with radius $ \varepsilon $. Choose $ \bm k \in \Z^q $ such that $ \sum_{\bm j \in \Z^q} 2^{-|\bm j|} - \sum_{|\bm j| \le \bm k} 2^{-|\bm j|} < \varepsilon/2 $. Consider now the uncountable many disjoint sets
\[ A_{\eta} = \{ y \in S : \| y_{\bm i} - x_{\bm i} \|_\infty \le \eta, | \bm i | \le |\bm k |, \text{\ and\ } \| y_{\bm i_0} - x_{\bm i_0} \|_\infty < \eta , \text{for some }  \bm i_0  \text{ with }  | \bm i_0 | \le | \bm k| \} \] 
for $ 0 < \eta < \varepsilon/(2 \cdot 3^q) $. Since $ \sum_{ \bm j \in \Z^q } 2^{-|\bm j|} = 3^q $ and $ d_0 $ is bounded by $1$, any $ y \in A_\eta $ satisfies
\[
  \rho(x,y) \le \sum_{|\bm j| \le \bm k} 2^{-|\bm k|} d_0( x_{\bm k}, y_{\bm k} ) + \varepsilon/2 \le \eta 3^q + \varepsilon/2 < \varepsilon.
\]
Hence, $ A_\eta \subset B(x,\varepsilon) $. This shows that the system of boundaries of the sets $ \mathcal{A}_{x,\varepsilon} $ contains uncountably many disjoint sets, such that by \cite{Bil1999} the weak convergence in $ (\R^{p \times pn})^{\Z^q} $ coincides with the convergence of the finite-dimensional distributions.

Therefore, Theorem~\ref{CLTmultiv} implies the following result.

\begin{theorem} 
	\label{CLTmultiv2}
	Under the conditions of Theorem~\ref{CLTmultiv}
\[
  G_{\bm n} \Rightarrow B,
\]	
as $ \bm n \to \infty $, where $ B = \{ B_{\bm j} : \bm j \in \Z^q \} $.
\end{theorem}

\subsection{Subsampling}
\label{Sec:Subsampling}

Let us briefly review how subsampling for random fields works. For more details we refer to \cite{PolitisRomanoWolf}, especially Chapter~5.3 therein. Define
\[
\bm E_{\bm{u}}\coloneqq\{\bm t\in\mathbb Z^q: 0<t_k\leq u_k, k=1,2,\ldots,q\}
\]
and suppose that given the random field $\{X(\bm{t}),\bm{t}\in \bm E_{\bm{n}}\}$, we have a consistent estimator $\widehat{\theta}_{\bm{n}}=\widehat{\theta}_{\bm{n}}(X(\bm{t}),\bm{t}\in \bm E_{\bm{n}})$ of an unknown real-valued parameter $\theta(\mathbb P)$. If we define $J_{\bm{n}}( \mathbb P)$ as the sampling probability law\index{sampling probability law} of $\tau_{\bm{n}}\left(\widehat{\theta}_{\bm{n}}-\theta(\mathbb P)\right)$ under $ \mathbb P $, where $\tau_{\bm{n}}$ is a normalizing constant, and write $J_{\bm{n}}(x,\mathbb P)$ for the corresponding sampling probability distribution function\index{sampling probability distribution function}, i.e.\
\[
J_{\bm{n}}(x,\mathbb P)\coloneqq \mathbb P\{\tau_{\bm{n}}(\widehat{\theta}_{\bm{n}}-\theta(\mathbb P))\leq x\},
\qquad x \in \R,
\]
the aim of subsampling is to find an approximation of $J_{\bm{n}}(x,\mathbb P)$ 
by recomputing the statistic $\widehat{\theta}_{\bm{n}}$ over random fields of smaller size than $\bm{n}$ and by considering the
empirical distribution function of these subsampled values.
Thus, for $\bm{b},\bm{h}\in\mathbb Z^q$ we define these smaller random fields by
\[
Y_{\bm{j}}\coloneqq\{X(\bm{t}),\bm{t}\in \bm E_{\bm{j},\bm{b},\bm{h}}\},
\]
where $E_{\bm{j},\bm{b},\bm{h}}$ stands for the rectangle containing the points $\bm{i}\in\mathbb Z^q$ with $(j_k-1)h_k<i_k\leq(j_k-1)h_k+b_k$ for $1\leq k\leq q$. The point $\bm{b}$ represents the size of the smaller random field defined on $E_{\bm{j},\bm{b},\bm{h}}$, the point $\bm{h}$ determines how many of these smaller random fields are taken into account, as $Y_{\bm{j}}$ is only defined if $0<j_k\leq N_k$, where $N_k=\left\lfloor \frac{n_k-b_k}{h_k}\right\rfloor+1$. This means, the vector $ \bm h $ defines a grid and at each point of that grid a rectangle (block) defined by the block size $ \bm b $ is located. Depending on the choice of $ \bm h $ and $ \bm b $, those rectangles may overlap or not.
Now, one can easily see that the number of considered subrandom fields decreases when $\bm{h}$ increases and that it increases when $\bm{h}$ decreases with a maximum for $\bm{h}=(1,1,\ldots,1)$.

We further define the subsample value $\widehat{\theta}_{\bm{n},\bm{b},\bm{i}}$ as the statistic $\widehat{\theta}_{\bm{b}}$ evaluated at the smaller random field $Y_{\bm{i}}$, i.e.\ $\widehat{\theta}_{\bm{n},\bm{b},\bm{i}}\coloneqq\widehat{\theta}_{\bm{b}}(Y_{\bm{i}})$. The desired subsampling approximation\index{subsampling approximation} of the sampling probability distribution function $J_{\bm{n}}(x,\mathbb P)$ is then defined as
\[
L_{\bm{n},\bm{b}}(x)\coloneqq |\bm N|^{-1}\sum_{i_1=1}^{N_1}\sum_{i_2=1}^{N_2}\ldots\sum_{i_q=1}^{N_q}\mathds{1}_{\{\tau_{\bm{b}}(\widehat{\theta}_{\bm{n},\bm{b},\bm{i}}-\widehat{\theta}_{\bm{n}})\leq x\}},
\]
where $|\bm N|=\prod_{i=1}^q N_i$.

The question arises under which conditions this approximation is consistent. The main assumption for this to hold is that $J_{\bm{n}}(\mathbb P)$ converges weakly to a limit law $J(\mathbb P)$ with corresponding distribution function $J(x,\mathbb P)$, as $\bm n \to \infty$; the precise assumptions are as in \cite[Theorem~5.3.1]{PolitisRomanoWolf} and are listed in the theorem below, which is adapted to our setting. The assumptions mainly control the growth rate of $\bm{b}$ with respect to $\bm{n}$ and the mixing coefficients of the underlying random field. A common choice studied in greater detail in Section~\ref{Sec82} is $ \bm b = \lfloor \bm n^\gamma \rfloor $ for $ \gamma \in (0,1) $.

The following theorem establishes the consistency of subsampling for statistics calculated from the underlying random field or, going beyond that case, which depend on (collections) of the terms arising in the lag $ \bm j $ sample autocovariances. Define $ Y_{\bm i}( \bm j) = \xi_{\bm i} \xi_{\bm i + \bm j} - \gamma( \bm j ) $, $ \bm i \in \Z^q $, for any $ \bm j \in \Z^q $.

\begin{theorem}
	\label{Subsampling}
	Let $ \widehat{\theta}_{\bm n} $ be a real-valued statistic depending on $ \{ \xi_{\bm i} : \bm 0 < \bm i \le \bm n \} $, $ \{ Y_{\bm i}( \bm j) : \bm i \in \wt{\Gamma}_{\bm n}(\bm j)  \} $ or $  \{ Y_{\bm i}(\bm j) : \bm i \in \wt{\Gamma}_{\bm n}(\bm j), |\bm j | \le \bm m \} $, used for estimating the unknown real-valued parameter $ \theta(\mathbb P) $, such that
	\[
	J_{\bm n}( x, \mathbb P ) = \mathbb P( |\bm n|^{-1/2}( \widehat{\theta}_{\bm n} - \theta( \mathbb P )) \le x )
	\Rightarrow J(x,\mathbb P),
	\]
	as $ \bm n \to \infty $. Assume that $ \bm h$ is a non-zero constant and assume that
	\begin{equation}
	\label{Subs1}
	\prod_{j=1}^q b_j / (n_j-b_j) = o(1),
	\end{equation}
	as $ \bm n \to \infty $, and 
	\begin{equation}
	\label{WeakMixingCondSubs}
	\frac{1}{|\bm N|} \sum_{k=1}^{N^\star} k^{q-1} \varphi(k-2m^\star) = o(1),
	\end{equation}
	as $ \bm n \to \infty $. Then, in continuity points $ x $ of $ J(x,\mathbb P) $,
	\[
	L_{\bm n, \bm b} \stackrel{\mathbb P}{\to} J(x,\mathbb P),
	\]
	as $ \bm n \to \infty $. If $ J(\cdot, \mathbb P ) $ is continuous, then 
	\[ 
	q_{\bm n, \bm b}(\gamma) = \inf \{ x : L_{\bm n, \bm b}(x) \ge \gamma \}
	\]
	converges in probability ot $ q(\gamma) = \inf\{ x : J(x, \mathbb P ) \ge \gamma \} $ and
	\[
	\mathbb P( |\bm n|^{-1/2}( \widehat{\theta}_{\bm n} - \theta(\mathbb P) ) \le q_{\bm n, \bm b}( \gamma ) ) \to \gamma,
	\]
	as $ \bm n \to \infty $.
\end{theorem}

\subsection{Extensions to multiplicative models}

An important subclass of spatial-temporal random fields arises as the additive superposition of a time series introducing the serial dependencies and a spatial univariate random field, i.e.
\[
  Y_{\bm i} = \eta_{i_1}^{(T)} + \eta_{\overline{\bm{\iota}}}^{(S)}, \quad\bm i=(i_1,\overline{\bm{\iota}})\in\mathbb Z^q.
\]
As discussed in the introduction, this is a reasonable model when capturing, say, $ n_1 $ images of a fixed experimental situation, in order to estimate the image, its spatial dependence structure and  the serial correlations of the data acquisition process. For a related detection procedures we refer to \cite{PrauseSteland2015} and \cite{PrauseSteland}, Section 4.2. Obviously, the above model can be also formulated as a multiplicative model, and thus we assume from now on that 
$ \xi_{\bm i} = ( \xi_{\bm i}^{(1)}, \dots, \xi_{\bm i}^{(p)} )' $ with 
\begin{align}\label{multimodel}
\xi_{\bm i}^{(j)}  \coloneqq\varepsilon_{i_1}^{(T)} \varepsilon_{\overline{\bm{\iota}}}^{(S,j)},\quad\bm i=(i_1,\overline{\bm{\iota}})\in\mathbb Z^q, \ j = 1, \dots, p,
\end{align}
where $\varepsilon^{(T)}=\{\varepsilon_{i}^{(T)}:i\in\mathbb Z\}$ is a strictly stationary second order $\varphi$-mixing univariate time series and $\varepsilon^{(S)}=\{\varepsilon_{\overline{\bm{\iota}}}^{(S)}:\overline{\bm{\iota}}\in\mathbb Z^{q-1}\}$ is a strictly stationary  $\varphi$-mixing multivariate random field taking values in $ \R^p $ with $ \max_{1 \le j \le p} \mathbb E( ( \varepsilon_{\bm 0}^{(S,j)} )^4 )<\infty$. We further assume that $\varepsilon^{(T)}$ and $\varepsilon^{(S)}$ have $\varphi$-mixing coefficients satisfying (\ref{condphimix}) for dimension 1 and $q-1$, respectively, and are independent from each other.

In the present setting, one may center the $ \xi_{\bm i} $ at their temporal average and therefore we define
\begin{align*}
\check{\gamma}_{\bm n}(\bm j)\coloneqq\frac{1}{\left|\widetilde{\Gamma}_{\bm n}(\bm j)\right|}\sum_{(i_1,\overline{\bm{\iota}})\in\widetilde{\Gamma}_{\bm n}(\bm j)}\left(\xi_{\bm{i}}-\overline{\xi}_{\cdot,\overline{\bm{\iota}}}\right)\left(\xi_{\bm{i+j}}-\overline{\xi}_{\cdot,\overline{\bm{\iota}}}\right)',
\end{align*}
where
$\overline{\xi}_{\cdot,\overline{\bm{\iota}}}\coloneqq\frac{1}{n_1}\sum_{l=1}^{n_1}\xi_{l,\overline{\bm{\iota}}} $,
and introduce the estimator
\begin{align*}
\check{\sigma}^2_{\bm n}\coloneqq\sum_{|\bm j|\leq \bm m}w_{\bm m}(\bm j)\check{\gamma}_{\bm n}(\bm j)
\end{align*}
for $ \sigma^2 $. 
We then obtain the following theorem.
\begin{theorem}\label{unknownrefsignal}
Suppose the noise process satisfies in addition to the assumptions of Theorem \ref{consistencyvariance}, the multiplicative model (\ref{multimodel}). Then $\check{\sigma}^2_{\bm n} \stackrel{\mathbb P}{\to} \sigma^2$, as $\bm n\to\infty$,
and $ \mathbb E \| \check{\sigma}^2_{\bm n} - \sigma^2 \|_\infty \to 0 $, as $\bm n\to\infty$.
\end{theorem}

\subsection{Discussion of the mixing assumptions}

The above results assume that the random field is $ \varphi $-mixing. Inspecting the proof shows that this condition can be replaced by other weak dependence conditions ensuring that
\begin{equation}
\label{WhatWeReallyNeed}
| {\bm n } |^{-1} \mathbb E \left| \sum_{\bm k \in \bm 1:\bm n} \xi_{\bm 0} \xi_{\bm k + \bm j} - \gamma( \bm j ) \right|^2 = O(1),
\end{equation}
as $ \bm n \to \infty $, cf. Lemma~\ref{RPHI}. For example, a natural condition is to assume that $ \xi_{\bm k} $ is $ \rho^* $-mixing, which implies that the strictly stationary random field $ Y_{\bm k}(\bm j) = \xi_{\bm 0} \xi_{\bm k + \bm j} - \gamma( \bm j ) $ has a continuous spectral density function, cf. Lemma~\ref{LemmaSpectral}, which is then given by
\[
f(t) = \lim_{n \to \infty} n^{-q} \mathbb E \left| \sum_{\bm k \in \bm 1: n \bm 1 } e^{-i \bm k'\lambda_t} Y_{\bm k}( \bm j) \right|^2, \qquad t \in S^q.
\]
where from now on we put $ n \bm 1 = (n, \dots, n) $ for $ n \in \N $. Further,
$ S^q = \{ \bm x \in \R^q : \| x \|_2 = 1 \} $ and $ \lambda_t $ is the unique number $ \lambda \in (- \pi, \pi] $ such that
$ e^{-i \lambda_j} = t_j $ for $ j =1, \dots, q $, see \cite[Theorem~28.21]{Bradley}. Here $ \{ \xi_{\bm k} \} $ is $ \rho^*$-mixing, if $ \rho^*(r) \to 0 $, as $ r \to \infty $, where
\[ 
\rho^*(r) = \sup \rho( \sigma( \xi_{\bm k} : \bm k  \in A_1 ), \sigma( \xi_{\bm k} : \bm k \in A_ 2) ).
\]
The supremum is taken over all nonempty disjoint sets $ A_1, A_2 \subseteq \Z^q $ with distance $ \text{dist}(A_1,A_2) \ge r $, and the $ \rho$-mixing coefficient $ \rho( \mathcal{G}, \mathcal{H} ) $ for two sub-$\sigma$-fields $ \mathcal{G} $ and $ \mathcal{H} $ of $ \mathcal{F} $ is defined as 
\[ \rho( \mathcal{G}, \mathcal{H} ) = \sup | \text{Cor}( U, V ) |, \]
where the supremum is taken over all $ \mathcal{G}$-measurable random variables $ U $ and all $ \mathcal{H} $-measurable random variables $V$, both with finite second moment. The distance between two sets $ A_1, A_2 \subseteq \Z^q $ is defined by $ \text{dist}(A_1,A_2) = \inf\{ \| a_1 - a_2 \|_2 :  a_1 \in A_1, a_2 \in A_2 \}$, where $ \| \cdot \| $ is the usual Euclidean vector norm. 

It is also worth mentioning that \cite{BerkesMorrow} establish a strong invariance principle for partial sums of $ \alpha $-mixing random fields taking values in $\R^p $ and (\ref{WhatWeReallyNeed}), if the mixing coefficients, defined as 
\[
  \alpha( \bm E_1, \bm E_2 ) = \sup_{A \in \sigma( \xi_{\bm i} : \bm i \in \bm E_1 ) \atop B \in \sigma( \xi_{\bm i} : \bm i \in \bm E_2 ) } | \mathbb P( A \cap B ) - \mathbb P(A) \mathbb P(B) | 
\]
for any two disjoint non-empty sets $ \bm E_1, \bm E_2 \subseteq \N^{q} $, can be bounded in terms of the distance, namely by
\[
  \alpha( \bm E_1, \bm E_2 ) \le C \left( \text{dist}( \bm E_1, \bm E_2 )  \right)^{-q(1+\varepsilon)(1+2/\delta)}
\]
for some $ 0 < \varepsilon < 1/2 $ and $ \delta > 0 $ is such that $ E | \xi_{\bm 0}^{(j)} |^{2+\delta} < \infty $  holds. The result of \cite{BerkesMorrow} is limited to upper summation limits $ \bm n $ satisfying a constraint. That constraint is, however, satisfied for the important special case of proportional sample sizes, i.e. $ n_j = c_j n_1 $, $ j = 2, \dots, q $, holds for constants $ c_2, \dots, c_q $. The strong invariance principle of \cite{BerkesMorrow} and (\ref{WhatWeReallyNeed}) have been also established by \cite{BulinskiShaskin2006} for random fields, which are weakly dependent in the sense that the covariance between
$ f(\xi_{\bm i} : \bm i \in \bm I ) $ and $ g(\xi_{\bm i} : \bm i \in \bm J ) $, for any pair of disjoint finite sets $ \bm I, \bm J $ with sup-distance $r$, see (\ref{DefSupDistance}) for a definition of the latter distance, and any pair of bounded Lipschitz functions $ f $ and $ g $, is of the order $ \theta_r $, for a sequene $ \theta_r $ decaying exponentially fast.

\section{Threshold Cut-Off Estimation of the Asymptotic Variance}\label{ImpEst}
As the simulation studies in Section~\ref{SimVE} will show, the RMSE of the variance estimator $\widehat{\sigma}_{\bm{n}}^2$ increases for larger values of $\bm{m}$.  Moreover, the smallest possible RMSE in most situations is still quite high. 
This problem often arises when the spatial autocovariance matrix is sparse, i.e.\ the number of non-vanishing entries is much smaller than the number of null entries, or when it is close to sparsity. The latter typically occurs when the autocovariances decrease fast. This estimation of null entries increases the variance but does not reduce the bias. By thresholding sample autocovariances which are small in magnitude, we may reduce the variance. 

We allow for parameterized thresholding functions and propose to determine remaining unknown parameters by a subsampling procedure, which implicitly determines an approximating $m$-dependent random field. For that purpose, we adopt a known result about the consistency of subsampling for random fields, see \cite{PolitisRomanoWolf}, to the setting of this paper.

\subsection{Cut-off estimation}

These observations motivate to propose a thresholding estimator for $\sigma^2$, where the idea is to multiply $\widehat{\gamma}_{\bm{n}}(\bm{j})$ with some weight function $g$. This leads to the following estimator $\widehat{\sigma}_{\bm{n},th}^2$, defined as
\begin{align}\label{improvedvarestthres}
\widehat{\sigma}_{\bm{n},th}^2\coloneqq\sum_{|\bm{j}|\leq \bm{m}}w_{\bm{m}}(\bm{j})\widehat{\gamma}_{\bm{n}}(\bm{j})g\left(\widehat{\gamma}_{\bm{n}}(\bm{j}),c_{\bm{n}}(\bm{j})\right),
\end{align}
where $g$ is a function depending on $\widehat{\gamma}_{\bm{n}}(\bm{j})$ and some sequence $c_{\bm{n}}(\bm{j})$ and which satisfies the following assumption.

\textbf{Assumption 2:} Let $g\colon\mathbb R\times[0,\infty)\to[0,1]$ be a bounded function such that
\[
\mathbb E\left(g\left(\widehat{\gamma}_{\bm{n}}(\bm{j}),c_{\bm{n}}(\bm{j})\right)\right)\to 1,
\]
as $\bm{n}\to\infty$. 

The next theorem states some sufficient conditions for $\widehat{\sigma}_{\bm{n},th}^2$ to be weakly consistent.
\begin{theorem}\label{consistencysigmath}
Under the conditions of Theorem~\ref{consistencyvariance} and Assumption 2 the thresholding estimator $\widehat{\sigma}_{\bm{n},th}^2$ defined as in (\ref{improvedvarestthres}) is a weakly consistent estimator for $\sigma^2$.
\end{theorem}

For the special choice of $g$ as 
\[
g\left(\widehat{\gamma}_{\bm{n}}(\bm{j}),c_{\bm{n}}(\bm{j})\right)=\mathds{1}_{\left\{\left|\widehat{\gamma}_{\bm{n}}(\bm{j})\right|>c_{\bm{n}}(\bm{j})\right\}}
\]
we obtain the so-called cut-off estimator
\begin{align}\label{improvedvarest}
\widehat{\sigma}_{\bm{n},c}^2\coloneqq\sum_{|\bm{j}|\leq \bm{m}}w_{\bm{m}}(\bm{j})\widehat{\gamma}_{\bm{n}}(\bm{j})\mathds{1}_{\left\{\left|\widehat{\gamma}_{\bm{n}}(\bm{j})\right|>c_{\bm{n}}(\bm{j})\right\}},
\end{align}
which belongs to the  subclass of hard-thresholding estimators, where only autocovariances are taken into account whose absolute value exceeds the threshold $c_{\bm{n}}(\bm{j})$. Theorem~\ref{consistencysigmath} then simplifies to the following Corollary.
\begin{corollary}\label{consistencysigmac}
Under the conditions of Theorem~\ref{improvedvarestthres} and if for each $\bm{j}\in\mathbb Z^q$, $c_{\bm{n}}(\bm{j})\to c(\bm{j})$ for $\bm{n}\to\infty$ with $c(\bm{j})<|\gamma(\bm{j})|$, then $\widehat{\sigma}_{\bm{n},c}^2$ defined as in (\ref{improvedvarest}) is a weakly consistent estimator for $\sigma^2$.
\end{corollary}

The extension to multivariate random fields is as follows. Define the matrix-valued threshold estimator
\begin{equation}
\label{improvedvarestthres_general}
\widehat{\sigma}_{\bm{n},th}^2 \coloneqq 
\sum_{|\bm{j}|\leq\bm{m}} w_{\bm{m}}(\bm{j})\widehat{\gamma}_{\bm{n}}(\bm{j}) \circ  g\left(\widehat{\gamma}_{\bm{n}}(\bm{j}),c_{\bm{n}}(\bm{j})\right),
\end{equation}
where for a $ p \times p $ matrix $ \bm A = (a_{\nu\mu} )_{1 \le \nu \le p \atop 1 \le \mu \le p } $ and a real number $ c \ge 0 $ we extend the function $ g $ to the domain $ \R^{p \times p} \times [0, \infty) $ by setting
\[
  g( \bm A, c ) = \biggl( g(a_{\nu\mu}, c ) \biggr)_{1 \le \nu \le p \atop 1 \le \mu \le p}.
\]
In (\ref{improvedvarestthres_general}) the symbol $ \circ $ denotes the Hadamard product, i.e. the element-wise multiplication of two matrices of the same dimension. This means, we threshold all elements of the variance-covariance matrix using the same threshold $c_{\bm{n}}(\bm{j})$ depending on the lag $ \bm j $.

\begin{theorem}\label{consistencysigmath_mult}
	Under the conditions of Theorem~\ref{consistencyvariance} and Assumption 2 the thresholding estimator $\widehat{\sigma}_{\bm{n},th}^2$ for a multivariate random field, defined in (\ref{improvedvarestthres_general}), is a weakly consistent estimator for $\sigma^2$, i.e.
	\[
	  \|  \widehat{\sigma}_{\bm{n},th}^2 - \sigma^2 \|_\infty  \stackrel{\mathbb P}{\to} 0,
	\]
	as $ \bm n \to \infty $.
\end{theorem}

Motivated by our empirical studies presented in the next section, we propose to use rules of the form
\begin{align}\label{rule1}
c_{\bm{n}}(\bm{j})=\frac{\left(\sqrt{j_1^2+ \cdots + j_q^2}\right)^{\alpha}}{n_1 \cdots n_q}-\delta.
\end{align}
with $\alpha\geq 0$ and $\delta=0.0001$. For this rule we have $c_{\bm{n}}(\bm{j})\to -\delta\eqqcolon c(\bm{j})$ for all $\bm{j}\in\mathbb Z^q$ and $\bm{n}\to\infty$, such that the condition $c(\bm{j})<|\gamma(\bm{j})|$ of Corollary~\ref{consistencysigmac} is fulfilled in any case.

Let us consider a special choice of the lag truncation constants $ \bm m $ for the case $ q = 2 $ corresponding to image data, in order to clarify, for a valid choice of the lag truncation constants $ \bm m $, possible values for $ \alpha $ and the relationship of the lag-dependent truncation rule $ c_{\bm n}( \bm j ) $ to a constant-in-lag truncation constant typically used in time series analysis. Hence, let us assume that $n_2=fn_1$ for some known $f>0$, the aspect ratio of the image, and let
\begin{align*}
m_i=c_in_1^{\gamma},\quad i=1,2,
\end{align*}
for constants $c_1, c_2$ and some $0<\gamma<1/3$, such that the sufficient condition (\ref{mnconv}) of Theorem~\ref{consistencyvariance} is satisfied. Clearly, for larger $ \gamma $ a larger number of spatial autocovariances is taken into account. Then for all $|\bm{j}|\leq\bm{m}$ we have
\[
c_{\bm{n}}(\bm{j})\leq\frac{\left(\sqrt{m_1^2+m_2^2}\right)^{\alpha}}{fn_1^2}-\delta=\frac{\left(c_1^2+c_2^2\right)^{\alpha/2}}{f}n_1^{\alpha\gamma-2}-\delta.
\]
If $\alpha=2/\gamma$, then at the boundary ($\bm{j}=\bm{m}$) the constant (in sample size) threshold
\[
c_{th}=\frac{\left(c_1^2+c_2^2\right)^{\alpha}}{f}-\delta
\]
applies, whereas for $\alpha<2/\gamma$ the boundary threshold converges to $-\delta$, as $n_1\to\infty$. For $\gamma$ slightly smaller than $1/3$, in order to take into account a large number of autocovariances, the parameter $\alpha$ can be selected slightly larger than 6. This is in nice agreement with our empirical findings concerning the choice of $\alpha$ when aiming at precise estimation of the asymptotic variance $ \sigma^2 $. However, when the estimated asymptotic variance is used to standardize an image test statistic, our simulations indicate that smaller values of $ \alpha $ (around $3.6$) are preferable, which are admissible for all admissible values of $ \gamma $.

\subsection{Data-adaptive estimation of lag truncation constants and  optimal thresholds}
\label{Sec:DataAdaptive}

The proposed estimators require to select the lag truncation constants $ \bm m $ and, for the cut-off threshold estimator, the thresholds $ c_{\bm n}( \bm j ) $. Let us assume that the thresholds are given by a parametric family of functions with parameter $ \alpha  $, i.e. $ c_{\bm n}( \bm j ) = c_{\bm n}( \bm j; \alpha )  $, e.g. as in (\ref{rule1}). To determine $ \bm m $ and $ \alpha $ we propose to proceed as follows:
To determine the optimal value $ \bm m_{opt} $ for the discrete-valued parameter $ \bm m $ from the set $ \{ k \bm 1 : k \in \N \} $, we  apply a sequential testing procedure which analyzes the sample autocovariances on the unit spheres with respect to the maximum norm and stops when they are no longer statistically different from zero. The statistical tests are based on subsampling in order to determine critical values. This procedure has the nice property that it is consistent for $m$-dependent random fields. For other fields, applying this procedure may serve as a guide, in order to approximate the random field under investigation by an $m$-dependent one, e.g. a spatial moving average model.

A reasonable measure to evaluate $\widehat{\sigma}_{\bm{n},c}^2$  is, of course, the root mean squared error (RMSE), which is a function of $ \bm m $ and $ \alpha $. We propose to plug-in $ \bm m_{opt} $ and then to determine $ \alpha $ by minimizing an estimate of the RMSE obtained by a subsampling procedure, too.

The subsampling-based proposal to deal with the estimators
(\ref{sigmaquadrat}) and (\ref{improvedvarest}) is as follows. 
In accordance with the above notation introduced in Section~\ref{Sec:Subsampling}, $Y_{\bm{j}}$ now stands for the subrandom field $\{\xi_{\bm{t}},\bm{t}\in \bm E_{\bm{j},\bm{b},\bm{h}}\}$, and we write $\widehat{\sigma}^2_{\bm{n},\bm{b},\bm{i}}$ for the subsample value that equals the statistic $\widehat{\sigma}^2_{\bm{b}}$ with constant weights equal to one evaluated at the random field $Y_{\bm{i}}$, i.e.\ $\widehat{\sigma}^2_{\bm{n},\bm{b},\bm{i}}=\widehat{\sigma}^2_{\bm{b}}(Y_{\bm{i}})$ with $w_{\bm{m}}(\bm{j})\equiv 1$. We choose the constant weights here as these do not depend on $\bm{m}$ (and thus not on $\bm{n}$) to avoid additional difficulties when working with the smaller random fields of size $\bm{b}$.

In order to find a subsampling approximation for the RMSE we need to find a reasonable centering term. A natural choice is $\widehat{\sigma}^2_{\bm{n}}$ itself. However, as the summation in (\ref{sigmaquadrat}) still depends on $\bm{m}$, we need to choose $ \bm m $ first. To do so we propose the following sequential testing procedure.
For each subrandom field of size $\bm{b}$ we start by considering the test statistic
\begin{align}\label{rmstat}
R_{\bm b}(\bm{m})\coloneqq\sum_{\bm{j}: |j_k| = m_k}\widehat{\gamma}_{\bm{b}}(\bm{j})
\end{align}
which estimates $ r(\bm m) = \sum_{\bm{j}: |j_k| = m_k} \gamma(\bm{j}) $. We reject $ H_0: r( \bm m ) = 0 $ if the $ 90\% $-confidence interval for $ r( \bm m ) $ obtained by subsampling does not contain $0$. We propose to determine $ \bm m_{opt} $ by the following sequential testing procedure: Starting with  $\bm{m}=\bm{1}$ and proceeding with $ k \bm 1 $, $ k = 1, 2, \dots $, we apply the above test until it rejects the associated null hypothesis $ H_0: r( \bm m' ) = 0 $ for the first time. Then we put $\bm{m}_{opt}=\bm{m}'-\bm{1}$.

The consistency of the proposed subsampling testing procedure follows from Theorem~\ref{Subsampling} and Theorem~\ref{CLT}, which imply that the root
\[
  K_{\bm n}(\mathbb P) =  \mathcal{L} \left(  |\bm n|^{-1/2} ( R_{\bm n}( \bm m ) - r(\bm m) )  \right)
\]
converges weakly to a Gaussian random vector:

\begin{corollary}
\label{CorollarySubs}
		Under the conditions of Theorem~\ref{CLT} and if (\ref{Subs1}) and (\ref{WeakMixingCondSubs}) hold,
		\[
		  K_{\bm n}(\mathbb P) \Rightarrow K(\mathbb P) = \mathcal{L}\left( \sum_{| \bm j | \le \bm m} B_{\bm j} \right),
		\]
		as $ \bm n \to \infty $, where $ ( B_{\bm j} : |\bm j| \le \bm m ) $ is as in Theorem~\ref{CLT}. If $ q(\gamma) $ denotes the $ \gamma $-quantile of $ K(\cdot, \mathbb P) $ and $ q_{\bm n, \bm b}(\gamma) $ the $ \gamma $-quantile of the empirical subsampling distribution for the statistic $ R(\bm m) $, then the asymptotic coverage of the confidence interval \[ [R_{\bm n}(\bm m) - |\bm n |^{-1/2} q_{\bm n, \bm b}(\gamma), R_{\bm n}(\bm m) + |\bm n |^{-1/2} q_{\bm n, \bm b}(\gamma) ] \] is the nominal confidence level $ 2 \gamma $, $ \gamma \in (0,1/2) $.
\end{corollary}

	Once we have determined $\bm{m}_{opt}$, we can evaluate $\widehat{\sigma}^2_{\bm{n}}$ for $\bm{m}_{opt}$ and use $\widehat{\sigma}^2_{\bm{n},\bm{m}_{opt}}$ as centering term for the subsampling approximation for the RMSE given a block size $ \bm b $, whose selection is discussed and studied in Section~\ref{Sec82}. This approach finally leads to the  approximation 
	\begin{align}\label{RMSEsub}
	\widehat{\textnormal{RMSE}}\left(\widehat{\sigma}^2_{\bm{n}}\right)\coloneqq\sqrt{N^{-1}\sum_{i_1=1}^{N_1}\sum_{i_2=1}^{N_2}\ldots\sum_{i_q=1}^{N_q}\left(\widehat{\sigma}^2_{\bm{n},\bm{b},\bm{i}}-\widehat{\sigma}^2_{\bm{n},\bm{m}_{opt}}\right)^2},
	\end{align}
	where $ \widehat{\sigma}^2_{\bm{n},\bm{b},\bm{i}} $, $ \bm i \in \bm 1:\bm N $, are the subsampling replicates and $ \bm N = (N_1, \dots, N_q) $. That approximation is now minimized to determine the tuning parameter $ \alpha $.

\section{Simulation Studies}\label{Simu}
In this section we present simulations about the proposed methods. Sections~\ref{SimVE} to \ref{Sec82} provide a thorough analysis of the variance estimators of Sections~\ref{EstAsyVar} and~\ref{ImpEst}, where for the latter we will focus our attention on the cut-off estimator. The results show that thresholding can improve in terms of the RMSE and, especially, is more robust with respect to the choice of the lag truncation constants $ \bm m $. Then we investigate the proposed data-adaptive procedure to estimate the remaining unknown tuning parameters of these estimators via subsampling. 

Section~\ref{SimsH0Test} provides results about the accuracy of significance tests for image data based on partial sums standardized with the proposed cut-off estimator. The simulations show that highly accurate testing is achieved by our proposal, even in the presence of substantial image correlations.

\subsection{Simulation Results for the Variance Estimators}\label{SimVE}

Having in mind applications in imaging, where quite often spatial dependencies are local, we investigate the statistical behavior of the variance estimators $\widehat{\sigma}^2_{\bm{n}}$ and $\widehat{\sigma}^2_{\bm{n},c}$ for selected stationary two- and three-dimensional random fields, including an autoregressive model to take into account serial correlations present in sequences of images.

Further, we  investigate both estimators for different weighting functions and different values of $\bm{m}$ and $\bm{n}$ and analyse the resulting behaviour of the root mean square error and the bias. Finally, we will also analyse the second estimator dependent on the sequence $c_{\bm{n}}(\bm{j})$ and make a concrete proposal for its selection.

\subsubsection{Spatial Moving Average Models of Order One in Two Dimensions}\label{sub811}
The first model for the random field for which we want to analyse the behaviour of the estimators is a spatial moving average model of order one in two dimensions. For that purpose, take i.i.d.\ innovations $\eta_{i,j}$ with $\eta_{i,j}\sim\mathcal{N}(0,1)$ for all $i$ and $j$ and define
\begin{align*} 
(\textnormal{M1})\,\,\, \varepsilon_{i,j}&\coloneqq a_1\eta_{i-1,j-1}+a_2\eta_{i-1,j}+a_3\eta_{i-1,j+1}+a_4\eta_{i,j-1}\\&+a_5\eta_{i,j}+a_6\eta_{i,j+1}+a_7\eta_{i+1,j-1}+a_8\eta_{i+1,j}+a_9\eta_{i+1,j+1}
\end{align*}
with real weights $a_k,\ k=1,\ldots,9$. We first take $a_5=1$ and $a_k=0.3$ for $k\neq 5$. 
We can now calculate the theoretical asymptotic variance as follows. If we write $\bm{a}=(a_1,\ldots,a_9)^{\prime}$ for the column vector with the weights $a_k, k=1,\ldots,9$, and
\begin{align*}  \bm{\eta}_{i,j}=(\eta_{i-1,j-1},\eta_{i-1,j},\eta_{i-1,j+1},\eta_{i,j-1},\eta_{i,j},\eta_{i,j+1},\eta_{i+1,j-1},\eta_{i+1,j},\eta_{i+1,j+1})^{\prime} 
\end{align*}
for the column vector with the innovation $\eta_{i,j}$ and all the innovations that are located around it, we can rewrite (M1) as scalar product of the vectors $\bm{a}$ and $\bm{\eta}_{i,j}$, i.e.\ as $\varepsilon_{i,j}=\bm{a}^{\prime}\bm{\eta}_{i,j}$. Thus, we obtain

\begin{align}\label{calculationsigmaquadrat}
\sigma^2=\sum_{(h_1,h_2)\in\mathbb Z^2}\mathbb E\left(\varepsilon_{0,0}\varepsilon_{h_1,h_2}\right)
=\sum_{(h_1,h_2)\in\mathbb Z^2}\mathbb E\left((\bm{a}^{\prime}\bm{\eta}_{0,0})(\bm{a}^{\prime}\bm{\eta}_{h_1,h_2})\right)
=\sum_{i,j=1}^9a_ia_j,
\end{align}
as only $9^2=81$ summands in the above summation over $(h_1,h_2)\in\mathbb{Z}^2$ are non-zero, namely one for each combination of $a_i$ and $a_j$, $i,j=1,\ldots,9$. For our concrete vector $\bm{a}$ one can thus calculate that the theoretical variance is $\sigma^2=11.56$. 

Throughout this study we choose the weights $w_{\bm{m}}(\bm{j})$, that are needed for the calculation of $\widehat{\sigma}_{\bm{n}}^2$ in (\ref{sigmaquadrat}), as the product of one-dimensional weights, 
where we take the $w_{m_i}(j_i)$, $1\leq i\leq 2$, either as constant weights (CW) with $w_{m_i}(j_i)=1$, $1\leq i\leq 2$, or as the Quadratic Spectral weight sequence (QS), i.e.\
\begin{align*}
w_{m_i}(j_i)=\frac{25}{12\pi^2\left(j^{(i)}_m\right)^2}\left(\frac{\sin(6\pi j^{(i)}_m/5)}{6\pi j^{(i)}_m/5}-\cos(6\pi j^{(i)}_m/5)\right)
\end{align*}
for $1\leq i\leq 2$ and $j^{(i)}_m=j_i/(m_i+b_w)$, where $b_w$ is a bandwidth parameter that still has to be chosen. In dimension one \cite{Andrews} could show that the QS kernel with a properly chosen bandwidth $b_w$ is the optimal choice with respect to the asymptotic MSE, 
see Theorem~2 in \cite{Andrews}. In his simulation studies, however, the author could also show that in a lot of situations the constant weights lead to more efficient estimates than the quadratic spectral weights. The disadvantage of the constant weights, however, is that in contrast to the QS weights they do not necessarily generate nonnegative variance estimates.

Now, we first consider a random field of size $(30{,}40)$ and investigate our estimator in (\ref{sigmaquadrat}) for several values of $\bm{m}=(m_1,m_2)\in\mathcal M_2$ with
\begin{align}\label{setm2}
\mathcal M_2&=\{(m_1,m_2)\in\{0,\ldots,29\}^2:m_1=m_2\}\\&\cup\{(1,2),(2,3),(3,4),(4,5),(5,6),(6,7),(10,13),(15,20)\}.\notag
\end{align}
Note, that $\mathcal M_2$ mainly consists of pairs with equal components. These pairs are quite natural choices for $\bm m$, as (M1) is a symmetric model where the influence of the innovations around $\eta_{i,j}$ is the same in both directions. Nevertheless, we also consider some pairs where the components differ, including the cases $(10{,}13)$ and $(15{,}20)$ which correspond to a third and a half of the random field of size $(30{,}40)$, respectively.

The first table shows the values of the estimator and its RMSE for selected choices of $\bm{m}\in\mathcal M_2$ for the constant kernel based on 10000 repetitions.
\begin{table}[H]
\centering
\begin{tabular}{crrrrrr}
\hline
 $\bm{m}$ & (0,0) & (1,1) & (2,2) & (3,3) & (4,4) \\
\hline
Mean($\widehat{\sigma}^2_{\bm{n}}$) & 1.7217 & 8.6957 & 11.5924 & 11.5879 & 11.5588 \\
RMSE($\widehat{\sigma}^2_{\bm{n}}$) & 9.8391 & 2.9996 & \textcolor{red}{1.8478} & 2.8328 & 3.8371 \\
\hline
\end{tabular}
\vspace{0.5cm}

\centering
\begin{tabular}{crrrrr}
\hline
 $\bm{m}$ & (5,5) & (6,6) & (7,7) & (8,8) & (9,9) \\
\hline
Mean($\widehat{\sigma}^2_{\bm{n}}$)  & 11.5290 & 11.5233 & 11.5297 & 11.5304 & 11.5165 \\
RMSE($\widehat{\sigma}^2_{\bm{n}}$) & 4.8846 & 5.9867 & 7.1408 & 8.3481 & 9.5736  \\
\hline
\end{tabular}
\caption{Simulation results for (M1) for different values of $\bm{m}$ with constant kernel.}
\label{tab1}
\end{table}

Table~\ref{tab1} shows that the RMSE, seen as a function of $\bm{m}$, is convex with a minimum in $\bm{m}=(2{,}2)$, which corresponds to the largest lag for which the theoretical autocovariance in (M1) is non-zero. We also see that for values of $\bm{m}$ larger than $(2{,}2)$ the RMSE increases quite fast while the bias does not vary much. This implies that the variance increases.

Next, we also examine our estimator for the QS kernel, for which we still have to choose the bandwidth parameter $b_w$ which we choose the same in all dimensions. We want to choose this parameter in such a way that the RMSE gets as small as possible. Figure~\ref{Modell1bw} depicts for different values of $b_w$ the optimal RMSE with respect to $\bm{m}$ (i.e.\ when $\bm{m}$ ranges over the values of $\mathcal M_2$) and the corresponding bias.
\begin{figure}
  \begin{center} 
  \includegraphics[width=60mm]{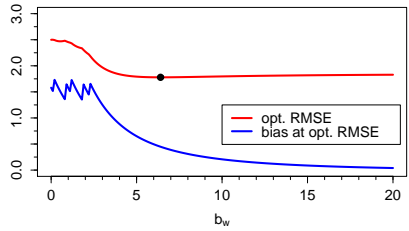}
  \caption{Optimal RMSE with respect to $ \bm m $ and corresponding bias of (M1) for bandwidths $b_w$.}
  \label{Modell1bw}
  \end{center}
\end{figure}

To abbreviate the notation we write $\bm{m}+b_w$ for $(m_1,m_2)+(b_w,b_w)$.
We see that the smallest RMSE is achieved for $b_w=6.4$ which corresponds to $\bm{m}=(2{,}2)$. Thus, we will fix the bandwidth at this value.
Furthermore, we can also see that for $b_w\geq 4$ the RMSE stays nearly fixed such that different choices of $b_w$ near the optimal one would lead to nearly as good estimates as for the optimal bandwidth $b_w=6.4$. The jags in the bias for small choices of the bandwidth parameter arise from the fact, that we allow $\bm{m}$ to range over all the values of $\mathcal M_2$. A jag occurs, when the value of $\bm{m}$, for which the smallest RMSE is attained, changes. For example, for $b_w=0$ the optimal $\bm{m}$ is $(4{,}5)$ whereas for $b_w\geq2.3$ it is $(2{,}2)$, which is again the largest lag for which the theoretical autocovariance in (M1) is non-zero. For $0<b_w\leq 2.3$ the optimal value of $\bm{m}$ varies.

Table~\ref{tab2} shows the values of the estimator and its RMSE in the above scenario of (M1) for selected choices of $\bm{m}$ for the QS kernel with $b_w=6.4$ for 10000 repetitions.
\begin{table}[H]
\centering
\begin{tabular}{crrrrrr}
\hline
 $\bm{m}$ & (0,0) & (1,1) & (2,2) & (3,3) & (4,4) \\
\hline
Mean($\widehat{\sigma}^2_{\bm{n}}$) & 1.7217 & 8.4380 & 11.1120 & 11.2030 & 11.2502 \\
RMSE($\widehat{\sigma}^2_{\bm{n}}$) & 9.8391 & 3.2388 & \textcolor{red}{1.7793} & 2.5630 & 3.3187 \\
\hline
\end{tabular}
\vspace{0.5cm}

\centering
\begin{tabular}{crrrrr}
\hline
 $\bm{m}$ & (5,5) & (6,6) & (7,7) & (8,8) & (9,9) \\
\hline
Mean($\widehat{\sigma}^2_{\bm{n}}$)  & 11.2799 & 11.3132 & 11.3471 & 11.3741 & 11.3899 \\
RMSE($\widehat{\sigma}^2_{\bm{n}}$) & 4.0534 & 4.7795 & 5.4992 & 6.2172 & 6.9214 \\
\hline
\end{tabular}
\caption{Simulation results for (M1) for different values of $\bm{m}$ with QS kernel and bandwidth $\bm{m}+6.4$.}
\label{tab2}
\end{table}

The smallest RMSE is now $1.7793$ and again achieved for $\bm{m}=(2{,}2)$. Similar as in Table~\ref{tab1} the RMSE increases for values of $\bm{m}$ larger than $(2{,}2)$, but not as rapid as in Table~\ref{tab1}. Thus, we see that in this scenario the QS kernel with a reasonably chosen bandwidth is nearly $4\%$ more efficient than the constant kernel which is in accordance with the simulation results in \cite{Andrews}.

As this simulation study shows, the smallest RMSE is still quite high and what is even more problematic, the RMSE depends quite heavily on the proper choice of $\bm{m}$. As already explained in the motivation for the improved variance estimator, this problem often appears when a lot of `zeros' have to be estimated. This is the case here, as the dependencies of the random field are only weak, as the autocovariances vanish for $|\bm{j}|>2$. The estimation of many such null entries worsens the RMSE a lot.
Thus we shall now investigate the improved variance estimator $\widehat{\sigma}_{\bm{n},c}^2$ in (\ref{improvedvarest}) which is designed to circumvent this issue. For that we need to choose an appropriate cutting rule $c_{\bm{n}}(\bm{j})$. 

We focus on the rule $ c_{\bm n}( \bm j) $ introduced in the previous section, which is parameterized by $ \alpha > 0 $.
The parameter $\alpha$ is a further tuning parameter that we want to choose in such a way that the RMSE of $\widehat{\sigma}_{\bm{n},c}^2$ attains the smallest possible value when $\bm{m}$ ranges again over the values of $\mathcal M_2$. 

In the following, we investigate the same situation as before, the spatial moving average model (M1) with weights $a_5=1$ and $a_i=0.3$ for $i\neq5$. 
Figure \ref{Modell1CWQSalpha} shows for both weighting schemes the curves of the optimal RMSE and the corresponding bias as a function of $\alpha$.
\begin{figure}[h]
  \begin{center}
   \subfigure[Constant kernel.]{\includegraphics[width=60mm]{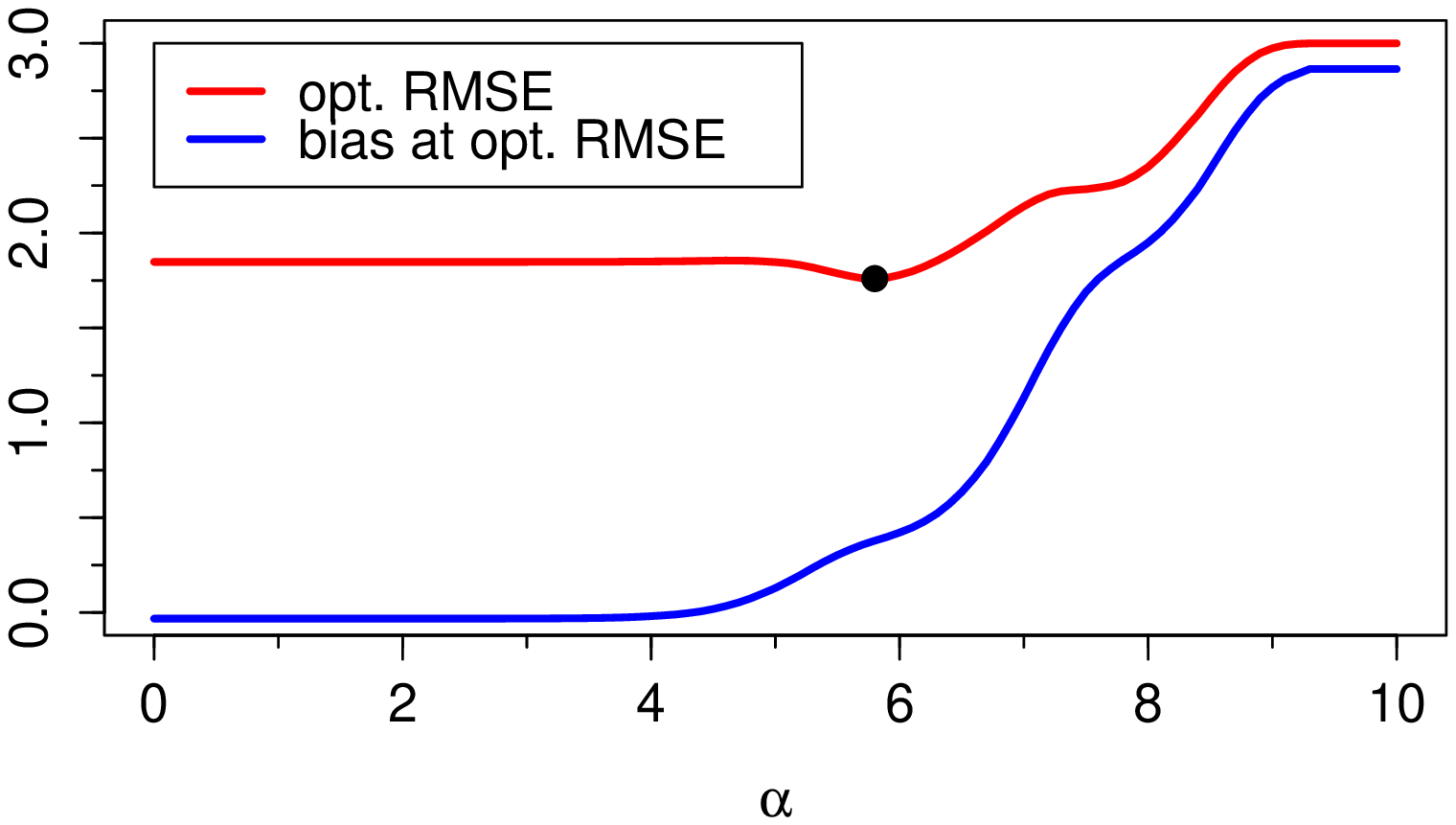}}\qquad
   \subfigure[QS kernel.]{\includegraphics[width=60mm]{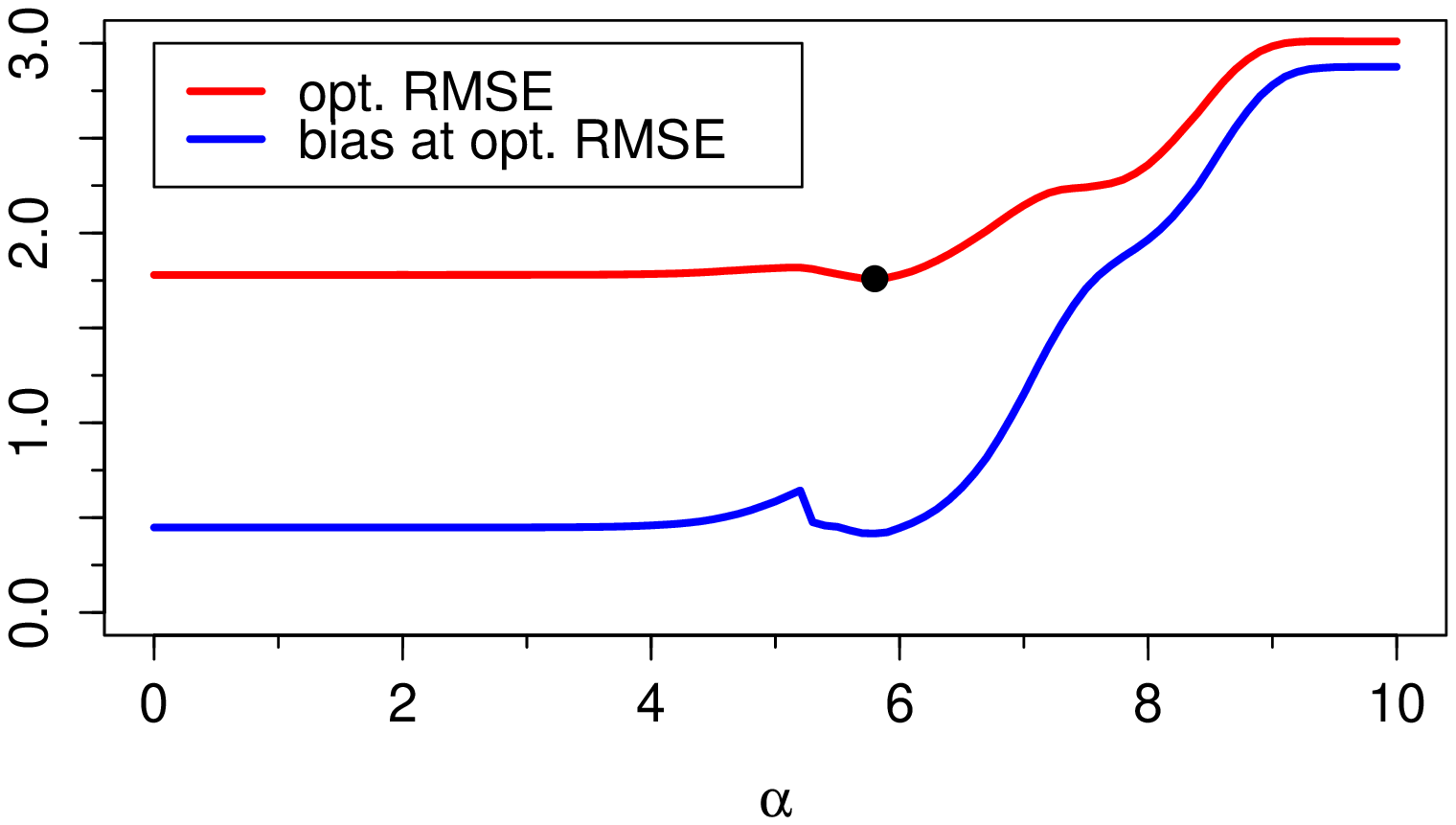}}
  \caption{Optimal RMSE and corresponding bias of (M1).}
  \label{Modell1CWQSalpha}
  \end{center}
\end{figure}

We can see that the curves of the RMSE and the bias are very similar to each other. If $\alpha$ is too small the thresholds $c_{\bm{n}}(\bm{j})$ do not lead to an improvement of the optimal RMSE. However, for $\alpha\in[5.0,6.2]$ and constant weights, and for $\alpha\in[5.6,6.0]$ and QS weights, one has slight improvements of the optimized RMSE. Here the optimal choice for $\alpha$ is $\alpha=5.8$ for both kernels. If $\alpha>6.2$ and $\alpha>6.0$ respectively the optimized RMSE gets worse than without cutting. The jags in the bias arise again due to the fact that we do not fix $\bm{m}$ but allow it to vary over $\mathcal M_2$. A jag occurs when the value of $\bm{m}$ for which the optimal RMSE is attained changes.

We can, however, see more easily the improvement of the estimator $\widehat{\sigma}_{\bm{n},c}^2$ upon $\widehat{\sigma}_{\bm{n}}^2$ when we inspect tables similar to \ref{tab1} and \ref{tab2}, but with $\widehat{\sigma}_{\bm{n}}^2$ replaced by $\widehat{\sigma}_{\bm{n},c}^2$ and a cutting rule of the form (\ref{rule1}) with $\alpha=5.8$. Tables \ref{tab9} and \ref{tab10} show, that the RMSE for both kernels levels off at approximately $1.76$, which is in both cases an improvement of the RMSE of the estimator without cutting rule, though this improvement is higher for the estimator with constant weights. This shows that the proper choice of $\bm{m}$ is now of minor importance than before, since now each choice of $\bm{m}$, that is larger or equal than the largest lag $\bm{j}$ for which the theoretical autocovariances are non-zero, leads to nearly the same RMSE. If we had chosen a different $\alpha$ than the optimal one, but one close to it, we would not have improved the optimal RMSE, but nevertheless the effect for $\bm{m}$ larger than $(2{,}2)$ would have been the same.
\begin{table}[!ht]
\centering
\begin{tabular}{crrrrrr}
\hline
 $\bm{m}$ & (0,0) & (1,1) & (2,2) & (3,3) & (4,4) \\
\hline
Mean($\widehat{\sigma}^2_{\bm{n},c}$) & 1.7217 & 8.6957 & 11.1815 & 11.1815 & 11.1815 \\
RMSE($\widehat{\sigma}^2_{\bm{n},c}$) & 9.8391 & 2.9996 & 1.7597 & 1.7597 & 1.7597  \\
\hline
\end{tabular}
\caption{Simulation results for (M1) for different values of $\bm{m}$ with constant weights, $\alpha=5.8$.}
\label{tab9}
\end{table}
\begin{table}[!ht]
\centering
\begin{tabular}{crrrrrr}
\hline
 $\bm{m}$ & (0,0) & (1,1) & (2,2) & (3,3) & (4,4) & (29,29) \\
\hline
Mean($\widehat{\sigma}^2_{\bm{n},c}$) & 1.7217 & 8.4380 & 10.7599 & 10.8431 & 10.9041 & 11.1572 \\
RMSE($\widehat{\sigma}^2_{\bm{n},c}$) & 9.8391 & 3.2388 & 1.8038 & 1.7867 & 1.7768 & 1.7594  \\
\hline
\end{tabular}
\caption{Simulation results for (M1) for different values of $\bm{m}$ with QS kernel, $\alpha=5.8$.}
\label{tab10}
\end{table}

We can observe a further interesting fact by inspecting Table~\ref{tab11}. We have already obtained the results of the last two columns of Table~\ref{tab11} before. What is new now are the results of the first column. If we take $b_w=0$ and thus have $j_m^{(i)}=j_i/m_i, i=1,2,$ as argument in the QS kernel, the smallest RMSE, that is attained for a random field of size $(30{,}40)$ in model (M1), is $2.4996$ for $\bm{m}=(4{,}5)$. If we now optimize the cutting rule (\ref{rule1}) for this situation as above, we obtain again $\alpha=5.8$ as the optimal value for $\alpha$. This leads to a smallest possible RMSE of $1.7593$ for $\bm{m}=(28{,}28)$. Thus we see that the estimator $\widehat{\sigma}_{\bm{n},c}^2$ with $\alpha=5.8$ and a QS kernel with bandwidth $b_w=0$ leads to a smaller RMSE than the estimator $\widehat{\sigma}_{\bm{n}}^2$ with a QS kernel and bandwidth $b_w=6.4$ or than the estimator $\widehat{\sigma}_{\bm{n}}^2$ with constant weights and to an equally good RMSE if we further use the cutting rule. Similar observations also hold true for larger sample sizes.
\begin{table}[H]
\centering
\begin{tabular}{crrr}
\hline
 kernel & QS, $b_w=0$ & QS, $b_w=6.4$ & CW \\
\hline
RMSE($\widehat{\sigma}^2_{\bm{n}}$)  & 2.4996 & 1.7793 & 1.8478 \\
bias($\widehat{\sigma}^2_{\bm{n}}$)  & -1.5802 & -0.4480  & 0.0324 \\
$\bm{m}$                  & (4,5)  & (2,2)   & (2,2) \\
\hline
RMSE($\widehat{\sigma}^2_{\bm{n},c}$) & 1.7593 & 1.7593 & 1.7597 \\
bias($\widehat{\sigma}^2_{\bm{n},c}$) & -0.4173 & -0.4162  & -0.3785 \\
$\bm{m}$                 & (28,28)  & (22,22)   & (2,2) \\
\hline
impr.(\%)& 29.62 & 1.12 & 4.76 \\\hline
\end{tabular}
\caption{Simulation results of both estimators for (M1) for both kernels and different bandwidths and $\alpha=5.8$.}
\label{tab11}
\end{table}

Finally, we also want to investigate the behaviour of both estimators for both kernels for different sample sizes when we choose different weights for $a_k$. So, in the following let $a_5=1$ and $a_k=a$ for $k\neq 5$. We again fix $\alpha=5.8$ and $b_w=6.4$.
This time we do not only calculate the optimal RMSE (denoted by $\textnormal{RMSE}_{opt}$), the corresponding bias and the value of $\bm{m}$ for which $\textnormal{RMSE}_{opt}$ is attained, but also the quotient $\textnormal{RMSE}_{29}/\textnormal{RMSE}_{opt}$, where $\textnormal{RMSE}_{29}$ is the RMSE of the estimator for $\bm{m}=(29{,}29)$. This is to get an impression of how much the RMSE can vary when we choose $\bm{m}$ too large in the case that we do not use a cutting rule.
Table \ref{tab5} shows the corresponding simulation results.
\begin{table}[!ht]
\centering
\begin{tabular}{crrrrrrr}
\hline
weights $a$                      & 0.01     & 0.1      & 0.3     & 0.5     & 0.7     & 0.9     \\
\hline
$\widehat{\sigma}^2_{\bm{n},c}$                       & 1.1664   & 3.24     & 11.56   & 25      & 43.56   & 67.24   \\\hline
$\textnormal{RMSE}_{opt}$ (CW)                            & 0.1367 & 0.4272 & 1.8478 & 3.9033 & 6.7235 & 10.3083 \\
bias (CW)                             & -0.0015 & -0.3154 & 0.0324 & 0.0681 & 0.1171 & 0.1792 \\
$\bm{m}$ (CW)                              & (1,1)    & (1,1)    &(2,2)    &(2,2)    &(2,2)    &(2,2)   \\
$\textnormal{RMSE}_{29}/\textnormal{RMSE}_{opt}$  & 54.5744 & 42.6213 & 33.1515 & 33.4509 & 33.6324 & 33.7522 \\\hline
$\textnormal{RMSE}_{opt}$ (CW, cut)                      & 0.1367 & 0.4272 & 1.7597 & 3.9650 & 6.8246 & 10.3694 \\
bias (CW, cut)                       & -0.0015 & -0.3154 & -0.3785 & -0.3899 & -0.0901 & 0.0933 \\
$\bm{m}$ (CW, cut)                        & (1,1)    & (1,1)    &(2,2)    &(2,2)    &(2,2)    &(2,2)   \\
$\textnormal{RMSE}_{29}/\textnormal{RMSE}_{opt}$  & 1.1185 & 1.0494 & 1.0000 & 1.0127 & 1.1006 & 1.2141 \\
\hline
$\textnormal{RMSE}_{opt}$ (QS)                             & 0.1324 & 0.4543 & 1.7793 & 3.8008 & 6.5890 & 10.1438 \\
bias (QS)                             & -0.0076 & -0.2483 & -0.4480 & -1.1059 & -2.0494 & -3.2787 \\
$\bm{m}$ (QS)                              & (1,1)    & (1,2)    &(2,2)    &(2,2)    &(2,2)    &(2,2)   \\
$\textnormal{RMSE}_{29}/\textnormal{RMSE}_{opt}$  & 21.6538 & 16.5137 & 14.6394 & 14.7177 & 14.7478 & 14.7626 \\\hline
 $\textnormal{RMSE}_{opt}$ (QS, cut)                      & 0.1324 & 0.4490 & 1.7593 & 3.9650 & 6.7261 & 10.2171 \\
bias (QS, cut)                       & -0.0078 & -0.1991 & -0.4162 & -1.4951 & -2.2255 & -3.3517 \\
$\bm{m}$ (QS, cut)                        & (1,1)    & (29,29)    &(22,22)    &(2,2)    &(2,2)    &(2,2)   \\
$\textnormal{RMSE}_{29}/\textnormal{RMSE}_{opt}$ & 1.1525 & 1.0000 & 1.0000 & 1.0102 & 1.1115 & 1.2253 \\
\hline
\end{tabular}
\caption{Simulation results for (M1) for constant / QS kernel, $\bm{n}=(30{,}40)$, $\alpha=5.8$, and different weights $a$.}
\label{tab5}
\end{table} 

The table shows that both kernels compete reasonably well with each other. The estimator $\widehat{\sigma}_{\bm{n}}^2$ with a QS kernel, for example, ranges from being $6\%$ less efficient to $4\%$ more efficient than the estimator $\widehat{\sigma}_{\bm{n}}^2$ with a constant kernel in the case of a random field of size $(30{,}40)$ with different weights $a$. Similar observations hold true for the other cases. 
In some scenarios the optimized RMSE of $\widehat{\sigma}_{\bm{n},c}^2$ is worse than the one for $\widehat{\sigma}_{\bm{n}}^2$ as we do not vary the value of $\alpha$, but keep it fix as $\alpha=5.8$ which was the value of $\alpha$ adapted to the situation $\bm{n}=(30{,}40)$ with $a_k=0.3$ for $k\neq 5$.
However, the main advantage of the estimator $\widehat{\sigma}_{\bm{n},c}^2$ has again to be seen in the fact that it stabilizes the RMSE for values of $\bm{m}$ larger than $(2{,}2)$. 
This can be seen by the fact that for $\widehat{\sigma}_{\bm{n},c}^2$ the quotient $\textnormal{RMSE}_{29}/\textnormal{RMSE}_{opt}$ is close to one in all cases. 
For example, for $\bm{n}=(30{,}40)$, $a_k=0.7$ for $k\neq 5$, and the constant kernel the optimized RMSE worsens from $6.7235$ for $\widehat{\sigma}_{\bm{n}}^2$ to $6.8246$ for $\widehat{\sigma}_{\bm{n},c}^2$, while the quotient $\textnormal{RMSE}_{29}/\textnormal{RMSE}_{opt}$ improves from $33.6324$ to $1.1006$. So regarding the RMSE we see that for $\widehat{\sigma}_{\bm{n},c}^2$ it does not make a big difference if we choose $\bm{m}$ as $(2{,}2)$ or $(29{,}29)$ or as some value in between, whereas for $\widehat{\sigma}_{\bm{n}}^2$ it does. The same is also true in the case that we use a QS kernel, with the only difference that for $\widehat{\sigma}_{\bm{n}}^2$ the quotient $\textnormal{RMSE}_{29}/\textnormal{RMSE}_{opt}$ equals $14.7478$ which is a lot smaller than $33.6324$ in the case of a constant kernel.
This observation also holds true for all other scenarios, i.e.\ for $\widehat{\sigma}_{\bm{n}}^2$ this quotient is larger if we use the constant kernel instead of the QS kernel.
Naturally, we could further improve the simulation results of the estimators with QS weights if we chose the bandwidth $b_w$ differently in each scenario. Similarly, we could also improve the simulation results of $\widehat{\sigma}_{\bm{n},c}^2$ with an $\alpha$ chosen differently in each scenario.

\subsubsection{Spatial Moving Average Models of Order Three in Two Dimensions}
\label{Sec:ModelM2}
We also investigate the behaviour of the estimators for stronger dependence structures such as a spatial moving average model of order three. The model, (M2), used here is similar to (M1), but now with weights 1, $a_1, a_2$ and $a_3$ for the center pixel and the pixels in the first, second and third ring around it instead of only giving weights to the center pixel and the pixels in the first ring around it. If $H_{i,j}$ is the  $(7\times 7)$-matrix with the column by column entries $\eta_{i-3,j-3},\eta_{i-3,j-2},\eta_{i-3,j},\ldots,\eta_{i,j},\ldots,\eta_{i+3,j+1},\eta_{i+3,j+2},\eta_{i+3,j+3}$
and write $vec(\cdot)$ for the vectorization of a matrix, which transforms a matrix into a vector by writing the columns on top of one another, we can formulate model (M2) as
\begin{align*}
(\textnormal{M2})\quad
\varepsilon_{i,j}=(vec(W))^{\prime}vec\left(H_{i,j}\right).
\end{align*}
Our findings are in agreement with model (M1), in particular the threshold estimator again stabilizes the RMSE for large values of $ \bm m $, and are therefore deferred to the supplement.

\subsubsection{Autoregressive Spatial Moving Average Mixture Models in Three Dimensions}
The last model that we consider is a mixture between an autoregressive model of order one in the time domain and a moving average model of order two in the spatial domain. More specifically we put for $\rho\in(-1,1)$
\begin{align*}
(\textnormal{M4})\quad \varepsilon_{t,i,j}=X_{t,i,j}+v_{t,i,j},\qquad
X_{t,i,j}=\rho X_{t-1,i,j}+u_{t,i,j}
\end{align*}
where 
$u_{t,i,j}\stackrel{i.i.d.}\sim\mathcal N(0,1)$ for all $t,i,j\in\mathbb{Z}$ and the $v_{t,i,j}$ follow for each fixed $t$ model (M1) and are uncorrelated for different values of $t$. Moreover, we suppose that $u_{t_1,i_1,j_1}$ and $v_{t_2,i_2,j_2}$ are uncorrelated for all $t_1,t_2,i_1,i_2,j_1,j_2\in\mathbb{Z}$.

We now calculate the asymptotic variance as follows. Note, that by the one-dimensional theory for AR(1)-models in the time series literature we get for the $X_{t,i,j}$ the representation as a linear process by
\begin{align*}
X_{t,i,j}=\sum_{k=0}^{\infty}\rho^k u_{t-k,i,j}.
\end{align*}
For $(h_t,h_i,h_j)\in\mathbb{Z}^3$ we then obtain
\begin{align*}
\mathbb E\left(\varepsilon_{t,i,j}\varepsilon_{t+h_t,i+h_i,j+h_j}\right)
&=\sum_{k=0}^{\infty}\sum_{l=0}^{\infty}\rho^{k+l}\mathbb E\left(u_{t-k,i,j}u_{t+h_t-l,i+h_i,j+h_j}\right)\\&
+\mathbb E\left(v_{t,i,j}v_{t+h_t,i+h_i,j+h_j}\right)
+\sum_{k=0}^{\infty}\rho^k \mathbb E\left(u_{t-k,i,j}v_{t+h_t,i+h_i,j+h_j}\right)\\&
+\sum_{l=0}^{\infty}\rho^l \mathbb E\left(u_{t+h_t-l,i+h_i,j+h_j}v_{t,i,j}\right)
\eqqcolon E_1+E_2+E_3+E_4.
\end{align*}
$E_3$ and $E_4$ equal zero as $u_{t_1,i_1,j_1}$ and $v_{t_2,i_2,j_2}$ are uncorrelated by assumption.
$E_1$ is only non-zero if $h_i=h_j=0$. In this case we get
\[
E_1=\sum_{k=0}^{\infty}\rho^{2k+h_t}=\frac{\rho^{h_t}}{1-\rho^2}.
\]
$E_2$, on the contrary, is only non-zero if $h_t=0$ and it then equals the autocovariances of model (M1).
We thus obtain
\begin{align*}
\sigma^2=\sum_{(h_t,h_i,h_j)\in\mathbb{Z}^3}\mathbb E\left(\varepsilon_{t,i,j}\varepsilon_{t+h_t,i+h_i,j+h_j}\right)
=\frac{1}{(1-\rho)^2}+\sum_{i,j=1}^9a_ia_j,
\end{align*}
which corresponds to the sum of the variance caused by the AR(1)-model in the time domain and the asymptotic variance of (M1).
In our first simulation setting for (M4) we take $\rho=0.2, a_5=1$ and $a_i=0.3$ for $i\neq 5$ and consequently have $\sigma^2=1.5625+11.56=13.1225$.
Now, we first consider a random field of size $(20{,}30{,}40)$ where the first component corresponds to the time domain, and investigate the estimator in (\ref{sigmaquadrat}) for several values of $\bm{m}=(m_1,m_2,m_3)\in\mathcal M_3\cup\{(1,2,2)\}$ with $\mathcal M_3$ defined as 
\begin{align}\label{setm3}
\mathcal M_3&=\{(m_1,m_2,m_3)\in\{0,\ldots,15\}^3:m_1=m_2=m_3\}.
\end{align}
Note that this time we include the triple $(1{,}2{,}2)$, as this is the largest lag with a significant contribution to the asymptotic variance of model (M4).

Table~\ref{tab24} shows the corresponding simulation results for the constant kernel using 10000 repetitions.
\begin{table}[!ht]
\centering
\begin{tabular}{crrrrrrrrrrrrrrrr}
\hline
$\bm{m}$ & (0,0,0) & (1,1,1) & (1,2,2)  & (2,2,2) & (3,3,3) & (4,4,4) \\
\hline
Mean($\widehat{\sigma}^2_{\bm{n}}$) & 2.7079 & 10.0659 & 12.9475 & 13.0239 & 13.0259 & 12.9901 \\
RMSE($\widehat{\sigma}^2_{\bm{n}}$) & 10.4146 & 3.0834 & \textcolor{red}{0.8432} & 1.0995 & 2.0106 & 3.1359 \\
\hline
\end{tabular}
\vspace{0.5cm}

\centering
\begin{tabular}{crrrrrrrrrrrrrrrr}
\hline
$\bm{m}$ & (5,5,5) & (8,8,8) & (10,10,10) & (12,12,12) & (15,15,15) \\
\hline
Mean($\widehat{\sigma}^2_{\bm{n}}$) & 12.9906 & 12.9402 & 12.8049 & 12.7995 & 12.7426 \\
RMSE($\widehat{\sigma}^2_{\bm{n}}$) & 4.4694 & 9.9540 & 15.0064 & 21.3540 & 34.7365 \\
\hline
\end{tabular}
\caption{Simulation results for (M4) for different values of $\bm{m}$ with constant kernel.}
\label{tab24}
\end{table}

We can see that the smallest RMSE is attained for $\bm{m}=(1{,}2{,}2)$ which corresponds in the spatial domain to the largest lag with non-zero autocovariances and in the time domain to the largest lag with a significant contribution to the asymptotic variance as the autoregressive model is of order one. Greater values of $\bm{m}$ lead to a rapid increase of the RMSE while the bias stays stable.

In the case of QS weights we first have to determine the bandwidth parameter $b_w$. This time we allowed $b_w$ to vary from $0$ to $40$, but for these values no minimum is attained (except on the boundary for $b_w=40$). On the contrary, the RMSE decreases slowly when $b_w$ increases. As, however, for $b_w\geq 20$ the RMSE decreases so slowly that visually it is constant, we take $b_w=20$ from now on. Note that in the previous three simulation models a minimal RMSE with respect to $b_w$ was always attained, but in all these cases the RMSE was also nearly constant around the minimum. Table~\ref{tab25-1} shows the corresponding simulation results.
\begin{table}[!ht]
\centering
\begin{tabular}{crrrrrrrrrrrrrrrrr}
\hline
$\bm{m}$ & (0,0,0) & (1,1,1) & (1,2,2) & (2,2,2) & (3,3,3) & (4,4,4) \\
\hline
Mean($\widehat{\sigma}^2_{\bm{n}}$) & 2.7079 & 10.0322 & 12.8747 & 12.9503 & 12.9588 & 12.9309  \\
RMSE($\widehat{\sigma}^2_{\bm{n}}$) & 10.4146 & 3.1165 & \textcolor{red}{0.8515} & 1.0910 & 1.9561 & 2.9977  \\
\hline
\end{tabular}
\vspace{0.5cm}

\centering
\begin{tabular}{crrrrrrrrrrrrrrrrr}
\hline
$\bm{m}$ & (5,5,5) & (8,8,8) & (10,10,10) & (12,12,12) & (15,15,15)  \\
\hline
Mean($\widehat{\sigma}^2_{\bm{n}}$) & 12.9362 & 12.9115 & 12.8108 & 12.8094 & 12.7996  \\
RMSE($\widehat{\sigma}^2_{\bm{n}}$) & 4.1905 & 8.7241 & 12.5226 & 16.9393 & 25.3493  \\
\hline
\end{tabular}
\caption{Simulation results for (M4) for different values of $\bm{m}$ with QS kernel.}
\label{tab25-1}
\end{table}

We see that also for the QS kernel the smallest RMSE is achieved for $\bm{m}=(1{,}2{,}2)$ which is now slightly larger than for the constant kernel. However, if we chose a larger value for $b_w$ we could probably get a smaller RMSE as in the previous simulation settings.

Next, we want to investigate how the improved variance estimator $\widehat{\sigma}^2_{\bm{n},c}$ with the cutting rule 
\begin{align}\label{rule3}
	c_{\bm{n}}(\bm{j})=\frac{\left(\sqrt{j_1^2+j_2^2+j_3^2}\right)^{\alpha}}{n_1n_2n_3}-\delta
	\end{align}
	with $\alpha\geq 0$ and $\delta=0.0001$,
behaves in this model.
If we optimize again the RMSE with respect to $\alpha$, we obtain for the constant as well as for the QS weights an optimal value of $\alpha$ as $\alpha=9.4$. Due to very high computational costs we allowed $\bm{m}$ only to vary up to values of $\mathcal M_3$ smaller than $\bm{m}=(7{,}7{,}7)$ which does not effect the results when using constant weights, but does not lead to the smallest possible RMSE when using QS weights. The corresponding results are gathered in Tables \ref{tab26} and~\ref{tab27}.
\begin{table}[!ht]
\centering
\begin{tabular}{crrrrrrrrrrrrrrrrr}
\hline
$\bm{m}$ & (0,0,0) & (1,1,1) & (1,2,2) & (2,2,2) & (3,3,3)  \\
\hline
Mean($\widehat{\sigma}^2_{\bm{n},c}$) & 2.7079 & 10.0659 & 12.5863 & 12.6569 & 12.6569 \\
RMSE($\widehat{\sigma}^2_{\bm{n},c}$) & 10.4146 & 3.0832 & 0.7600 & 0.7169 & 0.7169 \\
\hline
\end{tabular}
\caption{Simulation results for (M4) for different values of $\bm{m}$ with constant kernel, $\alpha=9.4$.}
\label{tab26}
\end{table}
\begin{table}[!ht]
\centering
\begin{tabular}{crrrrrrrrrrrrrrrrr}
\hline
$\bm{m}$ & (0,0,0) & (1,1,1) & (1,2,2) & (2,2,2) & (3,3,3) & (15,15,15) \\
\hline
Mean($\widehat{\sigma}^2_{\bm{n},c}$) & 2.7079 & 10.0322 & 12.5219 & 12.5917 & 12.5972 & 12.6310 \\
RMSE($\widehat{\sigma}^2_{\bm{n},c}$) & 10.4146 & 3.1163 & 0.8040 & 0.7578 & 0.7542 & 0.7327 \\
\hline
\end{tabular}
\caption{Simulation results for (M4) for different values of $\bm{m}$ with QS kernel, $\alpha=9.4$.}
\label{tab27}
\end{table}

We see that for both weighting schemes we again have the stabilization property of the RMSE for values of $\bm{m}$ larger than $\bm{m}=(2{,}2{,}2)$. For the constant weights the RMSE levels off at approximately 0.71 and for the QS weights between 0.73 and 0.76. But what is even more appealing now is that we can actually improve the optimized RMSE by almost $15\%$ in the case of constant weights and by nearly $14\%$ in the case of QS weights. The reason for this is that in this model the autocovariances for $\bm{m}$ larger than $(1{,}2{,}2)$ are close to zero, but not identically zero as in the previous simulation settings.

Finally, we want to investigate if we still have an improvement of the optimized RMSE if we take the cutting rule with the optimal $\alpha=9.4$ in cases where we change the weights of the autoregressive parameter $\rho$ in model (M4). All other parameters are chosen as before.

Table \ref{tab28} shows that this is indeed the case. For $\rho=0.4$, for example, we have an improvement of around $17\%$ for both kernels. If we optimized $\alpha$ to this situation we would probably obtain an even higher improvement. Also the other values for $\rho$ lead to an improvement of the optimized RMSE of nearly or even more than $10\%$.
\begin{table}[!ht]
\centering
\begin{tabular}{crrrrrrr}
\hline
weights $\rho$                      & 0.01     & 0.1      & 0.2     & 0.3     & 0.4     & 0.5     \\
\hline
$\sigma^2$                       & 12.5803   & 12.7946     & 13.1225   & 13.6008      & 14.3378   & 15.56   \\\hline
$\textnormal{RMSE}_{opt}$ (CW)    & 0.8019 & 0.8143 & 0.8432 & 0.9233 & 1.1546 & 1.5319 \\
bias (CW)                             & -0.0449 & -0.0779 & -0.1750 & -0.3749 & -0.7595 & -0.8961 \\
$\bm{m}$ (CW)                              & (1,2,2)    & (1,2,2)    &(1,2,2)    &(1,2,2)    &(1,2,2)    &(2,2,2)   \\\hline
$\textnormal{RMSE}_{opt}$ (CW, cut)                      & 0.6625 & 0.6821 & 0.7169 & 0.7837 & 0.9573 & 1.3909 \\
bias (CW, cut)                       & -0.4065 & -0.4254 & -0.4656 & -0.5506 & -0.7644 & -1.2529 \\
$\bm{m}$ (CW, cut)                        & (1,2,2)    & (2,2,2)    &(2,2,2)    &(2,2,2)    &(2,2,2)    &(2,2,2)   \\\hline
impr. (\%)    & 17.38     & 16.24    & 14.98   & 15.12   & 17.09   & 9.21   \\
\hline
$\textnormal{RMSE}_{opt}$ (QS)    & 0.7993 & 0.8145 & 0.8515 & 0.9463 & 1.1971 & 1.5658 \\
bias (QS)                             & -0.1166 & -0.1501 & -0.2478 & -0.4485 & -0.8339 & -0.9784 \\
$\bm{m}$ (QS)                              & (1,2,2)    & (1,2,2)    &(1,2,2)    &(1,2,2)    &(1,2,2)    &(2,2,2)   \\\hline
$\textnormal{RMSE}_{opt}$ (QS, cut)                      & 0.6894 & 0.7073 & 0.7437 & 0.8137 & 0.9931 & 1.4341 \\
bias (QS, cut)                       & -0.4489 & -0.4679 & -0.5090 & -0.5953 & -0.8109 & -1.3021 \\
$\bm{m}$ (QS, cut)                        & (7,7,7)    & (7,7,7)    &(7,7,7)    &(7,7,7)    &(7,7,7)    &(7,7,7)   \\\hline
impr. (\%)    & 13.75     & 13.16    & 12.66   & 14.01   & 17.04   & 8.41   \\
\hline
\end{tabular}
\caption{Simulation results for (M4) for constant / QS kernel, $\alpha=9.4$ and different weights $\rho$.}
\label{tab28}
\end{table}

\subsection{Data-adaptive variance estimation using subsampling}\label{Sec82}
	In the previous section we have performed an extensive simulation study about the behaviour of the RMSE of the variance estimators (\ref{sigmaquadrat}) and (\ref{improvedvarest}). We have seen that dependent on the dependence structure of the underlying random fields it is quite useful to consider the improved variance estimator in (\ref{improvedvarest}) instead of the one in (\ref{sigmaquadrat}), as the former is very robust with respect to the proper choice of $\bm{m}$ and can also lead to improvements of the optimized RMSE up to more than $15\%$ in some situations.

In the following simulation study we want to check how well the data-adaptive subsampling procedure proposed in Section~\ref{Sec:DataAdaptive}  works. For this puropose, we consider model (M1) and choose $\bm{b}=\left\lfloor \bm{n}^{\gamma}\right\rfloor=(\left\lfloor n_1^{\gamma}\right\rfloor,\left\lfloor n_2^{\gamma}\right\rfloor)$ with $\gamma\in\{0.7,0.8,0.9\}$ as well as $\bm{h}=\bm{1}$, so that we consider all possible subrandom fields of size $\bm{b}$. Moreover, we do not only investigate the subsampling approximation of the RMSE, but also the one for the mean which is defined as
\begin{align}\label{Meansub}
\widehat{\textnormal{Mean}}\left(\widehat{\sigma}^2_{\bm{n}}\right)\coloneqq |\bm N|^{-1}\sum_{i_1=1}^{N_1}\sum_{i_2=1}^{N_2}\ldots\sum_{i_q=1}^{N_q}\widehat{\sigma}^2_{\bm{n},\bm{b},\bm{i}}.
\end{align}
The general simulation setting of model (M1) stays the same as in Subsection~\ref{sub811}; in particular we consider a random field of size $(30{,}40)$ leading to subrandom fields of size $(10{,}13)$, $(15{,}19)$ and $(21{,}27)$ for $\gamma=0.7, 0.8$ and $0.9$, respectively. Table~\ref{M1sub} shows the simulation results for different values of $\bm{m}$ and 10000 repetitions. We can see that the subsampling approximation of the mean works very well for all values of $\gamma$ and $\bm{m}$. Regarding the subsampling approximation of the RMSE, we see that this is most adequate if we choose $\gamma=0.9$. Therefore, we fix this value from now on.

As our main goal is, however, to find a good estimate for the parameter $\alpha$ used in the cutting rules of the improved variance estimator (\ref{improvedvarest}), we now want to investigate how well we can imitate the graph of the RMSE in Figure~\ref{Modell1CWQSalpha} by the subsampling approximation. We compute the (optimal) subsampling approximation $\widehat{\textnormal{RMSE}}\left(\widehat{\sigma}^2_{\bm{n},c}\right)$ of the RMSE of $\widehat{\sigma}^2_{\bm{n},c}$ with respect to $\bm{m}$ (i.e.\ when $\bm{m}$ ranges over the values of $\mathcal M_2$ in (\ref{setm2})) analogously to the one of $\widehat{\sigma}^2_{\bm{n}}$ in (\ref{RMSEsub}), whereby we now have to replace $c_{\bm{n}}(\bm{j})$ by $c_{\bm{b}}(\bm{j})$.
\begin{table}[h!]
\centering
\begin{tabular}{crrrrrrrr}
\hline
 $\bm{m}$ & (0,0) & (1,1) & (3,3) & (4,4) & (6,6) & (7,7) \\
\hline
Mean($\widehat{\sigma}^2_{\bm{n}}$) & 1.7217 & 8.6957 & 11.5879 & 11.5588 & 11.5233 & 11.5297 \\
RMSE($\widehat{\sigma}^2_{\bm{n}}$) & 9.8391 & 2.9996 & 2.8328 & 3.8371 & 5.9867 & 7.1408  \\\hline
Mean$_{Sub}$($\widehat{\sigma}^2_{\bm{n}}$), $\gamma=0.7$ & 1.7232 & 8.7022 & 11.5825 & 11.5530 & 11.5377 & 11.5527 \\
RMSE$_{Sub}$($\widehat{\sigma}^2_{\bm{n}}$), $\gamma=0.7$ & 7.2646 & 2.6152 & 8.7550 & 12.1514 & 20.6803 & 26.2114 \\\hline
Mean$_{Sub}$($\widehat{\sigma}^2_{\bm{n}}$), $\gamma=0.8$ & 1.7231 & 8.7012 & 11.5781 & 11.5554 & 11.5585 & 11.5674 \\
RMSE$_{Sub}$($\widehat{\sigma}^2_{\bm{n}}$), $\gamma=0.8$ & 9.0152 & 2.7599 & 4.9774 & 7.0740 & 11.8211 & 14.4589 \\\hline
Mean$_{Sub}$($\widehat{\sigma}^2_{\bm{n}}$), $\gamma=0.9$ & 1.7232 & 8.7021 & 11.5810 & 11.5564 & 11.5538 & 11.5565 \\
RMSE$_{Sub}$($\widehat{\sigma}^2_{\bm{n}}$), $\gamma=0.9$ & 9.9136 & 3.5601 & 2.9771 & 3.8902 & 6.7117 & 8.3712 \\
\hline
\end{tabular}
\caption{Subsampling approximation of the mean and the RMSE in model (M1)}
\label{M1sub}
\end{table}

Figure~\ref{Modell1CWSub} shows the optimal and subsampled RMSE in model (M1) for random fields of size $(30{,}40)$ and $(45{,}60)$ using constant weights.
\begin{figure}[ht]
  \begin{center} 
  \includegraphics[width=60mm]{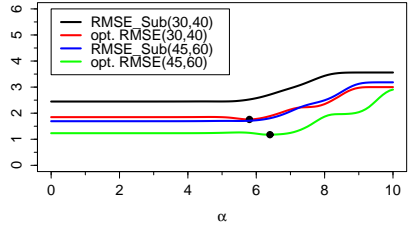}
  \caption{Optimal and subsampled RMSE in model (M1) for random fields of size $(30{,}40)$ and $(45{,}60)$ using constant weights.}
  \label{Modell1CWSub}
  \end{center}
\end{figure}

We can see that the shape of the curves of the subsampling approximation of the RMSE coincides quite well with the shape of the optimal curves of the RMSE. Moreover, the difference between the approximated and the optimal curves becomes smaller for larger sample sizes. The only problem is that the curves of the approximation attain their minimum for $\alpha=0$; however, this is not very suprising as the minimum of the curves of the optimal RMSE is not very marked. To deal with such situations, we choose $\alpha$ as the largest value for which the corresponding subsampled RMSE differs less than $1\%$, say, from the subsampled RMSE for $\alpha=0$. In the considered situation this leads to $\alpha=5.4$ compared to the optimal value of $\alpha=5.8$ for a random field of size $(30{,}40)$ and to $\alpha=5.7$ compared to the optimal value of $\alpha=6.4$ for a random field of size $(45{,}60)$. For a tolerance of $3\%$ one would obtain $\alpha=5.7$ for the smaller and $\alpha=6.1$ for the larger random field. This shows that the proposed method leads to reasonable estimates for $\alpha$ if one chooses an appropriate tolerance level. Even if the optimal and the estimated values for $\alpha$ do not match exactly, they are still close enough to each other to guarantee good results and to profit from the advantages that one achieves when using $\widehat{\sigma}^2_{\bm{n},c}$ instead of $\widehat{\sigma}^2_{\bm{n}}$.

\subsection{Accurary for Testing Image Data}
\label{SimsH0Test}

Lastly, we analyzed the accuracy of the proposed estimator when used to standardize the image test statistic based on the partial sums $ S_{\bm n} $, in order to test for the presence of deviations from a null or reference model for the image, as introduced and discussed in the introduction. Concretely, we investigated the test which rejects the null hypothesis $ H_0 $, if
\begin{equation}
\label{H0Test}
  | \bm n |^{-1/2} | S_{\bm n} | > \widehat{\sigma}_{\bm n,c}  \Phi^{-1}(1 - \alpha/2).
\end{equation}
Observe that we use the threshold cut-off estimator with constant weights to standardize the partial sum. We used the threshold rule $ c_{\bm n}(\alpha) $ investigated above in detail and investigated how one should select $ \alpha$ to obtain accurate tests in terms of the type I error rate. 

We simulated image data from spatial moving average models of order $d$ given by
\[
\xi_{i_1,i_2} = \sum_{j_1=-d}^{d} \sum_{j_2=-d}^{d} \theta_{j_1,j_2} \epsilon_{i_1+j_1,i_2+j_2} + \epsilon_{i_1,i_2},
\qquad 1 \le i_1 \le n_1, 1 \le i_2 \le n_2,
\]
where $ \epsilon_{ij} $ are i.i.d. $ N(0,1) $. The first model is model (M2), i.e. a spatial moving average model of order $d=3$ with weights $ a_1 = 0.5 $, $ a_2 = 0.3 $ and $ a_3 = 0.1 $, cf. Section~\ref{Sec:ModelM2}. The second model allows for much stronger dependencies and is of order $ d = 40 $ with coefficients given by 
\[ 
(M5) \qquad   \theta_{j_1,j_2} = \rho^{\sqrt{j_1^2+j_2^2}},
\] 
for $ | \rho | < 1 $. The degree was chosen as $ 40 $, in order to approximate a spatial linear process with $d=\infty$ and those coefficients, and nevertheless keeping the computational costs at a reasonable level. Depending on the parameter $ \rho$ the correlations can be substantial.

Table~\ref{SimH0Test} provides the results for random fields following those models with sizes ranging between $ n_1 = n_2 = 50 $  and $ n_1=n_2=250 $. For the second model the parameter $ \rho $ was chosen as $ 0.1 $, $ 0.3 $ and $ 0.5 $, in order to evaluate the procedure for weak, intermediate and stronger dependence structures. Indeed, visual inspection of random fields with $ \rho = 0.5 $ shows that the covariances induce an interesting texture-like random pattern which comes close to those observed in certain inhomogenous materials. For each case the type I error rate was simulated based on 5,000 simulation runs.

 Compared to the results addressing the estimation accurarcy, where values of $ \alpha $ around $ 5.8$  led to optimal RMSE values, it can be seen that smaller values around $ 3.6 $ are preferable, in order to obtain accurate type I error rates for the settings studied in this simulation, whatever the degree of correlation of the random field. This is, however, in good agreement with the findings of the previous simulations, especially the RMSE and bias curves shown in Figure~\ref{Modell1CWQSalpha}. Indeed, according to those results for values around $ 3.6 $ the variance is slightly larger (leading to the slightly higher RMSE values), but the bias is smaller. Our simulations reveal the interesting observation that for the accuracy in terms of the type I error rate it seems to be better to have a smaller bias. 
 
\begin{table}[h!]
	\centering
	\begin{tabular}{cccrrrrrrrrr}
		\hline
	&	& \multicolumn{10}{c}{$\alpha$} \\
$n$ & $ \rho $ & 3 & 3.1 & 3.2 & 3.3 & 3.4 & 3.5 & 3.6 & 3.7 & 3.8 & 3.9 \\ \hline
\multicolumn{2}{l}{Model (M2)} \\
$50$ &
$-$ &0.0622 & 0.06 & 0.0592 & 0.0518 & 0.0558 & 0.0558 & 0.0478 & 0.0506 & 0.0518 & 0.0474   \\
$100$ &
$-$ & 0.0618 & 0.0598 & 0.0564 & 0.0504 & 0.0574 & 0.0538 & 0.0518 & 0.0456 & 0.0502 & 0.0458 \\
$250$ &
$-$ & 0.0594 & 0.0524 & 0.0496 & 0.0542 & 0.0466 & 0.053 & 0.0574 & 0.051 & 0.05 & 0.0498 \\ 
\multicolumn{2}{l}{Model (M5)} \\
$ 50 $ &
0.1 & 0.0574 & 0.0586 & 0.0588 & 0.061 & 0.055 & 0.0528 & 0.046 & 0.0556 & 0.0558 & 0.0558 \\
&0.3 & 0.0622 & 0.0632 & 0.0566 & 0.0548 & 0.062 & 0.0578 & 0.0624 & 0.0562 & 0.0602 & 0.0632 \\
&0.5 & 0.0914 & 0.085 & 0.078 & 0.0784 & 0.0812 & 0.0782 & 0.0768 & 0.0768 & 0.0802 & 0.081 \\
100 & 
0.1 & 0.0498 & 0.0534 & 0.0502 & 0.0492 & 0.0504 & 0.0468 & 0.0466 & 0.0502 & 0.0468 & 0.048 \\
& 0.3 & 0.054 & 0.0502 & 0.0572 & 0.0496 & 0.0542 & 0.0452 & 0.0496 & 0.0532 & 0.0508 & 0.0582 \\
& 0.5 & 0.0624 & 0.06 & 0.0642 & 0.0628 & 0.062 & 0.0538 & 0.0574 & 0.061 & 0.0652 & 0.0654 \\ 
250 &
0.1 & 0.0516 & 0.0496 & 0.0524 & 0.0412 & 0.0502 & 0.0522 & 0.0554 & 0.056 & 0.0526 & 0.0516 \\
& 0.3 & 0.046 & 0.054 & 0.0462 & 0.0496 & 0.046 & 0.05 & 0.0572 & 0.0472 & 0.049 & 0.048 \\
& 0.5 & 0.0562 & 0.0538 & 0.0508 & 0.052 & 0.0534 & 0.0546 & 0.053 & 0.0566 & 0.0536 & 0.0524 \\
\hline
\end{tabular}
\caption{Simualated type I error rates of the test (\ref{H0Test}) for different parameter values $ \alpha$ for the threshold variance estimator. The test is applied to images of size $n \times n $ given by a SMA(40,40) model with covariance parameter $ \rho $}
\label{SimH0Test}
\end{table}

\subsection{Concluding Remarks}
\label{concludingremarks}
In all the settings of Subsection~\ref{SimVE} we presented results for positive values of the MA- and AR-parameters. However, in our simulation studies we also considered negative values and mixtures of positive and negative values for these parameters. As the results are qualitatively the same as for only positive weights we do not present them here.

Moreover, the question arises whether cutting rules different from those in (\ref{rule1}) and (\ref{rule3}) respectively, would lead to similar or even better results regarding the RMSE of the improved variance estimator. One could for example also think of a cutting rule of the form 
\begin{align}\label{rule2}
c_{\bm{n}}(\bm{j})=\frac{\left(\max\{|j_1|,|j_2|,|j_3|\}\right)^{\alpha}}{n_1n_2n_3}-\delta
\end{align}
or even of using a constant cutting rule, i.e.\ a rule that does not depend on the lag $\bm{j}$ and the sample size $\bm{n}$. However, the simulation studies showed that a constant cutting rule in general cannot improve the RMSE as much as a cutting rule depending on $\bm{j}$ and $\bm{n}$. Moreover, rules as in (\ref{rule2}) or of similar forms lead to comparable but not better results as the cutting rules in (\ref{rule1}) and (\ref{rule3}) concerning the optimized RMSE. That is why we only focused our attention for the analysis of $\widehat{\sigma}^2_{\bm{n},c}$ on sequences $c_{\bm{n}}(\bm{j})$ as in (\ref{rule1}) and (\ref{rule3}).

Lastly, we may conclude that the proposed estimators work well in various models of practical importance and lead to highly accurate results, both in terms of estimation accuracy and in terms of the type I error rates when considering significance testing for image data, when the parameters are selected appropriately. The latter can be achieved in a data-adaptive way by the proposed subsampling-based selection procedure.

\section{Proofs}
\label{Sec:Proofs}

This section is devoted to rigorous proofs of the presented theorems. The consistency proofs for the case $ p = 1 $, except Theorem~\ref{unknownrefsignal}, are part of the first author's Ph.D. thesis, see \cite{Prause}. The extensions to multivariate fields ($ p > 1 $) as well as Theorems~\ref{CLT}, \ref{CLTmultiv}, \ref{CLTmultiv2} and \ref{Subsampling} are newer and due to the second author.

\subsection{Proofs of Section~\ref{EstAsyVar}}

Frequently we shall use the following result, see \cite{Bil1968}: Fix $ 1 \le j \le q $ and put $ \mathcal{F}_{a}^b = \sigma( \xi_{\bm i} :  a \le i_j \le b ) $. Then $ \xi_{\bm 0} $ is $ \mathcal{F}_{-\infty}^0 $-measurable and $ \xi_{\bm \ell} $ is $ \mathcal{F}_{\ell_\star}^\infty $-measurable, where $ \ell_\star = \min_j \ell_j $. It follows that
\[
  | \mathbb E( \xi_{\bm 0} \xi_{\bm \ell} ) - \mathbb E( \xi_{\bm 0}  ) \mathbb E(\xi_{\bm \ell}) | \le 
2 \varphi^{1/2}_j(\ell_\star) \sqrt{ \mathbb E \xi_{\bm 0}^2 }  \sqrt{ \mathbb E \xi_{\bm \ell}^2 }
\le  2 \varphi^{1/2}(\ell_\star) \sqrt{ \mathbb E \xi_{\bm 0}^2 }  \sqrt{ \mathbb E \xi_{\bm \ell}^2 },
\]
by definition of $ \varphi $. For the proof of Theorem~\ref{consistencyvariance} we need the following lemma.
 

\begin{lemma}\label{propertiesY}
Let $\left\{\xi_{\bm{i}},i\in\Gamma_{\bm{n}}\right\}$ with $ \xi_{\bm i} = ( \xi_{\bm i}^{(1)}, \dots, \xi_{\bm i}^{(p)} ) $, 
 be a strictly stationary, $\varphi$-mixing ($\rho^*$-mixing) random field taking values in $ \R^p $ with $ \max_{1 \le j \le p} \mathbb E\left(\xi_{\bm{0}}^{(j)}\right)^4<\infty$. Fix $ 1 \le \nu, \mu \le p $. Then the random field $\left\{Y_{\bm{i}}^{(\nu,\mu)}(\bm{j}),\bm i\in\widetilde{\Gamma}_{\bm{n}}(\bm{j})\right\}$, defined by
\begin{align}\label{Yij}
Y_{\bm{i}}^{(\nu,\mu)}(\bm{j})\coloneqq\xi_{\bm{i}}^{(\nu)}\xi_{\bm{i}+\bm{j}}^{(\mu)}-\mathbb E\left(\xi_{\bm{0}}^{(\nu)} \xi_{\bm{j}}^{(\mu)} \right),
\end{align}
is as well a strictly stationary, $\varphi$-mixing ($\rho^*$-mixing) random field with $\mathbb E\left(Y_{\bm{0}}^{(\nu,\mu)}(\bm{j})\right)=0$ and $\mathbb E\left(Y_{\bm{0}}^{(\nu,\mu)}(\bm{j})\right)^2<\infty$. The mixing coefficients satisfy
\begin{equation}
\label{PropMixYij}
\varphi_{Y(j;\nu, \mu)}(i;r)\leq\varphi_{\xi}(i;r-2j^{\star})\quad\textnormal{for}\quad r>2j^{\star}  
\end{equation}
Fix $\bm{j}\in\mathbb Z^q$ and $ \nu, \mu \in \{1, \dots, p \} $. If (\ref{condphimix}) holds, i.e. $\sum_{r=1}^{\infty}r^{q-1}\varphi^{\frac{1}{2}}(r)<\infty$, then we have
\begin{align}\label{Ysummix}
\sum_{r=1}^{\infty}r^{q-1}\varphi_{Y(\bm{j};\nu,\mu)}^{\frac{1}{2}}(r)<\infty,
\end{align}
where $\varphi_{Y(\bm j;\nu, \mu)}$ denote the $\varphi$-mixing coefficients of the random field $\left\{Y_{\bm{i}}^{(\nu,\mu)}(\bm{j}),i\in\widetilde{\Gamma}_{\bm{n}}(\bm{j})\right\}$.
\end{lemma}

\begin{proof}
The stationarity of the random field $\left\{\xi_{\bm{i}},i\in\Gamma_{\bm{n}}\right\}$ directly implies the stationarity of $\left\{Y_{\bm{i}}^{(\nu,\mu)}(\bm{j}),\bm i\in\widetilde{\Gamma}_{\bm{n}}(\bm{j})\right\}$. Moreover, we have
\[
\mathbb E\left(Y_{\bm{0}}^{(\nu,\mu)}(\bm{j})\right)^2
\leq \mathbb E\left(\xi_{\bm{0}}^{(\nu)} \xi_{\bm{j}}^{(\mu)}\right)^2
\leq \mathbb \max_{1 \le j \le p} E\left(\xi_{\bm{0}}^{(j)} \right)^4,
\]
by the Cauchy-Schwarz inequality and the stationarity of $\left\{\xi_{\bm{i}}\right\}$. It remains to show that $\left\{Y_{\bm{i}}^{(\nu,\mu)}(\bm{j})\right\}$ is $\varphi$-mixing and that the sum in (\ref{Ysummix}) is finite. For this purpose set $j^{\star}=\max_{1\leq i\leq q}|j_i|$. On the one hand we have
\begin{align*}
\mathcal A_{Y(\bm{j})^{(\nu,\mu)}}^-(i;0)
&=\sigma\left(\left\{Y_{n_1,\ldots,n_q}^{(\nu,\mu)}(\bm j): n_i\leq 0, n_l\ \textnormal{unrestricted for}\ l\neq i \right\}\right)\\&
=\sigma\left(\left\{\xi_{n_1,\ldots,n_q}^{(\nu)}\xi_{n_1+j_1,\ldots,n_q+j_q}^{(\mu)}: n_i\leq 0, n_l\ \textnormal{unrestricted for}\ l\neq i \right\}\right)\\&
\subseteq \sigma\left(\left\{\xi_{n_1,\ldots,n_q}\xi_{n_1+j_1,\ldots,n_q+j_q}: n_i\leq 0, n_l\ \textnormal{unrestricted for}\ l\neq i \right\}\right)\\&
\subseteq\sigma\left(\left\{\xi_{n_1,\ldots,n_q}: n_i\leq j^{\star}, n_l\ \textnormal{unrestricted for}\ l\neq i \right\}\right)\\&
=\mathcal A_{\xi}^-(i;j^{\star}).
\end{align*}
On the other hand we have for all $r\geq 0$ that
\begin{align*}
\mathcal A_{Y(\bm{j})^{(\nu,\mu)}}^+(i;r)
&=\sigma\left(\left\{Y_{n_1,\ldots,n_q}^{(\nu,\mu)}(\bm j): n_i\geq r, n_l\ \textnormal{unrestricted for}\ l\neq i \right\}\right)\\&
=\sigma\left(\left\{\xi_{n_1,\ldots,n_q}^{(\nu)} \xi_{n_1+j_1,\ldots,n_q+j_q}^{(\mu)}: n_i\geq r, n_l\ \textnormal{unrestricted for}\ l\neq i \right\}\right)\\&
\subseteq\sigma\left(\left\{\xi_{n_1,\ldots,n_q} \xi_{n_1+j_1,\ldots,n_q+j_q}: n_i\geq r, n_l\ \textnormal{unrestricted for}\ l\neq i \right\}\right)\\&
\subseteq\sigma\left(\left\{\xi_{n_1,\ldots,n_q}: n_i\geq r-j^{\star}, n_l\ \textnormal{unrestricted for}\ l\neq i \right\}\right)\\&
=\mathcal A_{\xi}^+(i;r-j^{\star}).
\end{align*}
Observe that the above inclusions also show that $ \mathcal A_{Y(\bm{j})^{(\nu,\mu)}}^\pm(i;r) \subseteq \mathcal A_{Y(\bm{j})}^\pm(i;r) $.
Thus, denoting the $\varphi$-mixing coefficients of $\left\{\xi_{\bm{i}},i\in\Gamma_{\bm{n}}\right\}$ by $\varphi_{\xi}$ and those of $\left\{Y_{\bm{i}}(\bm{j}),\bm i\in\widetilde{\Gamma}_{\bm{n}}(\bm{j})\right\},$ $\bm{j}\in\mathbb Z^q,$ by $\varphi_{Y(\bm j)}$ we get \[  \varphi_{Y(j;\nu,\mu)}(i;r) \le \varphi_{Y(j)}(i;r) \] 
as well as 
\[
\varphi_{Y(j;\nu, \mu)}(i;r)\leq\varphi_{\xi}(i;r-2j^{\star})\quad\textnormal{for}\quad r>2j^{\star}
\]
leading to
\begin{align}\label{Ymix}
\varphi_{Y(\bm j;\nu, \mu)}(r)=\max_{1\leq i\leq q}\varphi_{Y(\bm j;\nu, \mu)}(i;r)\leq\max_{1\leq i\leq q}\varphi_{\xi}(i;r-2j^{\star})=\varphi_{\xi}(r-2j^{\star})\quad\textnormal{for}\quad r>2j^{\star}.
\end{align}
Hence $\left\{Y_{\bm{i}}^{(\nu,\mu)} (\bm{j})\right\}$ is $\varphi$-mixing. Lastly, to show that the sum in (\ref{Ysummix}) is finite observe that
\begin{align*}
\sum_{r=1}^{\infty}r^{q-1}\varphi_{Y(\bm{j})}^{\frac{1}{2}}(r)
=\sum_{r=1}^{2j^{\star}}r^{q-1}\varphi_{Y(\bm{j})}^{\frac{1}{2}}(r)+\sum_{r=2j^{\star}+1}^{\infty}r^{q-1}\varphi_{Y(\bm{j})}^{\frac{1}{2}}(r).
\end{align*}
The second sum is finite by inequality (\ref{Ymix}). We obtain
\[
\sum_{r=2j^{\star}+1}^{\infty}r^{q-1}\varphi_{Y(\bm{j})}^{\frac{1}{2}}(r)
\leq\sum_{r=2j^{\star}+1}^{\infty}r^{q-1}\varphi_{\xi}^{\frac{1}{2}}(r-2j^{\star})
=\sum_{r=1}^{\infty}\left(r+2j^{\star}\right)^{q-1}\varphi_{\xi}^{\frac{1}{2}}(r),
\]
where the last sum is finite by assumption (\ref{condphimix}). 
\end{proof}

\begin{lemma} 
\label{LemmaSpectral}
	Let $ p = 1 $.
	If $\left\{\xi_{\bm{i}},i\in\Gamma_{\bm{n}}\right\}$  is $ \rho^* $-mixing, then $\left\{Y_{\bm{i}}(\bm{j}),\bm i\in\widetilde{\Gamma}_{\bm{n}}(\bm{j})\right\}$ is $ \rho^* $-mixing and posses a continuous spectral density function.
\end{lemma}

\begin{proof}
Observe that
\[
\sup_{A_1, A_2: d(A_1,A_2) \ge r}  \sup_{f, g \in L_2(P)}  | \text{Corr}( f( \{ \xi_{\bm i} \xi_{\bm i + \bm j} ) : \bm i \in A_1 \}), g( \{ \xi_{\bm i} \xi_{\bm + \bm j} : \bm i \in A_2 \} ) |
\]
is less or equal to 
\[
\sup_{A_1, A_2: d(A_1,A_2) \ge r}  \sup_{f, g \in L_2(P)}  \text{Corr}( f( \{ \xi_{\bm i} : \bm i \in A_1^{j^*}  \}), g( \{ \xi_{\bm i} : \bm i \in A_2^{j^\star} \} ),
\]
where $ A_i^{j^\star} = \{ \bm k \in \mathbb{Z}^q : d( \bm k, A_i ) \le j^\star \} $ is the $ j^\star $- enlargement of $ A_i $, $ i = 1, 2 $, which is bounded by
\[
\sup_{A_1, A_2: d(A_1,A_2) \ge r-2 j^\star}  \sup_{f, g \in L_2(P)}  \text{Corr}( f( \{ \xi_{\bm i} : \bm i \in A_1  \}), g( \{ \xi_{\bm i} : \bm i \in A_2 \} ) = \rho^*( r - 2 j^\star ).
\]
It follows now from Theorem~28.21 \cite{Bradley} that $ Y_{\bm i}( \bm j ) $ possesses a continuous spectral density function, since for any  $ \rho^* $-mixing random field the linear dependence measure $ \kappa(r) $ defined by
\[
  \kappa(r) = \sup \frac{ | E( \sum_{\bm i \in \bm Q} a_{\bm i} \xi_{\bm i} ) ( \sum_{\bm i \in \bm S} b_{\bm i} \xi_{\bm i}  ) |}{ \sqrt{ \sum_{\bm i \in \bm Q} |a_{\bm i}|^2 }  \sqrt{ \sum_{\bm i \in \bm Q} |b_{\bm i}|^2 } } 
\]
where the supremum is taken over all pairs of nonempty, finite and disjoint sets $ \bm Q, \bm S $ such that $ d( \bm Q, \bm S ) \ge r $, see \cite[p.~148]{Bradley}, satisfies the necessary and sufficient condition $ \kappa(r) = o(1) $, as $ r \to \infty $.
\end{proof}

For $ p = 1 $ the following basic result for $ \varphi$-mixing random fields can be found in \cite{Deo} without a proof.
 
\begin{lemma}\label{rphi1234} Suppose $ p = 1 $.
Let condition (\ref{condphimix}) be satisfied. Set $S_{\bm{n}}\coloneqq\sum_{1\leq\bm j\leq\bm n}\xi_{\bm j}$ for all $\bm n\geq\bm 1$.
Then the following three assertions hold true.
\begin{itemize}
	\item[(a)] $\sum_{\bm{j}\in\mathbb Z^q}\left|\gamma(\bm{j})\right|<\infty.$
	\item[(b)] $|\bm{n}|^{-1}\mathbb E\left(S_{\bm{n}}^2\right)\to\sigma^2$ as $\bm{n}\to\infty.$
	\item[(c)] $|\bm{n}|^{-1}\mathbb E\left(S_{\bm{n}}^2\right)\leq A(q,\varphi)\mathbb E\left(\xi_{\bm 0}^2\right)$ for all $\bm{n}\geq\bm{1}$, where 
	\[	
	A(q,\varphi)=1+2q\sum_{r=1}^{\infty}(2r+1)^{q-1}\varphi^{\frac{1}{2}}(r).
	\]
\end{itemize}
\end{lemma}

The following result generalizes Lemma~\ref{rphi1234} to a general random field taking values in $ \R^p $, $ p \in \N $. 

\begin{lemma} 
\label{RPHI}
	Assume that condition (\ref{condphimix}) is satisfied. Let $S_{\bm{n}}\coloneqq\sum_{1\leq\bm j\leq\bm n}\xi_{\bm j}$, $\bm n\geq\bm 1$.
	\begin{itemize}
		\item[(a)] $ \sum_{\bm j \in \Z^q} \| \gamma(\bm j) \|_\infty < \infty. $
		\item[(b)] It holds
		\[
		  \Var( | \bm n |^{-1/2} S_{\bm n} ) = \mathbb E( \xi_{\bm 0} \xi_{\bm 0}' ) + \sum_{-\bm n \le^* \ell \le^* {\bm n}} \mathbb E( \xi_{\bm 0} \xi_{\bm \ell}' )  + o(1),
		\]
		as $ {\bm n} \to \infty $, and
		\[ \Var( | \bm n |^{-1/2} S_{\bm n} )  = |\bm{n}|^{-1}\mathbb E\left(S_{\bm{n}}S_{\bm{n}}' \right) \to \sigma^2, \] as $ \bm n \to \infty $, in $ (\R^{p \times p}, \| \cdot \|_\infty ) $.
	\item[(c)] If $\bm{n}\geq\bm{1}$, then \[
	 \left\| |\bm{n}|^{-1}\mathbb E\left( S_{\bm{n}} S_{\bm{n}}' \right) \right\|_\infty \leq A(q,\varphi) \max_{1 \le j \le p} \mathbb E(\xi_{\bm 0}^{(j)} )^2,
	 \] 
	 where $ A(q,\varphi) $ is as in Lemma~\ref{rphi1234}.
		\item[(d)] For all $ \bm \ell \in \Z^q $ we have
		\[
		\| \mathbb E( \xi_{\bm 0} \xi_{\bm \ell}' ) \|_\infty = \max_{1 \le i,j \le p} | \mathbb E( \xi_{\bm 0}^{(i)} \xi_{\bm \ell}^{(j)} ) |
		\le 2 \varphi^{1/2}( l_\star ) \max_{1 \le j \le p} \mathbb E | \xi_{\bm 0}^{(j)} |^2,
		\]
		where $ l_\star = \min_j l_j $.
	\end{itemize}
\end{lemma}

\begin{proof}  By \cite[Lemma~1]{Bil1968} 
	\begin{equation}
	\label{BilInequ}
	\max_{1 \le i,j \le p} | \mathbb E( \xi_{\bm 0}^{(i)} \xi_{\bm \ell}^{(j)} ) |
	\le 2 \varphi^{1/2}( l_\star )  \sqrt{ \mathbb E | \xi_{\bm 0}^{(i)} |^2 } \sqrt{ \mathbb E | \xi_{\bm 0}^{(j)} |^2 }
	\le 2 \varphi^{1/2}( l_\star ) \max_{1 \le j \le p} \mathbb E | \xi_{\bm 0}^{(j)} |^2,
	\end{equation}
	for any $ \bm \ell \in \Z^q $. Hence, by summing over all maximum-norm $q$-dimensional unit spheres,
	\begin{align*}
	  \sum_{\bm \ell \in \Z^q} \| \gamma( \bm \ell ) \|_\infty & = 
	    \sum_{\bm \ell} \max_{1 \le i, j \le p} | \mathbb E( \xi_{\bm 0}^{(i)} \xi_{\bm \ell}^{(j)} ) | \\
	    &\le 2 \max_{1 \le j \le p} \mathbb E( \xi_{\bm 0}^{(j)} )^2 \sum_{\bm \ell \in \Z^q}  \varphi^{1/2}( l_\star )  \\
	    &= 2 \max_{1 \le j \le p} \mathbb E( \xi_{\bm 0}^{(j)} )^2 \sum_{r=0}^\infty \varphi^{1/2}( r ) d(r),
	\end{align*}
	where $d(r) = \#( \bm \ell \in \Z^q : l^\star = r ) \le 2 q (2r+1)^{q-1} $. Hence (a) and (d) are shown.  To verify (b), we have to show that
  	\[
    \Var( | \bm n |^{-1/2} S_{\bm n} ) = \mathbb E( \xi_{\bm 0} \xi_{\bm 0}' ) +
    \sum_{l_1=-n_1+1}^{n_1-1} \frac{n_1-\ell_1}{n_1} \cdots \sum_{\ell_q=-n_q+1}^{n_q-1} \frac{n_q-|\ell_q|}{n_q} I_{\bm \ell} 
    \mathbb E(\xi_{\bm 0} \xi_{\bm \ell}'),
    \]
    where $ I_{\bm \ell} = \eins( \exists 1 \le j \le q: \ell_j \not= 0) $, satisfies 
   \[
    \Var( | \bm n |^{-1/2} S_{\bm n} ) = 
    \sum_{-\bm n \le^* \bm \ell \le^* \bm n}  \mathbb E( \xi_{\bm 0} \xi_{\bm \ell}' ) + o(1),
   \]
   as $ {\bm n} \to \infty $. Clearly, this follows from
   \[
   \sum_{l_1=-n_1+1}^{n_1-1} \frac{n_1-|\ell_1|}{n_1} \cdots \sum_{\ell_q=-n_q+1}^{n_q-1} \frac{n_q-|\ell_q|}{n_q} I_{\bm \ell} 
   \mathbb E(\xi_{\bm 0} \xi_{\bm \ell}') = \sum_{-\bm n \le^* \bm \ell \le^* \bm n} \mathbb E( \xi_{\bm 0} \xi_{\bm \ell}') + o(1),
   \]
   as $ {\bm n} \to \infty $.   In what follows, $ \bm \ell = (\ell_1, \dots, \ell_q)' \in \N_0^q $,
   $ {\bm \ell}_{-k} = ( \ell_j )_{j = 1, \dots, q, j \not= k} $ and $ ( {\bm \ell}_{-k}, i_k ) $ is obtained from $ {\bm \ell} $ by replacing the $k$th entry, $ \ell_k $, by $ i_k $.  Note that
   \[
     \sum_{\ell_q=-n_q+1}^{n_q-1} \frac{n_q-|\ell_q|}{n_q} I_{\bm \ell}  \mathbb E( \xi_{\bm 0} \xi_{\bm \ell}' ) 
     = \sum_{\ell_q=0}^{n_q-1} I_{\bm \ell}  \mathbb E ( \xi_{\bm 0} \xi_{\bm \ell}' )
     - \frac{1}{n_q} \sum_{\ell_q=-n_q+1}^{n_q-1} I_{\bm \ell} |\ell_q|  \mathbb E ( \xi_{\bm 0} \xi_{\bm \ell}' ). 
   \]
   We shall show that the second term is $ o(1) $ and remains $ o(1) $ when propagated through the outer summations. 
   To show that  
   \[
	\wt{r}_q( \bm \ell_{-q} ) :=  \frac{1}{n_q} \sum_{\ell_q=-n_q+1}^{n_q-1} |\ell_q|  I_{\bm \ell} \mathbb E( \xi_{\bm 0} \xi_{\bm \ell}' )  = o(1),    
	\]
   it suffices to consider the case $ I_{\bm \ell} = 1 $, i.e. we may neglect this indicator.  Decompose $ \wt{r}_q( \bm \ell_{-q} ) = \wt{r}_q^+( \bm \ell_{-q} ) + \wt{r}_q^-( \bm \ell_{-q} ) $, where $ r_q^+( \bm \ell_{-q} ) := \frac{1}{n_q} \sum_{\ell_q=1}^{n_q-1} \ell_q \mathbb E( \xi_{\bm 0} \xi_{\bm \ell}' )  $ and $ r_q^-( \bm \ell_{-q} ) := \frac{1}{n_q} \sum_{\ell_q=1}^{n_q-1} \ell_q \mathbb E( \xi_{\bm 0} \xi_{\bm \ell_{-q},-\ell_q}' ) $ We treat $ r_q( \bm \ell_{-q} ) := r_q^+( \bm \ell_{-q} ) $, since
   $ r_q^-( \bm \ell_{-q} ) $ can handle analogously.
   Indeed, we will show that for all $ k \in \{1, \dots, q \} $
   \[
     r_k( \bm \ell_{-k} ) := \frac{1}{n_k} \sum_{\ell_k=1}^{n_k-1} \ell_k \mathbb E( \xi_{\bm 0} \xi_{\bm \ell}' ) 
     = \frac{1}{n_k} \sum_{\ell_k=1}^{n_k-1} \sum_{i_k=\ell_k}^{n_k-1} \mathbb E( \xi_{\bm 0} \xi_{\bm \ell_{-k},i_k}') 
      = o(1).
   \]
   We have
   $
     \| r_k( \bm \ell_{-k} ) \|_\infty \le \frac{1}{n_k} \sum_{\ell_k=1}^{n_k-1} \sum_{i_k \ge \ell_k} \| \mathbb E( \xi_{\bm 0} \xi_{\bm \ell_{-k},i_k}' ) \|_\infty. 
   $
   Observe that for any $ j \in \{1, \dots, q\} \backslash \{ k \} $
   \[
    \left\|  \sum_{\ell_j=0}^{n_j-1} r( \bm \ell_{-k} ) \right\|_\infty 
    \le \frac{1}{n_k}
    \sum_{\ell_k=1}^{n_k-1} \sum_{i_k \ge \ell_k} \sum_{\ell_j=0}^{\infty} 
    \| \mathbb E( \xi_{\bm 0} \xi_{\bm \ell_{-k},i_k}' ) \|_\infty = o(1),
   \]
   as $ \bm n \to \infty $, by a Cesaro sum argument, since, of course,
   $ \sum_{\ell_k=1}^{\infty} \sum_{\ell_j=1}^{\infty} \| \mathbb E( \xi_{\bm 0} \xi_{\bm \ell_{-k},i_k}' ) \|_\infty  \le \sum_{\ell \in \N^q} \| \mathbb E ( \xi_{\bm 0} \xi_{\bm \ell}') \|_\infty < \infty $. Similarly, using again a Cesaro sum argument,
  \[ 
    \left\| \sum_{\ell_j=0}^{n_j-1} \frac{n_j - \ell_j}{n_j}  r( \bm \ell_{-k} ) \right\|_\infty 
    = \left\| \sum_{\ell_j=0}^{n_j-1} r( \bm \ell_{-k} ) - \frac{1}{n_j} \sum_{\ell_j=0}^{n_j-1} \sum_{\ell_k=\ell_j}^{n_j-1} r( \bm \ell_{-k} ) \right\|_\infty = o(1),
  \]
  as $ \bm n \to \infty $. By combinding the results for $ r_k^+ $ and $ r_k^- $ we now also obtain
  \[ \left\| \sum_{\ell_j=-n_j+1}^{n_j-1} \frac{n_j - |\ell_j|}{n_j}  r( \bm \ell_{-k} ) \right\|_\infty  = o(1), \] as $ \bm n \to \infty $.
   More generally, by iterating the above argument, we have, firstly, for
  any pairwise different $ j_1, \dots, j_r \in \{ 1, \dots, q\} \backslash \{ k \} $, in the $ \| \cdot \|_\infty $-norm,
  $
    \sum_{\ell_{j_1}=-n_{j_1}+1}^{n_{j_1}-1} \cdots \sum_{\ell_{j_q}=-n_{j_q}+1}^{n_{j_q}-1} r( \bm \ell_{-k} ) = o(1)
  $
  and, secondly, 
  \begin{equation}
  \label{IteratedSumsVanish}
  \sum_{\ell_{j_1}=-n_{j_1}+1}^{n_{j_1}-1} \frac{ n_{j_1} - |\ell_{j_1}| }{ n_{j_1} } \cdots
  \sum_{\ell_{j_q}=-n_{j_q}+1}^{n_{j_q}-1} \frac{ n_{j_q} - |\ell_{j_q}| }{ n_{j_q} } r( \bm \ell_{-k} ) = o(1),
  \end{equation}
  as $ \bm n \to \infty $. Now we may conclude that 
  (\ref{IteratedSumsVanish}) yields
  \begin{align*}
    & \sum_{\ell_1=-n_1+1}^{n_1-1} \frac{n_1-|\ell_1|}{n_1} \cdots \sum_{\ell_q=-n_q+1}^{n_q-1} 
    \frac{n_q-|\ell_q|}{n_q}  I_{\bm \ell} \mathbb E( \xi_{\bm 0} \xi_{\bm \ell}' )  \\
    & \quad = \sum_{\ell_1=-n_1+1}^{n_1-1} \frac{n_1-|\ell_1|}{n_1} \cdots \sum_{\ell_{q-1}=-n_{q-1}+1}^{n_{q-1}-1} 
    \frac{n_{q-1}-|\ell_{q-1}|}{n_{q-1}} \sum_{\ell_q=-n_q+1}^{n_q-1}  I_{\bm \ell} \mathbb E( \xi_{\bm 0} \xi_{\bm \ell}' ) + o(1)\\
    & \quad = \cdots = \sum_{\ell_1=-n_1+1}^{n_1-1} \cdots \sum_{\ell_q=-n_q+1}^{n_q-1}  I_{\bm \ell} \mathbb E( \xi_{\bm 0} \xi_{\bm \ell}' ) + o(1) \\
    & \quad = \sum_{-\bm n \le^* \bm \ell \le^* \bm n} \mathbb E( \xi_{\bm 0} \xi_{\bm \ell}' ) + o(1), 
  \end{align*}
  as $ \bm n \to \infty $. This shows (b).
  Assertion (c) can be shown by summing over all $q$-dimensional unit squares w.r.t to the maximum norm,
  \begin{align*}
   \| \Var( | \bm n |^{-1/2} S_{\bm n} ) \|_\infty & = \left\| \frac{1}{| \bm n| } \mathbb E( S_{\bm n} S_{\bm n}' ) \right\|_\infty 
   \le \sum_{-\bm n \le^* \bm \ell \le^* \bm n} \left| \prod_{j=1}^q \frac{n_j-|\ell_j|}{n_j} \right|   \mathbb \| E( \xi_{\bm 0} \xi_{\bm \ell}' ) \|_\infty  \\
   & \le \sum_{-\bm n \le^* \bm \ell \le^* \bm n} \| \mathbb E( \xi_{\bm 0} \xi_{\bm 0}' ) \|_\infty 
   = \sum_{r=1}^\infty c(r) \| \mathbb E( \xi_{\bm 0} \xi_{r \bm 1}' ) \|_\infty, 
  \end{align*}
  where $ c(r) = \#( - \bm n \le^* \bm i \le \bm n : \max_j | i_j | = r ) \le 2 q( 2r + 1)^{q-1} $. Using again
  (\ref{BilInequ}),  we obtain
  \[
    \| \mathbb E( \xi_{\bm 0} \xi_{\bm \ell}' ) \|_\infty \le 2 \varphi^{1/2}( r ) \max_{1 \le j \le p} \mathbb E( \xi_{\bm 0}^{(j)} )^2
  \]
  for any $ \bm \ell \in {\bm 1} : {\bm n} $ with $ l_\star = r $. Hence
  \[
    \left\|
    \Var( | \bm n |^{-1/2} S_{\bm n} ) \right\|_\infty \le \| \mathbb E( \xi_{\bm 0} \xi_{\bm 0}' )  \|_\infty
    + \sum_{r=1}^{n^\star} 2 q (2r+1)^{q-1} 4 \varphi^{1/2}(r) \max_{1 \le j \le p} \mathbb E( \xi_{\bm 0}^{(j)} )^2,
  \] 
  and we arrive at 
    \[
    \left\|
    \Var( | \bm n |^{-1/2} S_{\bm n} ) \right\|_\infty
    \le A(\varphi,q) \max_{1 \le j \le p} \mathbb E( \xi_{\bm 0}^{(j)} )^2.
  \]
 \end{proof}

We are now in a position to prove the consistency of the estimator $ \wh{\sigma}_{\bm n}^2 $ for multivariate random fields.

\begin{proof}[Proof of Theorem \ref{consistencyvariance}]
Define the $ p \times p $ matrix
\[
\widetilde{\sigma}_{\bm{n}}^2\coloneqq\mathbb E\left(\widehat{\sigma}_{\bm{n}}^2\right)=\sum_{|\bm{j}|\leq \bm{m}}w_{\bm{m}}(\bm{j})\gamma(\bm{j}).
\]
We show that $ \| \widehat{\sigma}_{\bm{n}}^2-\widetilde{\sigma}_{\bm{n}}^2 \|_\infty \stackrel{\mathbb P}\to 0$ and $ \| \widetilde{\sigma}_{\bm{n}}^2-\sigma^2 \|_\infty \to 0$ as $\bm{n}\to\infty$.
Notice that
\begin{align*}
\left\|\widetilde{\sigma}_{\bm{n}}^2-\sigma^2\right\|_\infty
&=\left\|\sum_{|\bm{j}|\leq \bm{m}}w_{\bm{m}}(\bm{j})\gamma(\bm{j})-\sum_{\bm{j}\in\mathbb Z^q}\gamma(\bm{j})\right\|_\infty
\leq\left\|\sum_{|\bm{j}|\leq \bm{m}}\left(w_{\bm{m}}(\bm{j})-1\right)\gamma(\bm{j})\right\|_\infty + o(1)\\&
\leq\sum_{|\bm{j}|\leq \bm{m}}\left|w_{\bm{m}}(\bm{j})-1\right|\left\|\gamma(\bm{j})\right\|_\infty \mathds{1}_{\{|\bm{j}|\leq \bm{m}\}}+o(1),
\end{align*}
as $ \bm n \to \infty $, which implies $ \bm m \to \infty $, by Lemma~\ref{rphi1234}(a). 
For the first summand this follows by dominated convergence: For $\bm{j}\in\mathbb Z^q$ let
\[
f_{\bm{m}}(\bm{j})=\left|w_{\bm{m}}(\bm{j})-1\right|\left\|\gamma(\bm{j})\right\|_\infty \mathds{1}_{\{|\bm{j}|\leq \bm{m}\}}.
\]
Clearly, $f_{\bm{m}}(\bm{j})\to 0$ for each fixed $\bm{j}\in\mathbb Z^q$, as $\bm{m}\to\infty$. Moreover,
\[
\left|f_{\bm{m}}(\bm{j})\right|\leq (C_w+1)\left\|\gamma(\bm{j})\right\|_\infty\eqqcolon g(\bm{j})
\]
with $\sum_{\bm{j}\in\mathbb Z^q}\left|g(\bm{j})\right|<\infty$ by Lemma~\ref{rphi1234}(a). Thus, the first summand also converges to zero and hence we obtain $\| \widetilde{\sigma}_{\bm{n}}^2-\sigma^2 \|_\infty \to 0$.

We now show that $ \| \widehat{\sigma}_{\bm{n}}^2-\widetilde{\sigma}_{\bm{n}}^2 \|_\infty \stackrel{\mathbb P}\to 0$. By the Markov inequality we have
\[
\mathbb P\left(\left\|\widehat{\sigma}_{\bm{n}}^2-\widetilde{\sigma}_{\bm{n}}^2\right\|_\infty>\varepsilon\right)
\le \frac{ \mathbb E \| \widehat{\sigma}_{\bm{n}}^2-\widetilde{\sigma}_{\bm{n}}^2 \|_\infty }{\varepsilon}.
\]
We shall bound the right-hand side by an expression which is $ o(1) $. Combined with $ \| \wt{\sigma}_{\bm n}^2 - \sigma^2 \|_\infty = o(1) $, as $ \bm n \to \infty $, this then shows the second assertion of Theorem~\ref{consistencyvariancel2}. By boundedness of the weights $w_{\bm{m}}(\bm{j})$ we obtain
\begin{align}\label{stochconv}
\mathbb E \| \widehat{\sigma}_{\bm{n}}^2-\widetilde{\sigma}_{\bm{n}}^2 \|_\infty
&
\leq \sum_{|\bm{j}|\leq \bm{m}}\left|w_{\bm{m}}(\bm{j})\right| \mathbb E\left\|\widehat{\gamma}_n(\bm{j})-\gamma(\bm{j})\right\|_\infty\notag\\&
\leq C_w \sum_{\nu, \mu = 1}^{p} \sum_{|\bm{j}|\leq \bm{m}}
\mathbb E\left|\widehat{\gamma}_n^{(\nu,\mu)}(\bm{j})-\gamma^{(\nu,\mu)}(\bm{j})\right| \notag\\&
\leq C_w   \sum_{\nu, \mu = 1}^{p}  \sum_{|\bm{j}|\leq \bm{m}} \left(\mathbb E\left|\widehat{\gamma}_n^{(\nu,\mu)}(\bm{j})-\gamma^{(\nu,\mu)}(\bm{j})\right|^2\right)^{1/2}.
\end{align}
Observe that
\[
\widehat{\gamma}_n^{(\nu,\mu)}(\bm{j})-\gamma(\bm{j})^{(\nu,\mu)}
=\frac{1}{\left|\widetilde{\Gamma}_{\bm{n}}(\bm{j})\right|}\sum_{\bm i\in\widetilde{\Gamma}_{\bm{n}}(\bm{j})}\xi_{\bm{i}}^{(\nu)} \xi_{\bm{i}+\bm{j}}^{(\mu)}-\mathbb E\left(\xi_{\bm{0}}^{(\nu)} \xi_{\bm{j}}^{(\mu)} \right)
=\frac{1}{\left|\widetilde{\Gamma}_{\bm{n}}(\bm{j})\right|}\sum_{\bm i\in\widetilde{\Gamma}_{\bm{n}}(\bm{j})}Y_{\bm{i}}^{(\nu,\mu)}(\bm{j})
\]
with $Y_{\bm{i}}^{(\nu,\mu)}(\bm{j})$ defined in (\ref{Yij}).
Thus we obtain
\[
\mathbb E\left|\widehat{\gamma}_n^{(\nu,\mu)}(\bm{j})-\gamma^{(\nu,\mu)}(\bm{j})\right|^2
=\mathbb E\left(\frac{1}{\left|\widetilde{\Gamma}_{\bm{n}}(\bm{j})\right|}\sum_{\bm i\in\widetilde{\Gamma}_{\bm{n}}(\bm{j})}Y_{\bm{i}}^{(\nu,\mu)}(\bm{j})\right)^2
=\frac{1}{\left|\widetilde{\Gamma}_{\bm{n}}(\bm{j})\right|^2} \mathbb E\left(\sum_{\bm i\in\widetilde{\Gamma}_{\bm{n}}(\bm{j})}Y_{\bm{i}}^{(\nu,\mu)}(\bm{j})\right)^2.
\]
For each $\bm{j}\in\mathbb Z^q$ we can now apply Lemma~\ref{rphi1234}(c) to the expectation of the squared sum of the random field $\{Y_{\bm{i}}^{(\nu,\mu)}(\bm{j})\}$ as condition (\ref{Ysummix}) is fulfilled by Lemma~\ref{propertiesY}. This leads to
\begin{align}\label{sumest}
\mathbb E\left|\widehat{\gamma}_n^{(\nu,\mu)}(\bm{j})-\gamma^{(\nu,\mu)}(\bm{j})\right|^2
&\leq\frac{1}{\left|\widetilde{\Gamma}_{\bm{n}}(\bm{j})\right|}A\left(q,\varphi_{Y(j)}\right)\mathbb E\left(Y_{\bm{0}}^{(\nu,\mu)}(\bm{j})\right)^2\notag\\&
\leq\frac{\left(1+2q\sum_{r=1}^{\infty}(2r+1)^{q-1}\varphi_{Y(\bm{j})}^{\frac{1}{2}}(r)\right) \max_{1 \le j \le p} \mathbb E(\xi_{\bm{0}}^{(j)})^4} {\left|\widetilde{\Gamma}_{\bm{n}}(\bm{j})\right|},
\end{align}
where the last inequality is fulfilled by Lemma~\ref{propertiesY} and the definition of $A(q,\varphi)$ in Lemma~\ref{rphi1234}(c).
Combining (\ref{stochconv}) with (\ref{sumest}) we obtain 
\[
  \mathbb E \| \widehat{\sigma}_{\bm{n}}^2-\widetilde{\sigma}_{\bm{n}}^2 \|_\infty 
  \le  \max_{1 \le j \le p} \mathbb E^{1/2}(\xi_{\bm{0}}^{(j)})^4 \sum_{|\bm{j}|\leq \bm{m}}\left(\dfrac{1+2q\sum_{r=1}^{\infty}(2r+1)^{q-1}\varphi_{Y(\bm{j})}^{\frac{1}{2}}(r)}{\left|\widetilde{\Gamma}_{\bm{n}}(\bm{j})\right|}\right)^{1/2}
\]
and therefore
\begin{align}\label{estinr}
\mathbb P\left(\left\|\widehat{\sigma}_{\bm{n}}^2-\widetilde{\sigma}_{\bm{n}}^2\right\|_\infty >\varepsilon\right)
\leq \frac{p^2 C_w}{\varepsilon} \max_{1 \le j \le p} \mathbb E^{1/2}(\xi_{\bm{0}}^{(j)})^4 \sum_{|\bm{j}|\leq \bm{m}}\left(\dfrac{1+2q\sum_{r=1}^{\infty}(2r+1)^{q-1}\varphi_{Y(\bm{j})}^{\frac{1}{2}}(r)}{\left|\widetilde{\Gamma}_{\bm{n}}(\bm{j})\right|}\right)^{1/2}.
\end{align}
To estimate the right-hand side consider the decomposition
\[
\sum_{r=1}^{\infty}(2r+1)^{q-1}\varphi_{Y(\bm{j})}^{\frac{1}{2}}(r)
=\sum_{r=1}^{2j^{\star}}(2r+1)^{q-1}\varphi_{Y(\bm{j})}^{\frac{1}{2}}(r)
+\sum_{r=2j^{\star}+1}^{\infty}(2r+1)^{q-1}\varphi_{Y(\bm{j})}^{\frac{1}{2}}(r)
\eqqcolon S_1+S_2.
\]
Since the $\varphi$-mixing coefficients are decreasing with $\varphi(0)=1$, we obtain
\[
S_1\leq\sum_{r=1}^{2j^{\star}}(2r+1)^{q-1}\varphi_{Y(\bm{j})}^{\frac{1}{2}}(0)
=\sum_{r=1}^{2j^{\star}}(2r+1)^{q-1}
\leq 2j^{\star}(4j^{\star}+1)^{q-1}.
\]
As the sum in (\ref{estinr}) only ranges over $\bm{j}\in\mathbb Z^q$ with $|\bm{j}|\leq\bm{m}$ we know that $j^{\star}\leq m^{\star}$. This implies
\[
S_1\leq 2m^{\star}(4m^{\star}+1)^{q-1}\leq c_1(m^{\star})^q
\]
for some finite constant $c_1\in\mathbb R$ depending only on the dimension $q$ of the random field.
For the estimation of $S_2$ we use again (\ref{Ymix}) and get
\begin{align*}
S_2&\leq\sum_{r=2j^{\star}+1}^{\infty}(2r+1)^{q-1}\varphi_{\xi}^{\frac{1}{2}}(r-2j^{\star})\\&
=\sum_{r=1}^{\infty}(2(r+2j^{\star})+1)^{q-1}\varphi_{\xi}^{\frac{1}{2}}(r)
\leq(4j^{\star})^{q-1}\sum_{r=1}^{\infty}(2r+2)^{q-1}\varphi_{\xi}^{\frac{1}{2}}(r).
\end{align*}
Since $\sum_{r=1}^{\infty}(2r+2)^{q-1}\varphi_{\xi}^{\frac{1}{2}}(r)$ is finite by assumption, we can find a constant $c_2\in\mathbb R$ such that for all $|\bm{j}|\leq\bm{m}$ we have
\[
S_2\leq c_2(m^{\star})^{q-1}.
\]
Putting things toghether we obtain
\begin{align}\label{S1S2}
\sum_{r=1}^{\infty}(2r+1)^{q-1}\varphi_{Y(\bm{j})}^{\frac{1}{2}}(r)\leq c_3(m^{\star})^q
\end{align}
for $c_3\in\mathbb R$.

Finally, observe that
\begin{align}\label{prod}
\left|\widetilde{\Gamma}_{\bm{n}}(\bm{j})\right|
\geq\prod_{i=1}^q(n_i-j^{\star})
\geq\prod_{i=1}^q(n_{\star}-m^{\star})
=(n_{\star}-m^{\star})^q
\end{align}
for all $\bm{j}\in\mathbb Z^q$ with $|\bm{j}|\leq\bm{m}$. Combining (\ref{S1S2}) and (\ref{prod}) with (\ref{estinr}) we obtain
\begin{align*}
\mathbb P\left(\left\|\widehat{\sigma}_{\bm{n}}^2-\widetilde{\sigma}_{\bm{n}}^2\right\|_\infty>\varepsilon\right)
&\leq\frac{p^2C_w}{\varepsilon} \max_{1 \le j \le p} \mathbb E^{1/2}(\xi_{\bm{0}}^{(j)})^4 \sum_{|\bm{j}|\leq \bm{m}}\left(\dfrac{1+2qc_3(m^{\star})^q}{(n_{\star}-m^{\star})^q}\right)^{1/2}\\&
\leq\frac{C_w}{\varepsilon} \max_{1 \le j \le p} \mathbb E^{1/2}(\xi_{\bm{0}}^{(j)})^4 \left(\prod_{i=1}^q(2m_i+1)\right)\left(\dfrac{1+2qc_3(m^{\star})^q}{(n_{\star}-m^{\star})^q}\right)^{1/2}.
\end{align*}
As
$
\prod_{i=1}^q(2m_i+1)
\leq (2m^{\star}+1)^q
\leq c_5(m^{\star})^q
$
for some $c_5\in\mathbb R$, we can find a constant $C\in\mathbb R$ such that
\[
\mathbb P\left(\left\|\widehat{\sigma}_{\bm{n}}^2-\widetilde{\sigma}_{\bm{n}}^2\right\|_\infty>\varepsilon\right)
\leq C\left(\frac{(m^{\star})^3}{n_{\star}-m^{\star}}\right)^{q/2}
=C\left(\frac{n_{\star}}{(m^{\star})^3}-\frac{1}{(m^{\star})^2}\right)^{-q/2}.
\]
The last expression tends to zero for $\bm{n}\to\infty$ since by assumption
$
\left(m^{\star}\right)^3 / n_{\star} =o(1).
$
We obtain $\| \widehat{\sigma}_{\bm{n}}^2-\widetilde{\sigma}_{\bm{n}}^2 \|_\infty \stackrel{\mathbb P}\to 0$ which completes the proof.
\end{proof}

\begin{proof}[Proof of Theorem \ref{consistencyvariancel2}]
We have
$
\mathbb E\left[\left(\widehat{\sigma}_{\bm{n}}^2-\sigma^2\right)^2\right]=\mathbb V\textnormal{ar}\left(\widehat{\sigma}_{\bm{n}}^2\right)+\textnormal{Bias}^2\left(\widehat{\sigma}_{\bm{n}}^2\right),
$
where $\textnormal{Bias}^2\left(\widehat{\sigma}_{\bm{n}}^2\right)= ( \widetilde{\sigma}_{\bm{n}}^2-\sigma^2)^2\to 0$ as $\bm{n}\to\infty$ by the proof of Theorem~\ref{consistencyvariance}. For the variance we obtain by the Cauchy-Schwarz inequality and the uniform boundedness of the weights $w_{\bm{m}}(\bm{j})$ that
\begin{align}\label{var11}
\mathbb V\textnormal{ar}\left(\widehat{\sigma}_{\bm{n}}^2\right)
& \leq \mathbb E\left[\sum_{|\bm{j}|\leq \bm{m}}w^2_{\bm{m}}(\bm{j})\left(\widehat{\gamma}_n(\bm{j})-\gamma(\bm{j})\right)^2\sum_{|\bm{j}|\leq \bm{m}}1\right]\notag\\&
\leq c_1(m^{\star})^q\sum_{|\bm{j}|\leq \bm{m}}\mathbb E\left[\left(\widehat{\gamma}_n(\bm{j})-\gamma(\bm{j})\right)^2\right]
\end{align}
for some constant $c_1\in\mathbb R$. By the proof of Theorem~\ref{consistencyvariance} we moreover have
\begin{align}\label{var22}
\sum_{|\bm{j}|\leq \bm{m}}\mathbb E\left[\left(\widehat{\gamma}_n(\bm{j})-\gamma(\bm{j})\right)^2\right]
\leq c_2(m^{\star})^q\frac{(m^{\star})^q}{(n_{\star}-m^{\star})^q}
\end{align}
for some constant $c_2\in\mathbb R$. By combining (\ref{var11}) with (\ref{var22}) we finally obtain
$
\mathbb V\textnormal{ar}\left(\widehat{\sigma}_{\bm{n}}^2\right)\leq C\left(\frac{(m^{\star})^3}{n_{\star}-m^{\star}}\right)^q =o(1), 
$
by assumption, for some $C\in\mathbb R$. This completes the proof.
\end{proof}
\begin{proof}[Proof of Corollary \ref{autocovest}]
This is a direct consequence of the proof of Theorem~\ref{consistencyvariance}.
\end{proof}

Let us now show the consistency of the estimator, when centering the terms at the temporal average, under the multicplicate model.

\begin{proof}[Proof of Theorem \ref{unknownrefsignal}]
$\check{\gamma}_{\bm n}(\bm j)$ satisfies the decomposition
\begin{align}\label{gammadecomposition}
\check{\gamma}_{\bm n}(\bm j)=\widetilde{\gamma}_{\bm n}(\bm j)+\widetilde{R}_{\bm n}(\bm j),\quad \widetilde{R}_{\bm n}(\bm j)=R_{\bm n}(\bm j)+R_{\bm n}(\bm 0)+\overline{R}_{\bm n},
\end{align}
where
\[
\widetilde{\gamma}_{\bm n}(\bm j)\coloneqq\frac{1}{\left|\widetilde{\Gamma}_{\bm n}(\bm j)\right|}\sum_{\bm i\in\widetilde{\Gamma}_{\bm n}}\xi_{\bm i}\xi_{\bm{i+j}}',
\]
and
\begin{align*}
R_{\bm n}(\bm j)\coloneqq -\frac{1}{\left|\widetilde{\Gamma}_{\bm n}(\bm j)\right|}\sum_{(i_1,\overline{\bm{\iota}})\in\widetilde{\Gamma}_{\bm n}(\bm j)}\xi_{\bm{i+j}}\overline{\xi}_{\cdot,\overline{\bm{\iota}}}',\quad\overline{R}_{\bm n}\coloneqq\frac{1}{\left|\widetilde{\Gamma}_{\bm n}(\bm j)\right|}\sum_{(i_1,\overline{\bm{\iota}})\in\widetilde{\Gamma}_{\bm n}(\bm j)} \overline{\xi}_{\cdot,\overline{\bm{\iota}}} \overline{\xi}_{\cdot,\overline{\bm{\iota}}}' .
\end{align*}
(\ref{gammadecomposition}) induces the decomposition
\[
\check{\sigma}^2_{\bm n} = \sum_{|\bm j|\leq\bm m}w_{\bm m}(\bm j)\widetilde{\gamma}_{\bm n}(\bm j)+R_{\bm n}^{\sigma},\quad R_{\bm n}^{\sigma}\coloneqq R_{\bm n}^{\prime}+R_{\bm n}^{\prime\prime},
\]
with
\[
R_{\bm n}^{\prime}\coloneqq\sum_{|\bm j|\leq\bm m}w_{\bm m}(\bm j)R_{\bm n}(\bm j),\quad R_{\bm n}^{\prime\prime}\coloneqq\sum_{|\bm j|\leq\bm m}w_{\bm m}(\bm j)(R_{\bm n}(\bm 0)+\overline{R}_{\bm n}).
\]
Noting that Theorems \ref{consistencyvariance} and \ref{consistencyvariancel2} apply to the random field $\{\xi_{\bm i}:\bm i\in\mathbb Z^q\}$, the assertion follows if $R_{\bm n}^{\sigma}=o_{\mathbb P}(1)$, as $\bm n\to\infty$. We treat $R_{\bm n}^{\prime}$ as $R_{\bm n}^{\prime\prime}$ can be dealt with in a similar way. Observe that with $\bm j=(j_1,\overline{\bm j}), \overline{\bm j}=(j_2,\ldots,j_q)$, arranging terms leads to
\begin{align*}
R_{\bm n}(\bm j)&=-\frac{1}{\left|\widetilde{\Gamma}_{\bm n}(\bm j)\right|}\sum_{(i_1,\overline{\bm{\iota}})\in\widetilde{\Gamma}_{\bm n}(\bm j)}\varepsilon^{(T)}_{i_1+j_1}\varepsilon^{(S)}_{\overline{\bm{\iota}}+\overline{\bm j}}\frac{1}{n_1}\sum_{l=1}^{n_1}\varepsilon^{(T)}_l (\varepsilon^{(S)}_{\overline{\bm{\iota}}})'\\&
%
%
=-\left(\frac{1}{n_1}\sum_{l=1}^{n_1}\varepsilon^{(T)}_l\right)\left(\frac{1}{n_1-|j_1|}\sum_{i_1=1}^{n_1-|j_1|}\varepsilon^{(T)}_{i_1+j_1}\right)\widetilde{\gamma}_{\bm n}^{(S)}(\overline{\bm j}),
\end{align*}
where
\[
\widetilde{\gamma}_{\bm n}^{(S)}(\overline{\bm j})\coloneqq\frac{1}{\prod_{l=2}^q(n_l-|j_l|)}\sum_{\overline{\bm{\iota}}\in\widetilde{\Gamma}_{\bm n}(\overline{\bm j})} \varepsilon^{(S)}_{\overline{\bm{\iota}}}(\varepsilon^{(S)}_{\overline{\bm{\iota}}+\overline{\bm j}})'.
\]
Put $\widetilde{\gamma}^{(S)}(\overline{\bm j})\coloneqq\mathbb E(\varepsilon^{(S)}_{\bm 0}(\varepsilon^{(S)}_{\overline{\bm j}})')$,
$
\overline{S}_{\bm n}\coloneqq\frac{1}{n_1}\sum_{l=1}^{n_1}\varepsilon^{(T)}_l $, and $ \overline{S}_{\bm n}(j_1)\coloneqq\frac{1}{n_1-|j_1|}\sum_{l=1}^{n_1-|j_1|}\varepsilon^{(T)}_l.
$
Let us denote by $ R_{\bm n}^{(\nu,\mu)}(\bm j) $, $ \widetilde{\gamma}^{(S)}(\overline{\bm j})_{\nu\mu} $ and $ \widetilde{\gamma}^{(S)}(\overline{\bm j})_{\nu\mu} $ the $(\nu, \mu$)th element of the corresponding (random)  $p \times p $ matrices, $ 1 \le \nu, \mu \le p $. By Lemma \ref{rphi1234}(c) and independence, we obtain
\begin{align*}
& \max_{1 \le \nu, \mu \le p}
\mathbb E\left|R_{\bm n}^{(\nu,\mu)}(\bm j)\right| \leq
\max_{1 \le \nu, \mu \le p}
\mathbb E\left(\left|\overline{S}_{\bm n}\right|\left|\overline{S}_{\bm n}(j_1)\right|\right)\left(\mathbb E\left|\widetilde{\gamma}_{\bm n}^{(S)}(\overline{\bm j})_{\nu\mu}-\widetilde{\gamma}^{(S)}(\overline{\bm j})_{\nu\mu}\right|+\left|\widetilde{\gamma}^{(S)}(\overline{\bm j})_{\nu\mu} \right|\right)\\& \qquad
\leq \max_{1 \le \nu, \mu \le p}
\sqrt{\mathbb E\left(\overline{S}_{\bm n}^2\right)\mathbb E\left(\overline{S}_{\bm n}^2(j_1)\right)}\left(\sqrt{\mathbb E\left(\widetilde{\gamma}_{\bm n}^{(S)}(\overline{\bm j})_{\nu\mu}-\widetilde{\gamma}^{(S)}(\overline{\bm j})_{\nu\mu}\right)^2}+\left|\widetilde{\gamma}^{(S)}(\overline{\bm j})_{\nu\mu}\right|\right)\\& \qquad
\leq \max_{1 \le \nu, \mu \le p}
\frac{K}{\sqrt{n_1(n_1-|j_1|)}}\left(\sqrt{\mathbb E\left(\widetilde{\gamma}_{\bm n}^{(S)}(\overline{\bm j})_{\nu\mu}-\widetilde{\gamma}^{(S)}(\overline{\bm j})_{\nu\mu}\right)^2}+\left|\widetilde{\gamma}^{(S)}(\overline{\bm j})_{\nu\mu}\right|\right)
\end{align*}
with $K\coloneqq A(1,\varphi_{\varepsilon^{(T)}})\mathbb E\left((\varepsilon_0^{(T)})^2\right)$. Hence, for any $\varepsilon>0$
\[
\mathbb E \| R_{\bm n}' \|_\infty \leq
\frac{KC_w}{ n_1^{1/2}} \sum_{\nu, \mu=1}^p  \sum_{|\bm j|\leq\bm m}\sqrt{\mathbb E\left(\widetilde{\gamma}_{\bm n}^{(S)}(\overline{\bm j})_{\nu\mu}-\widetilde{\gamma}^{(S)}(\overline{\bm j})_{\nu\mu}\right)^2}
+\frac{KC_w}{ n_1^{1/2}} \sum_{\nu, \mu=1}^p \sum_{|\bm j|\leq\bm m}\left|\widetilde{\gamma}^{(S)}(\overline{\bm j})_{\nu\mu}\right|.
\]
Since $\sum_{\overline{\bm j}\in\mathbb Z^{q-1}}\left|\widetilde{\gamma}^{(S)}(\overline{\bm j})_{\nu\mu}\right|<\infty$ and $ (m^\star)^2 / n_1 = o(1) $, one may now argue as in the proof of Theorem \ref{consistencyvariance}, cf. (\ref{stochconv}), to show that
$ \mathbb E \| \check{\sigma}_{\bm n}^2 - \sigma^2 \|_\infty = o(1) $, which completes the proof.
\end{proof}

Next we establish the results on the asymptotic distribution of the estimators.

\begin{proof}[Proof of Lemma~\ref{CLTLemma}, Theorem \ref{CLT} and Theorem \ref{CLTmultiv}]
	Recall that $ Y_{\bm i}( \bm j ) = \xi_{\bm i} \xi_{\bm i + \bm j} - \gamma( \bm j ) $. We have 
	\begin{align*}
	\sqrt{ | \widetilde{\Gamma}_{\bm n}( \bm j ) | } ( \wh{\gamma}_{\bm n}( \bm j ) - \gamma ( \bm j) ) 
	&= \frac{1}{\sqrt{ | \widetilde{\Gamma}_{\bm n}( \bm j ) | }} \sum_{\bm i \in \wt{\Gamma}_{\bm n}(\bm j )} Y_{\bm i}( \bm j )  \\
	& =  \frac{1}{\sqrt{ | \widetilde{\Gamma}_{\bm n}( \bm j ) | }} \sum_{\bm i \in \bm 1: \bm n} Y_{\bm i}( \bm j ) -  R_{\bm n}^{(1)}(\bm j) \\
	& = \frac{1}{ \sqrt{| \bm n |} } \sum_{\bm i \in \bm 1: \bm n} Y_{\bm i}( \bm j )  -  R_{\bm n}^{(1)}(\bm j)  +  R_{\bm n}^{(2)}(\bm j),
	\end{align*}
	where
	\begin{align*}
	R_{\bm n}^{(1)}( \bm j ) & = \frac{1}{\sqrt{ | \widetilde{\Gamma}_{\bm n}( \bm j ) | }}
	\sum_{\bm i \in \bm 1:\bm n \backslash \wt{\Gamma}_{\bm n}( \bm j) } Y_{\bm i}( \bm j ) 
	\end{align*}
	and
	\begin{align*}
	R_{\bm n}^{(2)}(\bm j) & = \left( \sqrt{ \frac{ | \bm n| }{ | \wt{\Gamma}_{\bm n}( \bm j) | } } - 1 \right) \frac{1}{\sqrt{| \bm n |}} \sum_{\bm i \in \bm 1: \bm n} Y_{\bm i}( \bm j ).
	\end{align*}
	Because $ Y_{\bm i}( \bm j )  $ is a strictly stationary $ \varphi $-mixing random field taking values in $ \R^p $ with asymptotic variance $ \zeta_{\bm j}^2 > 0 $, since (\ref{CovFunctionB_univ}) defines a positive definite quadratic form, and satisfying (\ref{condphimix}), the Cram\'er-Wold device and  \cite{Deo} show that
	\[
	\frac{1}{\sqrt{| \bm n |}} \sum_{\bm i \in \bm 1: \bm n} Y_{\bm i}( \bm j )  
	\Rightarrow B_{\bm j},
	\]
	as $ \bm n \to \infty $, for each $ | \bm j | \le \bm m $. This immediately implies
	$ \max_{| \bm j | \le \bm m} \| R_{\bm n}^{(2)}(\bm j) \|_\infty = o_{\mathbb P}(1) $. 
	Observe that $ \bm 1: \bm n \backslash \wt{\Gamma}_{\bm n}( \bm j ) $ is a finite
	union of index rectangles $ I_{\bm n}^{(1,\ell)} $, $ \ell = 1, \dots, L $, each of which has at least
	one index, denoted by $ k_\ell $, ranging over $ n_{k_\ell}-j_{k_\ell}+1, \dots, n_{k_\ell} $ for some $ k_\ell \in \{ 1, \dots, q \} $. 
	Therefore, $ | I_{\bm n}^{(1,\ell)} | \le j_{k_\ell} \prod_{\nu=1, \nu \not= k_\ell}^{q} (n_\nu - j_\nu)  = o( | \bm n | ) $, as $ \bm n \to \infty $.
	It follows that  for each $ \nu, \mu \in \{1, \dots, p \} $,
	\[ | \wt{\Gamma}_{\bm n}(\bm j) |^{-1}  E | R_{\bm n}^{(1,\ell)}(\bm j)_{\nu\mu}  |^2 = O\left( j_{k_\ell} / n_{k_\ell}  \sqrt{ \mathbb E( Y_{\bm 0}( \bm j)_{\nu\mu} )^2 } \right),
	\]
	where $ R_{\bm n}^{(1,\ell)}(\bm j)_{\nu\mu} = \sum_{\bm i \in I_{\bm n}^{(1,\ell)} } Y_{\bm i}( \bm j )_{\nu\mu} $
	for $ \ell = 1, \dots, L $. Consequently, we obtain the estimate
	\begin{align*}
	  \mathbb P\left(  \max_{| \bm j | \le \bm m} \left\| R_{\bm n}^{(1)}( \bm j ) \right\|_\infty > \varepsilon \right)
	  & \le  \sum_{| \bm j| \le \bm m} \sum_{\nu, \mu=1}^p \sum_{\ell=1}^L \frac{ \mathbb E \left| R_{\bm n}^{(1,\ell)}( \bm j )_{\nu\mu} \right|^2 }{ | \wt{\Gamma}_{\bm n}( \bm j ) | \varepsilon^2 } \\
	  & = O\left( \sum_{| \bm j| \le \bm m} \sum_{\nu, \mu=1}^p \sum_{\ell=1}^L \frac{ j_{k_\ell} }{ n_{k_\ell}  } 
	  \sqrt{ \mathbb E( Y_{\bm 0}( \bm j)_{\nu\mu} )^2 }  \right) \\
	  & = o(1),
	\end{align*}
	as $ \bm n \to \infty $, for all $ \varepsilon > 0 $, such that
	\[
	  \max_{|\bm j| \le \bm m} \| R_{\bm n}^{(1)}( \bm j ) \|_\infty = o_{\mathbb P}(1),
	\] 
	as $ \bm n \to \infty $. Hence, we may change the scaling factor to $ \sqrt{|\bm n|} $ and shall now establish the weak convergence
	\begin{equation}
	\label{ShowWeakConv}
		\left( \sqrt{ | \bm n | } ( \wh{\gamma}_{\bm n}( \bm j ) - \gamma( \bm j) ) \right)_{|\bm j| \le \bm m} 
		\Rightarrow  \mathcal{S}',
	\end{equation}
	as $ \bm n \to \infty $, where $ \mathcal{S}' $ denotes the process $ \mathcal{S} $ for constant weights $ w_{\bm m}(\bm j ) = 1 $, $ | \bm j | \le \bm m $. 	This follows, if we prove 
	\begin{equation}
	\label{CLTmToShow}
	  \sum_{| \bm j | \le \bm m} \sum_{\nu,\mu=1}^p \lambda_{\bm j}^{(\nu,\mu)} | \bm n |^{-1/2} \sum_{ {\bm i} \in 1:{\bm n}} Y_{\bm i}( \bm j )_{\nu \mu}
	  \stackrel{d}{\to} 
	  \sum_{| \bm j | \le \bm m} \sum_{\nu,\mu=1}^p \lambda_{\bm j}^{(\nu,\mu)}  B_{\bm j}^{(\nu,\mu)}
	\end{equation}
	for all  arrays of coefficients $ \bm \lambda  = \{  \lambda_{\bm j}^{(\nu,\mu)} : 1 \le \nu, \mu \le p, | \bm j | \le \bm m \} $ of real numbers such that the variance of the limiting random variable is positive, by virtue of the  Cram\'er-Wold device, see e.g. \cite[Chapter~29.5]{Davidson1994}.
	Observe that the real-valued random field
	\[
	  Z_{\bm i}( {\bm \lambda} ) =  \sum_{| \bm j | \le \bm m}  \sum_{\nu,\mu=1}^p \lambda_{\bm j}^{(\nu,\mu)} Y_{\bm i}( \bm j )_{\nu \mu}, \qquad \bm i \in \Z^q,
	\] 
	is strictly stationary and $ \varphi$-mixing with mixing coefficients $ \varphi_{\bm \lambda}( r ) $ satisfying $  \varphi_{\bm \lambda}( r )  \le \varphi(r-m^\star) $, which implies
	 $ \sum_{r=1}^\infty r^{q-1} \varphi^{1/2}_{\bm \lambda}(r) < \infty $. Its asymptotic variance is positive, since the quadratic form induced by (\ref{CovFunctionB}) is positive definite. Hence we may apply \cite{Deo} and can conclude that the CLT holds true with asymptotic variance
	 \[
       \eta^2_{\bm \lambda} =  \sum_{\bm i \in \Z^q} \mathbb E( Z_{\bm 0}(  {\bm \lambda} ) Z_{\bm i}(  {\bm \lambda} ) ),
	 \]
	 where
	 \begin{align*}
	   \mathbb E( Z_{\bm 0}(  {\bm \lambda} ) Z_{\bm i}(  {\bm \lambda} ) )
	   &=
	   \sum_{| \bm j | \le \bm m} \sum_{| \bm k | \le \bm m} \sum_{\nu,\mu=1}^p \sum_{\nu',\mu'=1}^p
	   \lambda_{\bm j}^{(\nu,\mu)} \lambda_{\bm k}^{(\nu',\mu')}
	   \mathbb E( Y_{\bm 0}( \bm j )_{\nu \mu} Y_{\bm i}( \bm k )_{\nu' \mu'}). \\
	\end{align*}
	Hence,
	\begin{align*}
	   \sum_{ {\bm i} \in \Z^q} \mathbb E( Z_{\bm 0}(  {\bm \lambda} ) Z_{\bm i}(  {\bm \lambda} ) )
	   &=
	   \sum_{| \bm j | \le \bm m} \sum_{| \bm k | \le \bm m} \sum_{\nu,\mu=1}^p \sum_{\nu',\mu'=1}^p
	   \lambda_{\bm j}^{(\nu,\mu)} \lambda_{\bm k}^{(\nu',\mu')}
	   \sum_{ {\bm i} \in \Z^q}  \mathbb E( Y_{\bm 0}( \bm j )_{\nu \mu} Y_{\bm i}( \bm k )_{\nu' \mu'}).
	\end{align*}
	Substituting
	\[
	 	\mathbb E ( B_{\bm j}^{(\nu \mu)} B_{\bm k}^{(\nu' \mu')} ) = \sum_{ {\bm i} \in \Z^q}  \mathbb E( Y_{\bm 0}( \bm j )_{\nu \mu} Y_{\bm i}( \bm k )_{\nu' \mu'}),
	\]
	cf. (\ref{CovFunctionB}), it follows that 
     \[
       \eta_{\bm \lambda}^2 = \Var \left(  \sum_{| \bm j | \le \bm m}  \sum_{\nu, \mu = 1}^p \lambda_{\bm j}^{(\nu,\mu)}   B_{\bm j}^{(\nu \mu)} \right)
     \] 
     such that (\ref{ShowWeakConv}) is shown. 
 	Now the first assertion, i.e. the case of general weights $ w_{\bm m}( \bm i ) $, follows by an application of the continuous mapping theorem. The second assertion follows by noting that
	\[
	\sqrt{ | \bm n |} ( \wh{\sigma}_{\bm n}^2 - \mathbb E(  \wh{\sigma}_{\bm n}^2) ) 
	= \sum_{|\bm j| \le \bm m} w_{\bm m}( \bm j)  \sqrt{ \frac{ | \bm n| }{ \wt{\Gamma}_{\bm n}( \bm j ) } } \sqrt{ | \widetilde{\Gamma}_{\bm n}( \bm j ) | } ( \wh{\gamma}_{\bm n}( \bm j ) - \gamma_{\bm n}( \bm j) ) 
	\]
	and the fact that $ | \bm n | / | \wt{\Gamma}_{\bm n}( \bm j ) | \to 1 $, as $ \bm n \to \infty $. This completes the proof.
\end{proof}

Lastly, we prove the results about subsampling.

\begin{proof}[Proof of Theorem~\ref{Subsampling}]
	We show the result for $ \widehat{\theta}_{\bm n} = \widehat{\theta}_{\bm n}( Y_{\bm i}( \bm j) :  \bm 0 < \bm i \le \bm n ) $, $ \bm j $ fixed, since then the corresponding results for $ \widehat{\theta}_{\bm n} = \widehat{\theta}_{\bm n}( Y_{\bm i}( \bm j), |\bm j | \le \bm m :  \bm 0 < \bm i \le \bm n ) $ and $ \widehat{\theta}_{\bm n} = \widehat{\theta}_{\bm n}( \xi_{\bm i} :  \bm 0 < \bm i \le \bm n ) $ follow along the same lines. In \cite{PolitisRomanoWolf} the authors use the strong mixing coefficient defined by
	\[
	\widehat{\alpha}_{Y(\bm j)}(k; l_1 ) = \sup_{\bm E_2 = \bm E_1 + \bm t \atop \bm E_1 \subseteq \wt{\Gamma}_{\bm n}( \bm j ),  \bm t \in \wt{\Gamma}_{\bm n}( \bm j ) }
	\left\{ | \mathbb P(A_1 \cap A_2 ) - \mathbb P(A_1) \mathbb P(A_2) | : {A_1 \in \mathcal{E}_1, A_2 \in \mathcal{E}_2, \atop
		| \bm E_1 | \le l_1, d_\infty(\bm E_1, \bm E_2 ) \ge r } \right\}.
	\]
	Here $ \mathcal{E}_i = \sigma( Y_{\bm i}( \bm j) : \bm i \in \bm E_i ) $, $ i = 1,2 $, and
	the distance between two (index) sets $ \bm A,  \bm B \subset \Z^q $ is defined as 
	\begin{equation}
	\label{DefSupDistance} 
	d_\infty(\bm A ,\bm B ) = \inf \{ d_\infty(\bm a,\bm b ) : \bm a \in \bm A, \bm b \in \bm B \},
	\end{equation}
	using the sup distance $ d_\infty( \bm a, \bm b ) = \max_{1 \le j \le q} | a_j -  b_j | $ between two points $ \bm a = (a_1, \dots, a_q) $ and $ \bm b = (b_1, \dots, b_q) $. 
	
	If $ \bm E_1 $ and $ \bm E_2 $ have sup distance $ \ge r $, then there exists some coordinate $ k \in \{1, \dots, q \} $, such that one set is an element of $ \mathcal{A}^-(k;l) $ and the other set is an element of $ \mathcal{A}^+(k;l+r) $, where the $ \sigma $-fields $ \mathcal{A}^{\pm} $ are defined in (\ref{DefA-}) and (\ref{DefA+}). By shifting both sets, we can assume that $ l = 0 $. Denote the shifted sets by $ \bm E_1' $ and $ \bm E_2' $ and the associated $ \sigma $-fields by $ \mathcal{E}_1' $ and $ \mathcal{E}_2' $. If $ A \in \mathcal{E}_1 $ and $ B \in \mathcal{E}_2 $, then there is some Borel function $h$ such that
	\[
	1_{A} 1_{B} = h( Y_{\bm i}( \bm j) : \bm i \in \bm E_1, \bm i \in \bm E_2 )
	\stackrel{d}{=} h( Y_{\bm i}( \bm j) : \bm i \in \bm E_1', \bm i \in \bm E_2' ).
	\]
	Hence 
	\[
	\sup_{A \in \mathcal{E}_1, B \in \mathcal{E}_2 } | \mathbb P(B |A) - \mathbb P(B) |
	= 
	\sup_{A \in \mathcal{E}_1', B \in \mathcal{E}_2' } | \mathbb P(B|A) - \mathbb P(B) |,
	\]
	and we may conclude that
	\[
	\varphi( \mathcal{E}_1, \mathcal{E}_2 ) = \sup_{A \in \mathcal{E}_1, B \in \mathcal{E}_2} | \mathbb P(B|A) - \mathbb P(B) | \le \varphi(k, r - 2j_\star ) \le \varphi( r-2 j_\star ).
	\]
	Using the inquality $ \alpha( \mathcal{E}_1, \mathcal{E}_2 ) \le (1/2) \varphi(  \mathcal{E}_1, \mathcal{E}_2 ) $, see e.g. \cite[Theorem~25.16]{Bradley}, we therefore obtain the estimate
	\begin{equation}
	\label{UpperBoundAlphaHat}
	\widehat{\alpha}_{Y(\bm j)}( \cdot; | \bm b| ) \le (1/2) \varphi( \cdot -2 j_\star ).
	\end{equation}
	Put $ N_j = n_j - b_j + 1 $, $ j = 1, \dots, q $, and $ \bm N = (N_1, \dots, N_q) $. In \cite{PolitisRomanoWolf} it is shown that $ L_{\bm n,\bm b}(x) $ converges in probability to $ J(x,P) $, if (\ref{Subs1}) holds and  the mixing condition
	\[
	\frac{1}{| \bm N|} \sum_{k=1}^{N^\star}  k^{q-1} \widehat{\alpha}_\xi(k;|\bm b|) \to 0,
	\]
	as $ \bm n \to \infty $, is satisfied. By virtue of this estimate, the assertion can be shown as follows: Denote by $ \wt{L}_{\bm n, \bm b} $ the quantity $ L_{\bm n, \bm b} $ with the centering $ \wh{\theta}_{\bm n} $  replaced by $ \theta( \mathbb{P} ) $. Since $ \mathbb E( \wt{L}_{\bm n, \bm b}(x) ) = J_{\bm b}(x, \mathbb P) \stackrel{\mathbb P}{\to} J(x,\mathbb P) $, as $ \bm b \to \infty $, it suffices to show that $ \Var( \wt{L}_{\bm n, \bm b}(x) ) = o(1) $, as $ \bm n \to \infty $. Letting
	\[
	I_N = \{ \bm i \in \wt{\Gamma}_{\bm n}(\bm j) : | \bm i | \le \bm N \}
	\qquad \text{and} \qquad
	I^* = \{ \bm i \in  \wt{\Gamma}_{\bm n}(\bm j) : | \bm i | \le \bm b  \}
	\]
	and decomposing the variance as  $
	\Var( \wt{L}_{\bm n, \bm b}(x) ) = A^* + A
	$
	with
	\[
	A^* = \frac{1}{| \bm N |} \sum_{\bm i \in I^*} \prod_{j=1}^q \frac{N_j-|i_j|}{N_j} c( \bm i),
	\qquad
	A = \frac{1}{| \bm N |} \sum_{\bm i \in I_{\bm N} \backslash I^*} \prod_{j=1}^q \frac{N_j-|i_j|}{N_j} c( \bm i),
	\]
	where $ c( \bm i ) = \Cov( 1( |\bm b |^{-1/2}( \wh{\theta}_{{\bm n}, \bm b, \bm 0} - \theta(\mathbb P) ) \le x),  1( |\bm b |^{-1/2}( \wh{\theta}_{{\bm n}, \bm b, \bm i} - \theta(\mathbb P) ) \le x )  ) $,
	\cite{PolitisRomanoWolf} establish the estimates
	\[
	| A^* | = O\left( \prod_{j=1}^q b_j / (n_j-b_j) \right) = o(1),
	\]
	as $ \bm n \to \infty $, and
	\[
	| A | \le c_2 \frac{1}{| \bm N|} \sum_{k= \trunc{ b^\star / h_\star } + 1}^{N^\star} k^{q-1} \widehat{\alpha}_{Y(\bm j)}( k h_\star - b^\star; | \bm b | ),
	\]
	for some constant $ c_2 $, under the condition (\ref{Subs1}). The proof of the first assertion can now be completed by estimating the last expression using (\ref{UpperBoundAlphaHat}). 
	
	Recall that a sequence of distribution functions  converges in all continuity points of the limit function, if and only if the associated sequence of quantile functions converges in all continuity points of the limiting quantile function, see \cite[Lemma~21.2]{vanderVaart}. Consequently, since by continuity of $ J( x, \mathbb P ) $, $x \in \R $, 
	\[
	\sup_{x \in \R} | J_{\bm n}(x, \mathbb P) - J( x, \mathbb P ) | \stackrel{\mathbb P}{\to} 0,
	\]
	as $ \bm n \to \infty $, we obtain
	\[
	\sup_{\gamma \in (0,1)} | q_{\bm n, \bm b}(\gamma) - q(\gamma) | \stackrel{\mathbb P}{\to} 0,
	\]
	as $ \bm n \to \infty $.  Therefore,
	\[
	|\bm n|^{1/2} ( \wh{\theta}_{\bm n} - \theta(\mathbb P) ) - q_{\bm n, \bm b}(\gamma) \Rightarrow \mathcal{R} - q(\gamma),
	\]
	as $ \bm n \to \infty $, where $ \mathcal{R} \sim J(x, \mathbb P )  $, which implies
	\[
	\mathbb P( |\bm n|^{1/2} ( \wh{\theta}_{\bm n} - \theta(\mathbb P) )  \le  q_{\bm n, \bm b}(\gamma) ) \to \mathbb P( \mathcal{R} \le q( \gamma ) ) = \gamma, 
	\]
	as $ \bm n \to \infty $, which proofs the remaining assertions.
\end{proof}

\subsection{Proofs of Section~\ref{ImpEst}}

Let us now establish the consistency of the thresholding estimators of the asymptotic variance.

\begin{proof}[Proof of Theorem \ref{consistencysigmath} and Theorem~\ref{consistencysigmath_mult}] Let us first consider the univariate case $ p = 1 $.
	We show that $|\widehat{\sigma}_{\bm{n}}^2-\widehat{\sigma}_{\bm{n},th}^2|\to 0$ in probability and then the assertion follows with Theorem~\ref{consistencyvariance}. First, we have
	\begin{align*}
	|\widehat{\sigma}_{\bm{n}}^2-\widehat{\sigma}_{\bm{n},th}^2|
	&\leq\left|\sum_{|\bm{j}|\leq \bm{m}}w_{\bm{m}}(\bm{j})\left(\widehat{\gamma}_{\bm{n}}(\bm{j})-\gamma(\bm{j})\right)\right|
	+\left|\sum_{|\bm{j}|\leq \bm{m}}w_{\bm{m}}(\bm{j})\left(\gamma(\bm{j})-\gamma(\bm{j})g\left(\widehat{\gamma}_{\bm{n}}(\bm{j}),c_{\bm{n}}(\bm{j})\right)\right)\right|\\&
	+\left|\sum_{|\bm{j}|\leq \bm{m}}w_{\bm{m}}(\bm{j})\left(\gamma(\bm{j})-\widehat{\gamma}_{\bm{n}}(\bm{j})\right)g\left(\widehat{\gamma}_{\bm{n}}(\bm{j}),c_{\bm{n}}(\bm{j})\right)\right|
	\eqqcolon I_1+I_2+I_3.
	\end{align*}
	Analogously to the proof of Theorem~\ref{consistencyvariance} we directly obtain $I_1\stackrel{\mathbb P}\to 0$ and $I_3\stackrel{\ mathbb P}\to 0$ for $\bm{n}\to\infty$, as $g$ is bounded, so we only have to consider $I_2$. With the Markov inequality we obtain
	\begin{align*}
	\quad \mathbb P\left(|I_2|>\varepsilon\right)&
	=\mathbb P(|\sum_{|\bm{j}|\leq \bm{m}}w_{\bm{m}}(\bm{j})\gamma(\bm{j})\left(1-g\left(\widehat{\gamma}_{\bm{n}}(\bm{j}),c_{\bm{n}}(\bm{j})\right)\right)|>\varepsilon)\\&
	\leq
	\frac{1}{\varepsilon}
	\sum_{|\bm{j}|\leq \bm{m}}|w_{\bm{m}}(\bm{j})||\gamma(\bm{j})|\mathbb E\left(1-g\left(\widehat{\gamma}_{\bm{n}}(\bm{j}),c_{\bm{n}}(\bm{j})\right)\right)\\&
	\leq
	\frac{C_w}{\varepsilon}\sum_{\bm{j}\in\mathbb Z^q}|\gamma(\bm{j})|\mathbb E\left(1-g\left(\widehat{\gamma}_{\bm{n}}(\bm{j}),c_{\bm{n}}(\bm{j})\right)\right)\mathds{1}_{\left\{|\bm{j}|\leq\bm{m}\right\}}
	=\frac{C_w}{\varepsilon}\sum_{\bm{j}\in\mathbb Z^q} f_{\bm{n}}(\bm{j})
	\end{align*}
	with 
	\[
	f_{\bm{n}}(\bm{j})=|\gamma(\bm{j})|\mathbb E\left(1-g\left(\widehat{\gamma}_{\bm{n}}(\bm{j}),c_{\bm{n}}(\bm{j})\right)\right)\mathds{1}_{\left\{|\bm{j}|\leq\bm{m}\right\}}.
	\]
	We now want to apply the dominated convergence theorem. As $\left|f_{\bm{n}}(\bm{j})\right|\leq\left|\gamma(\bm{j})\right|$ with $\sum_{\bm{j}\in\mathbb Z^q}\left|\gamma(\bm{j})\right|<\infty$ by Lemma~1~(a) in \cite{Deo} we already have a convergent majorant and since 
	\[
	\mathbb E\left(g\left(\widehat{\gamma}_{\bm{n}}(\bm{j}),c_{\bm{n}}(\bm{j})\right)\right)\to 1,
	\]
	as $\bm{n}\to\infty$, by Assumption~3 the assertion follows. For the multivariate case denote the $ (\nu,\mu)$th element of $ \widehat{\sigma}_{\bm{n}}^2 $ by $ (\widehat{\sigma}_{\bm{n}}^2)_{\nu\mu}$,  $ ( \widehat{\sigma}_{\bm{n},th}^2 )_{\nu\mu} $ denotes the corresponding entry of  $ \widehat{\sigma}_{\bm{n},th}^2 $, $ 1 \le \nu, \mu \le p $. Now observing that
	\[
	  \| \widehat{\sigma}_{\bm{n}}^2-\widehat{\sigma}_{\bm{n},th}^2 \|_\infty \le \sum_{ \nu, \mu = 1}^p | (\widehat{\sigma}_{\bm{n}}^2 )_{\nu\mu} - (\widehat{\sigma}_{\bm{n},th}^2 )_{\nu\mu} |	  
	\]
	the consistency follows from the above estimates elaborated for the univariate case.
\end{proof}
\begin{proof}[Proof of Corollary \ref{consistencysigmac}]
	We only have to show that the special choice of $g$ in Corollary~\ref{consistencysigmac} fulfills Assumption~3 and then the assertion directly follows by Theorem~\ref{consistencysigmath}.
	This means that we have to show that 
	\[
	\mathbb E\left(1-g\left(\widehat{\gamma}_{\bm{n}}(\bm{j}),c_{\bm{n}}(\bm{j})\right)\right)=P\left(\left|\widehat{\gamma}_{\bm{n}}(\bm{j})\right|\leq c_{\bm{n}}(\bm{j})\right)\to 0
	\]
	for all $\bm{j}\in\mathbb Z^q$. 
	For that observe that by Lemma~5.3 in \cite{PrauseSteland} and the assumptions of the corollary we have
	\[
	|\widehat{\gamma}_{\bm{n}}(\bm{j})|-c_{\bm{n}}(\bm{j})\to|\gamma(\bm{j})|-c(\bm{j})\eqqcolon X
	\]
	in probability, as $\bm{n}\to\infty$. This implies the convergence in distribution, i.e.\ if $F_{\left|\widehat{\gamma}_{\bm{n}}(\bm{j})\right|- c_{\bm{n}}(\bm{j})}$ denotes the distribution function of $\widehat{\gamma}_{\bm{n}}(\bm{j})-c_{\bm{n}}(\bm{j})$ and $F_X$ the one of $X$, we have
	\[
	\mathbb P\left(\left|\widehat{\gamma}_{\bm{n}}(\bm{j})\right|\leq c_{\bm{n}}(\bm{j})\right)
	=F_{\left|\widehat{\gamma}_{\bm{n}}(\bm{j})\right|- c_{\bm{n}}(\bm{j})}(0)
	\to F_X(0),
	\]
	as $\bm{n}\to\infty$, if zero is a continuity point of $F_X$. But since
	\begin{align*}
	F_X(x)=\begin{cases}
	0,  & x<|\gamma(\bm{j})|-c(\bm{j})\\
	1, & x\geq|\gamma(\bm{j})|-c(\bm{j})
	\end{cases}
	\end{align*}
	and $c(\bm{j})<|\gamma(\bm{j})|$ by assumption this is fulfilled. Moreover, the last inequality also implies that $F_X(0)=0$. Thus we have $\mathbb P\left(\left|\widehat{\gamma}_{\bm{n}}(\bm{j})\right|\leq c_{\bm{n}}(\bm{j})\right)\to 0$ for all $\bm{j}\in\mathbb Z^q$ as $\bm{n}\to\infty$ which completes the proof.
\end{proof}

\section*{Acknowledgments} The authors thank an anonymous associate editor and an anonymous referee for their helpful comments. Part of the work of the second author has been supported by a grant from Deutsche Forschungsgemeinschaft (DFG), grant No. STE 1034/11-1, which he gratefully acknowledges.

\end{document}